\documentclass{article}
\usepackage[utf8]{inputenc}
\usepackage{commath}

\usepackage{authblk}

\usepackage{mathrsfs, amssymb}
\usepackage{amsthm}
\usepackage[intlimits]{empheq}
\usepackage{bbm} 
\usepackage{amsmath}
\usepackage{xcolor}
\usepackage[scientific-notation=true]{siunitx}
\usepackage{float}
\usepackage{dsfont}	
\usepackage[numbers]{natbib}
\usepackage[nottoc,numbib]{tocbibind}
\usepackage{bbm}
\usepackage{subcaption}
\usepackage[inner=2cm,outer=2cm,top=2cm,bottom=2cm]{geometry}
\usepackage[ruled,vlined]{algorithm2e}
\usepackage{enumitem}
\usepackage[colorlinks=true,linkcolor=black]{hyperref}

\usepackage[export]{adjustbox}
\usepackage{caption}
\DeclareMathOperator*{\argmin}{arg\,min}

\DeclareMathOperator*{\R}{\mathbb{R}}

\DeclareMathOperator*{\sym}{\operatorname{sym}}  
\DeclareMathOperator*{\skewp}{\operatorname{skew}}

\DeclareMathOperator*{\Cov}{{Cov}}

\theoremstyle{plain}
\newtheorem{theorem}{Theorem}[section]
\newtheorem{lemma}[theorem]{Lemma}
\newtheorem{corollary}[theorem]{Corollary}
\newtheorem{proposition}[theorem]{Proposition}
\newtheorem{assumption}{Assumption}

\newtheorem*{theorem*}{Theorem}
\newtheorem*{proposition*}{Proposition}

\theoremstyle{definition}
    \newtheorem{definition}[theorem]{Definition}

    \theoremstyle{remark}
    \newtheorem{remark}[theorem]{Remark}

    \usepackage[T1]{fontenc}
    \usepackage{lipsum}
    \usepackage{tocloft}
\newcommand{\N}{\mathbb{N}}

\newcommand{\Y}{\mathbf{Y}}

\newcommand{\YY}{\mathbb{Y}}

\makeatletter
\@tfor\next:=abcdefghijklmnopqrstuvwxyzABCDEFGHIJKLMNOPQRSTUVWXYZ\do{%
\def\command@factory#1{%
\expandafter\def\csname b#1\endcsname{\mathbf{#1}}
\expandafter\def\csname cl#1\endcsname{\mathcal{#1}}
\expandafter\def\csname bb#1\endcsname{\mathbb{#1}}
}
\expandafter\command@factory\next
}

\renewcommand{\textcolor}[2]{#2}

\title{Rough McKean-Vlasov dynamics for robust ensemble Kalman filtering}
\date{\today}

\author[1]{Michele Coghi}
\author[2]{Torstein Nilssen}
\author[3]{Nikolas N{\"u}sken}
\author[4]{Sebastian Reich}
\date{\today}

\affil[1]{Technische Universi\"at Berlin, Straße des 17. Juni 136, 10587 Berlin, Germany,
\href{mailto:coghi@tu-berlin.de}{coghi@math.tu-berlin.de}}
\affil[2]{
Institute of Mathematics, University of Agder, 4604 Kristiansand, Norway, \href{mailto:torstein.nilssen@uia.no}{torstein.nilssen@uia.no}}
\affil[3]{Institute of Mathematics, Universit\"at Potsdam, 14476 Potsdam, Germany, \href{mailto:nuesken@uni-potsdam.de}{nuesken@uni-potsdam.de}}
\affil[4]{Institute of Mathematics, Universit\"at Potsdam, 14476 Potsdam, Germany, \href{mailto:nuesken@uni-potsdam.de}{sebastian.reich@uni-potsdam.de}}

\begin{document}

\maketitle

\abstract{

Motivated by the challenge of incorporating data into misspecified and multiscale dynamical models, we study a McKean-Vlasov equation that contains the data stream as a common driving rough path. This setting allows us to prove well-posedness as well as continuity with respect to the driver in an appropriate rough-path topology. The latter property is key in our subsequent development of a robust data assimilation methodology: We establish propagation of chaos for the associated interacting particle system, which in turn is suggestive of a numerical scheme that can be viewed as an extension of the ensemble Kalman filter to a rough-path framework. Finally, we discuss a data-driven method based on subsampling to construct suitable rough path lifts and demonstrate the robustness of our scheme in a number of numerical experiments related to parameter estimation problems in multiscale contexts.
\\\\

}
 \setcounter{tocdepth}{1}
\tableofcontents

\section{Introduction}

Combining mathematical descriptions of reality with observational data is a key task in economics, science and engineering. In typical applications (such as meteorology \cite{klein2010scale} or molecular dynamics \cite{lelievre2016partial}) there is a hierarchy of models available, ranging from the highly accurate but conceptually and computationally demanding to the approximate but readily interpretable and scalable. Naturally, real-world data is almost always complex, multiscale, and increasingly high-dimensional, whereas the corresponding mathematical abstractions are often preferred to be simple and of comparatively low resolution. The discrepancy between intricate data and reduced-order models poses a significant challenge for their simultaneous treatment, and the failure of some standard statistical approaches in scenarios of this type is well documented \cite{ait2005often,olhede2010frequency,pavliotis2007parameter,ying2019bayesian,zhang2005tale}.
\\

\emph{Robustness to model misspecification and perturbation of the data} is a central concept in the design of statistical methodology capable of bridging scales: Consider a statistical model $\mathcal{M}_0$ -- to be thought off as simple \nolinebreak-- generating (time-dependent) data $(Y^0_t)_{t \ge 0}$ and a corresponding algorithmic procedure $\Phi$ producing the output  $\Phi((Y^0_t)_{t \ge 0})$. We expect that $\Phi$ deals adequately with complex data in the case when it is continuous in an appropriate sense. Indeed, $\mathcal{M}_0$ might be a simplified description of an underlying model family $(\mathcal{M}_\varepsilon)_{\varepsilon \ge 0}$, where the formal limit $\lim_{\varepsilon \rightarrow 0} \mathcal{M}_\varepsilon = \mathcal{M}_0$ encapsulates the passage from a complex to a reduced description. The output $\Phi((Y^\varepsilon_t)_{t \ge 0})$ on `real-world' data $(Y^\varepsilon_t)_{t \ge 0}$ is then close to $\Phi((Y^0_t)_{t \ge 0})$ by continuity even though $\Phi$ has been contructed on the basis of $\mathcal{M}_0$.
\\

Unfortunately, in various contexts (for instance, in parameter estimation for diffusions \cite{diehl2016pathwise} and stochastic filtering \cite{bain2008fundamentals})  the map $\Phi$ is given in terms of stochastic integrals against $(Y^0_t)_{t \ge 0}$ which are well known to be discontinuous with respect to standard topologies \cite{friz2020course,friz2010multidimensional}. The theory of \emph{rough paths} provides a principled route towards constructing continuous modifications $\Phi'$ of $\Phi$ (employed, for instance, in \cite{crisan2013robust,diehl2016pathwise}) by replacing stochastic integrals in terms of rough integrals defined on appropriately lifted paths $(Y,\mathbb{Y})$. Although the difficulty of obtaining or imposing the additional information $\mathbb{Y}$ is well known (see, however, \cite{bailleul2015inverse,flint2016discretely,levin2013learning} and Section \ref{sec: numerical examples} of this article), recent works have demonstrated the potential of including path signatures into data-driven methods to compress information efficiently or unveil multiscale structure \cite{chevyrev2016primer,lyons2014rough}. \textcolor{red}{Moreover, the rough path perspective provides refined insights into misspecification of diffusion models: Indeed, two paths may be `similar in a classical sense' (for instance, in supremum norm) in spite of their associated lifts exhibiting a large rough path distance, in turn leading to inaccurate inferences. In Section \ref{sec: numerical examples} we discuss how those discrepancies can be compensated in a systematic and data-driven manner.}
\\

In this paper, we follow the rough path paradigm just described and develop a robust version of the \emph{Ensemble Kalman Filter} (EnKF) \cite{bishop2020mathematical,evensen2009data}, drawing on a reformulation of the Kushner-Stratonovich SPDE for stochastic filtering \cite{bain2008fundamentals} in terms of McKean-Vlasov dynamics (known as the \emph{feedback particle filter} \cite{pathiraja2020mckean,taghvaei2018kalman,yang2013feedback}). The EnKF is a versatile  approximate procedure for Bayesian inference that is observed to perform particularly well in high-dimensional settings \cite{reich2015probabilistic} and has recently been applied to problems in machine learning \cite{gottwald2020supervised,pidstrigach2021affine}. Our analysis includes the case when the model and observation noises are correlated as this is precisely the setting in which significant obstacles in the construction of robust filters are known to occur \cite{crisan2013robust}. Furthermore, as shown in \cite{nusken2019state} and \textcolor{red}{recalled in} Appendix \ref{app:MLE}, stochastic filtering with correlated noises is a natural generalisation of maximum likelihood parameter estimation for diffusions in the context of inexact measurements. In this setting, we discuss the relationship between subsampling-based approaches to multiscale parameter estimation \cite{pavliotis2007parameter} to the task of estimating the L{\'e}vy area in a rough paths approach.
\\

The construction put forward in this paper requires the formulation of McKean-Vlasov equations in a rough path setting (as studied first in \cite{cass2015evolving} and more recently in the twin papers \cite{bailleul2019propagation,bailleul2020solving}, both considering the dynamics to be driven by a random rough path). 
The approach introduced in \cite{coghi2019rough} is more suitable for our needs as it allows a clear separation between independent Brownian motions and a common deterministic noise. However, the assumptions in \cite{coghi2019rough} do not cover unbounded coefficients, while the common noise coefficient may depend on both the state of the solution and its law. In equation \eqref{eq: main system} considered below, the coefficient $P$ only depends on the law of $\widehat{X}$, which simplifies the problem to some extent, but also allows us to use a different approach, where we treat the stochastic and the rough integrals in two separated steps. As it will become clear from the proofs in Section \ref{sec: rough McKean-Vlasov}, there is no need to create a joint rough path (or rough driver).
All previous works on rough McKean-Vlasov equations deal with bounded globally Lipschitz-continuous (actually smoother) coefficients. These results cannot be applied here, as the coefficient $P$ has linear growth in the measure of the solution and is locally-Lipschitz with Lipschitz constant depending on the moments of the solution.

\subsection{Setting and main results}

In this section we specify the exact setting and present our main results.
Throughout the paper we fix a filtered probability space $(\Omega, \mathcal{F},(\mathcal{F}_t)_{t\ge 0}, \mathbb{P})$ satisfying the usual conditions
and consider the following filtering model,
\begin{subequations}
	\label{eq:signal obs}
	\begin{align}
	\label{eq:signal}
	\mathrm{d}X_t & = f(X_t) \, \mathrm{d}t + G^{1/2} \, \mathrm{d}W_t, \quad && X_0 \sim \pi_0
	\\
	\label{eq:obs}
	\mathrm{d}Y_t & = h(X_t) \, \mathrm{d}t + U \, \mathrm{d}W_t + R^{1/2} \, \mathrm{d}V_t, \quad && Y_0 = 0.
	\end{align}
\end{subequations}
In the above display, \eqref{eq:signal} represents the (hidden) signal, whereas \eqref{eq:obs} specifies the available observations. Note straight away that the signal noise  $G^{1/2} \, \mathrm{d}W_t$ and the observation noise $U \, \mathrm{d}W_t + R^{1/2} \, \mathrm{d}V_t$ may be correlated. We assume that the signal is $D$-dimensional, that is, $X_t \in \mathbb{R}^D$, and that the observations are $d$-dimensional, $Y_t \in \mathbb{R}^d$. In applications it is often the case that $D \gg d$. The maps $f:\mathbb{R}^D \rightarrow \mathbb{R}^D$ and $h: \mathbb{R}^D \rightarrow \mathbb{R}^d$ are assumed to be sufficiently regular (see below). We assume  $G \in \mathbb{R}^{D \times D}$, $R \in \mathbb{R}^{d \times d}$ to be symmetric nonnegative definite and $U \in \mathbb{R}^{d \times D}$.
Furthermore, $(W_t)_{t \ge 0}$ and $(V_t)_{t \ge 0}$ denote independent $D$- and $d$- dimensional standard Brownian motions, respectively. Lastly, we assume that the observation covariance
\begin{equation}
C = UU^T + R \in \mathbb{R}^{d \times d}
\end{equation}
is strictly positive definite.
Defining the observation $\sigma$-algebras 
\begin{equation}
\mathcal{Y}_t = \sigma \left\{(Y_s)_{0 \le s \le t} \right\} \vee \mathcal{N}, \qquad t \ge 0,
\end{equation}
where $\mathcal{N} \subset \mathcal{F}$ is the collection of $\mathbb{P}$-null sets,
our objective is to compute or approximate the filtering measures
\begin{equation}
\label{eq:filtering measures}
\pi_t[\phi] = \mathbb{E} \left[ \phi(X_t) \vert \mathcal{Y}_t \right],
\end{equation}
for bounded and measurable test functions $\phi$; we refer to \cite{bain2008fundamentals} for technical details.


As a step towards tractable numerical approximations of \eqref{eq:filtering measures}, we \textcolor{red}{follow \cite{nusken2019state} and} introduce the McKean-Vlasov equation
\begin{equation}
\label{eq: true McKean Vlasov}
\left\{ \begin{array}{rl}
\mathrm{d}\widehat{X}_t & = f(\widehat{X}_t) \, \mathrm{d}t + G^{1/2} \, \mathrm{d}\widehat{W}_t + K_t(\widehat{X}_t)C^{-1}\circ \mathrm{d}I_t + \Xi_t(\widehat{X}_t) \, \mathrm{d}t\\
\mathrm{d}I_t & = \mathrm{d}Y_t - \left( h(\widehat{X}_t)\, \mathrm{d}t + U \, \mathrm{d}\widehat{W}_t +  R^{1/2} \, \mathrm{d}\widehat{V}_t\right),
\end{array}
\right.
\end{equation}
where $(\widehat{W}_t)_{t \ge 0}$ (resp. $(\widehat{V}_t)_{t \ge 0}$) is a given $D$-dimensional (resp. $d$-dimensional) Brownian motion, both independent from $(W_t)_{t \ge 0}$ and $(V_t)_{t \ge 0}$.  Denoting by $\widehat{\pi} := \mathcal{L}(\widehat{X}\mid \mathcal{Y})$ the conditional law with respect to the common noise $Y$ of solutions $\widehat{X}$ to \eqref{eq: true McKean Vlasov},
the coefficients $K_t: \R^D \to \R^{D\times d}$ and $\Xi_t: \R^D \to \R^{D}$ are required to solve the following partial differential equations,
\begin{equation}
\label{eq:Poisson K intro}
\nabla \cdot \left(\widehat{\pi}_t \left(K_t - BC \right)) \right) = -\widehat{\pi}_t \left( h - \widehat{\pi}_t[h] \right),
 \end{equation}
and
 \begin{equation}
 \label{eq:Poisson Gamma intro}
\nabla \cdot \left( \widehat{\pi}_t \Xi_t \right) = \frac{1}{2}\widehat{\pi}_t \left( \operatorname{Trace} (K_t C^{-1} (\nabla h)^T) - \widehat{\pi}_t\left[\operatorname{Trace} (K_tC^{-1} (\nabla h)^T) \right] \right),
\end{equation}
for $t \ge 0$ and $\mathbb{P}$-almost surely, where $B = G^{1/2} U^T C^{-1}$. Moreover, the common noise $Y$ coincides with the observation process \eqref{eq:obs} and the integral in $\circ \mathrm{d}Y$ is understood in the sense of Stratonovich.

The construction of the system \eqref{eq: true McKean Vlasov}-\eqref{eq:Poisson Gamma intro} as well as existence of solutions will be detailed in Section \ref{sec:McKean formulation}. 
The filtering problem and the McKean-Vlasov equation \eqref{eq: true McKean Vlasov} are related in the following way.
\begin{proposition*}[formal, see Proposition \ref{prop:McKean}]
If $\pi_t$ admits a density and the solution to \eqref{eq: true McKean Vlasov} is unique, then $\pi_t = \widehat{\pi}_t$, $\mathbb{P}$-a.s. for every $t \ge 0$.
\end{proposition*}
Reformulations of the filtering problem in terms of McKean-Vlasov dynamics similar to \eqref{eq: true McKean Vlasov}-\eqref{eq:Poisson Gamma intro} have been introduced in \textcolor{red}{\cite{crisan2010approximate,yang2013feedback}}, further analysed in \cite{pathiraja2020mckean}, and are commonly referred to as \emph{feedback particle filters}; the formulation \eqref{eq: true McKean Vlasov}-\eqref{eq:Poisson Gamma intro} combines Stratonovich integration \textcolor{red}{(as in \cite{pathiraja2020mckean})} with a stochastic innovation term $\mathrm{d}I_t$ \textcolor{red}{(as in \cite[Section 4.2]{reich2019data})} so as to allow for a transition to \textcolor{red}{geometric} rough paths in the setting of correlated noises \textcolor{red}{(as in \cite{nusken2019state})}.

One of the practical challenges posed by the system of equations \eqref{eq: true McKean Vlasov}-\eqref{eq:Poisson Gamma intro} is solving the PDEs \eqref{eq:Poisson K intro} and \eqref{eq:Poisson Gamma intro} in $K$ and $\Xi$ for a given measure $\hat{\pi}$. As has been shown in \cite{taghvaei2018kalman} for a similar system of equations, replacing $K$ and $\Xi$ by their best constant-in-space approximations in least-square sense recovers a version of the Ensemble Kalman filter dynamics \cite{bishop2020mathematical}. We follow this approach (see Lemma \ref{lem: approximate K} below) and replace the coefficient $K_t C^{-1}$ by $P: \mathcal{P}(\R^D) \to \R^{D\times d}$, explicitly defined as
\begin{equation}
    \label{eq: approx P}
    P(\pi) := \operatorname{\Cov}_{\pi}(x,h)C^{-1} + B := \pi[x (h - \pi[h])^T ] C^{-1} + B,
    \qquad
    \pi \in \mathcal{P}({\R} ^ {D} ),
\end{equation}
where $\mathcal{P}(\mathbb{R}^D)$ refers to the set of probability measures on $\mathbb{R}^D$.
Similarly, we replace $\Xi$ by $\Gamma : \mathcal{P}(\R^D) \to \R^{D}$, defined as 
\begin{equation}
    \label{eq: approx Gamma}
    \Gamma^\gamma(\pi)  = -\frac{1}{2} \operatorname{Trace}\left(P(\pi) \pi \left[x^\gamma \left( \textcolor{red}{\nabla h^\top}   - \pi[\textcolor{red}{\nabla h^\top}] \right)\right]\right),
    \qquad
    1\leq \gamma \leq D,
    \;
    \pi \in \mathcal{P}({\R}^{D}),
\end{equation}
which can be interpreted as an It\^o-Stratonovich correction term.\footnote{\textcolor{red}{Indeed, the system \eqref{eq: main system stochastic version} can be recognised as the Stratonovich version of the EnKF, see Appendix \ref{app:Strato EnKF}.}}
Thus, we obtain the system
\begin{equation}
\label{eq: main system stochastic version}
\left\{ \begin{array}{rl}
\mathrm{d}\widehat{X}_t & = f(\widehat{X}_t) \, \mathrm{d}t + G^{1/2} \, \mathrm{d}\widehat{W}_t + P(
\widehat{\pi}_t)\, \circ \mathrm{d}I_t + \Gamma(\widehat{\pi}_t ) \, \mathrm{d}t\\
\mathrm{d}I_t & = \mathrm{d} Y_t - \left( h(\widehat{X}_t)\, \mathrm{d}t + U \, \mathrm{d}\widehat{W}_t +  R^{1/2} \, \mathrm{d}\widehat{V}_t\right),
\end{array}
\right.
\end{equation}
where $\widehat{\pi}_t = \mathcal{L}(\widehat{X}_t | \mathcal{Y}_t )$. This equation is well posed according to Lemma \ref{lem:wellposedness McKean common noise} below.\footnote{\textcolor{red}{In this paper, we assume $h$ to be bounded, circumventing difficulties with the non-Lipschitzness of $P$ in the case when $h$ is unbounded. For results addressing this challenge (although in slightly different settings), we refer the reader to \cite{de2018long,de2020analysis,del2016stability,ding2020ensemble,kelly2014well,lange2020derivation,lange2019continuous,lange2020mean,lange2021continuous}}.}


In order to construct a robust filter, we replace the common noise in equation \eqref{eq: true McKean Vlasov} by a deterministic rough path $\bY \in \mathscr{C}^{\alpha}([0,T],\R^d)$ with regularity $\frac{1}{3} < \alpha \leq \frac{1}{2}$. As explained in the introduction, the rationale is that the solutions to rough differential equations are \textcolor{red}{expected to be} continuous in the rough path driver (see \cite{friz2020course}), in contrast to the It\^o solution map associated to stochastic differential equations. Moreover,
using a deterministic path is natural since we are conditioning on the observation, that we can assume to be given and deterministic. 
\\ 

Applying the  modifications from the preceding two paragraphs to equation \eqref{eq: true McKean Vlasov} we obtain the system
\begin{equation}
\label{eq: main system}
\left\{ \begin{array}{rl}
\mathrm{d}\widehat{X}_t & = f(\widehat{X}_t) \, \mathrm{d}t + G^{1/2} \, \mathrm{d}\widehat{W}_t + P(
\widehat{\pi}_t)\, \mathrm{d}I_t + \Gamma(\widehat{\pi}_t ) \, \mathrm{d}t\\
\mathrm{d}I_t & = \mathrm{d}\bY_t - \left( h(\widehat{X}_t)\, \mathrm{d}t + U \, \mathrm{d}\widehat{W}_t +  R^{1/2} \, \mathrm{d}\widehat{V}_t\right).
\end{array}
\right.
\end{equation}
From Lemma \ref{lem: P controlled path} below, the path $P(\widehat{\pi}_{\cdot})$ is controlled by $\bY$ with Gubinelli derivative \cite[Definition 4.6]{friz2020course}
\begin{equation}
    	\label{eq: Gubinelli derivative of P intro}
	    P(\widehat{\pi}_{s})^{\prime}
	    = P(\widehat{\pi}_s)^{\top}
	    \widehat{\pi}_s[(x-\pi[x]) D h^{\top}] C^{-1} = 
	    P(\widehat{\pi}_s)^{\top}
	    {\Cov}_{\widehat{\pi}}(x, D h ) C^{-1},
\end{equation}
so that the rough integral in equation \eqref{eq: main system} will make sense (see Section \ref{sec: controlled rough-paths} below for an overview on controlled rough paths). Moreover, when $Y$ is a semimartingale with covariance $\operatorname{Cov}(Y_t, Y_s) = C (t-s)$, for $s\leq t$, the correction between the Stratonovich and It\^{o} rough path lift is given by $-\frac{1}{2}\operatorname{Trace}(P(\widehat{\pi}_t)^{\prime}C)\, \mathrm{d}t$, which corresponds to $\Gamma(\widehat{\pi})\, \mathrm{d}t$, \textcolor{red}{see Appendix \ref{app:Strato EnKF}.}
\\

Our main result is the following well-posedness and stability theorem for \eqref{eq: main system}.

\begin{theorem}
\label{thm: main}
Let $1/3 < \alpha < 1/2$ and $\widehat{X}_0 \in L^{\rho}(\Omega, \R^D)$ with $\rho > 2/(1-2\alpha)$. Assume $h \in C^2_b(\R^D, \R^d)$ and that $f:\R^D \to \R^D$ is bounded and Lipschitz-continuous. Then equation \eqref{eq: main system} admits a unique solution. Moreover, the map $\mathscr{C}^{\alpha}([0,T],\R^d) \ni \bY \mapsto \widehat{\pi}_t \in \mathcal{P}(C([0,T],\R^D))$ is continuous. 
Moreover, if $\bY$ is the Stratonovich lift of $Y$ in \eqref{eq: main system stochastic version}, then the solutions to \eqref{eq: main system stochastic version} and \eqref{eq: main system} coincide, $\mathbb{P}$-almost surely.
\end{theorem}
The proof is a consequence of the results in Section \ref{sec: rough McKean-Vlasov} and can be found at the end of that section.
Associated to the McKean-Vlasov equation \eqref{eq: main system} we also study the following system of mean-field interacting particles,
\begin{subequations}
\label{eq: interacting particles}
\begin{align}
\mathrm{d}\widehat{X}^i_t & = f(\widehat{X}^i_t) \, \mathrm{d}t + G^{1/2} \, \mathrm{d}\widehat{W}^i_t + P(\mu^N_t) \, \mathrm{d}I^i_t + \Gamma(\mu^N_t) \, \mathrm{d}t,     
\\
\mathrm{d}I^i_t & = \mathrm{d}\bY_t - (h(\widehat{X}^i_t)  \, \mathrm{d}t+  U \, \mathrm{d}\widehat{W}^i_t + R^{1/2} \, \mathrm{d}\widehat{V}^i_t),
\end{align}
\end{subequations}
where $\mu^N_t := \frac{1}{N}\sum_{i=1}^N \delta_{X_t^i}$ is the empirical measure of the system and $\bY \in \mathscr{C}^{\alpha}([0,T],\R^d)$ is the canonical lift of a differentiable and bounded path $Y:[0,T]\to \R^d$ with bounded cadlag derivative. We have the following well-posedness and convergence result for the interacting particles.
\begin{theorem*}[see Remark \ref{rmk: particles coefficient conversion} and Theorem \ref{thm: particles convergence}]
Under the same assumptions of Theorem \ref{thm: main}, let $(\bY^{\delta})_{\delta > 0}$ be a family of lifts of bounded differentiable paths with cadlag derivatives that approximate $\bY$ in the rough-path metric. Let $\mu^{\delta, N}$ be the empirical measure of \eqref{eq: interacting particles} driven by $\Y^{\delta}$. Then there exists a sequence $\delta_N$ such that, for every $t\in [0,T]$, $\mu^{N,\delta_N}_t \overset{N\to\infty}{\to} \pi_t$,  in $\rho$-Wasserstein distance in $L^1$.
\end{theorem*}
\textcolor{red}{For results concerning similar mean-field limits in the classical setting we refer to \cite{ding2020ensemble,ding2021ensemble,lange2020mean}.}
The preceding two theorems suggest that numerical methods based on the interacting particle system \eqref{eq: interacting particles} are robust to perturbations in the data, and hence suitable for applications in multiscale contexts as described in the introduction.
Inspired by Davie's work \cite{davie2008differential}, we propose the following recursive numerical scheme,
\begin{subequations}
\label{eq:numerical scheme}
\begin{align}
X^i_{k+1} & = X^i_{k} + f(X_k^i)\Delta t  + G^{1/2} \sqrt {\Delta t}\,\xi_k^i + \widehat{P}_k\left( \Delta Y_k - (h(X^i_k)\Delta t + U \sqrt{\Delta t} \,  \xi_k^i + R^{1/2}  \sqrt{\Delta t }\,\eta_k^i) \right)
\\
\label{eq:num scheme2}
 & +  \widehat{\text{Cov}}(x,D h)\widehat{P}_k  \Delta \mathbb{Y}_k + \widehat{\Gamma}_k \Delta t,
\end{align} 
\end{subequations}
in the following referred to as the \emph{Rough-Path Ensemble Kalman Filter} (RP-EnKF). 
Here, $\Delta t > 0$ is the step size, and $(\xi_n^i)$ and $(\eta_n^i)$ denote independent zero mean Gaussian random variables of dimensions $D$ and $d$, respectively. The precise form of the estimator versions $\widehat{P}_k$, $\widehat{\Cov}$ and $\widehat{\Gamma}_k$  will be detailed in Section \ref{sec: numerical examples}.
Finally, the first term in \eqref{eq:num scheme2} is built after the Gubinelli derivative from \eqref{eq: Gubinelli derivative of P intro}, \textcolor{red}{with its precise meaning given by
\begin{equation}
    (\widehat{\text{Cov}}_k(x,Dh) \widehat{P}_k   \Delta \mathbb{Y}_k)_{\gamma}
    = \sum_{j,q=1}^d \sum_{r=1}^D\widehat{\text{Cov}}_k(x,Dh)_{\gamma,j,r} (\widehat{P}_k)_{r,q} 
    (\Delta \mathbb{Y}_k)_{q,j}
    \qquad
    \gamma = 1,\dots, D.
\end{equation}}
Crucially, the scheme \eqref{eq:numerical scheme} takes the lifted component $\Delta \mathbb{Y}_k$ (representing iterated integrals of the path $(Y_t)_{t \ge 0}$) as an input. This dependence allows our methodology to appropriately take into account multiscale structure and other information encoded in the signature \textcolor{red}{(such as discrepancies due to model misspecification)}, but also necessitates appropriate ways to estimate $\mathbb{Y}_k$ from data. We will discuss this in more depth in Sections \ref{sec:parameter estimation} and \ref{sec: numerical examples}, but note here that it is natural to decompose $\Delta \mathbb{Y}_k$ into its symmetric and skew-symmetric part \begin{equation}
\label{eq:lift decomp intro}
\Delta \mathbb{Y}_k = \Delta \mathbb{Y}^{\sym}_k +\Delta \mathbb{Y}^{\skewp}_k.
\end{equation}
The form of the symmetric part $\Delta \mathbb{Y}^{\sym}_k$ is suggested by the requirement that the lifted path is geometric,
\begin{equation}
\label{eq:sym from path}
\Delta \mathbb{Y}^{\mathrm{sym}}_k = \frac{1}{2} (y_{k+1} - y_k) \otimes (y_{k+1} - y_k);
\end{equation}
in particular this expression can readily be computed from discrete-time observations $y_k$. The difficulty thus resides in estimating the L{\'e}vy area contributions \textcolor{red}{(or corrections)} $\Delta \mathbb{Y}^{\skewp}_k$. We suggest a subsampling-based method, establishing connections to multiscale parameter estimation as investigated in \cite{pavliotis2007parameter}; see Section \ref{sec:parameter estimation}. Other approaches towards obtaining $\Delta \mathbb{Y}_k$ have been developed in \cite{bailleul2015inverse,flint2016discretely,levin2013learning}. We would like to stress that although estimating $\Delta \mathbb{Y}^{\skewp}_k$ works reasonably well in our experiments (see Sections \ref{subsec:magnetic field} and \ref{sec:Lorenz}), in many applications it might yield satisfactory results to neglect the skew-symmetric part, that is, to use the approximation $\Delta \mathbb{Y}_k \textcolor{red}{\approx} \Delta \mathbb{Y}^{\sym}_k$, either because the observation path is one-dimensional, or because the L{\'e}vy area term is comparatively small (see Section \ref{sec:two-scale potential}). In these cases, the RP-EnKF can be implemented straightforwardly without additional estimation steps, and represents a robust Stratonovich-version of the EnKF \textcolor{red}{(see Appendix \ref{app:Strato EnKF})}. 
\\

\textbf{It{\^o}, Stratonovich and rough integrals.} Before concluding this introduction, let us comment on the difference between the RP-EnKF scheme \eqref{eq:numerical scheme} and the more conventional EnKF updates

\begin{equation}
\label{eq:EnKF}
X^i_{k+1}  = X^i_{k} + f(X_k^i)\Delta t  + G^{1/2} \sqrt {\Delta t}\,\xi_k^i + \widehat{P}_k\left( \Delta Y_k - (h(X^i_k)\Delta t + U \sqrt{\Delta t} \,  \xi_k^i + R^{1/2}  \sqrt{\Delta t }\,\eta_k^i) \right),
\end{equation} 
see \cite{bishop2020mathematical,evensen2009data,reich2015probabilistic}. Clearly, \eqref{eq:numerical scheme} and \eqref{eq:EnKF} coincide up to the terms in \eqref{eq:num scheme2}. This difference can be attributed to alternative \textcolor{red}{(but equivalent)} perspectives on the underlying continuous-time dynamics, \textcolor{red}{and hence we expect the terms in \eqref{eq:num scheme2} to cancel each other in the limit as $\Delta t \rightarrow 0$, when $\Delta Y_k$ is obtained at discrete time-points from \eqref{eq:obs} and $\Delta\mathbb{Y}_k$ are the increments of the Stratonovich lift corresponding to piecewise linear approximation (that can be computed according to \eqref{eq:sym from path} and setting $\Delta \mathbb{Y}_k^{\mathrm{skew}} = 0$).} Indeed, \eqref{eq:EnKF} corresponds to the Euler-Maruyama discretisation of the McKean-Vlasov SDE 
\begin{equation}
\label{eq:Ito McKean}
\mathrm{d}\widehat{X}_t = f(\widehat{X}_t) \, \mathrm{d}t + G^{1/2} \, \mathrm{d}W_t + P(\widehat{\pi}_t) \left( \mathrm{d}Y_t - \left( h(\widehat{X}_t) \, \mathrm{d}t + U \, \mathrm{d}\widehat{W}_t + R^{1/2} \, \mathrm{d}\widehat{V}_t \right) \right),
\end{equation}
understood in the sense of It{\^o}. In contrast, the term $\widehat{\Gamma}_k \Delta t$ in \eqref{eq:num scheme2} can be seen as an It{\^o}-Stratonovich correction to \eqref{eq:Ito McKean}, \textcolor{red}{see Appendix \ref{app:Strato EnKF}}, while the first term in \eqref{eq:num scheme2} arises from a Milstein-type approximation scheme for Stratonovich SDEs \cite[Section 10.3]{kloeden1992stochastic} (or as part of the discrete approximation of rough integrals according to Davie \cite{davie2008differential}). \textcolor{red}{The equivalence of the classical and the rough path version, that is, of \eqref{eq: main system stochastic version} and \eqref{eq: main system}, is stated in Lemma \ref{lem:wellposedness McKean common noise}.} 

\textcolor{red}{The distinction between \eqref{eq:numerical scheme} and \eqref{eq:EnKF} becomes important when the data $\Delta Y_k$ and $\Delta \mathbb{Y}_k$ is described by \eqref{eq:signal obs} only in an approximate sense and hence robustness becomes a key issue.
Importantly, the passage to rough integrals is indispensable for the continuity statements in Theorem \ref{thm: main} and the numerical robustness of the RP-EnKF scheme demonstrated in Section \ref{sec: numerical examples}. The Stratonovich picture is natural in view of the Wong-Zakai theorem \cite[Section 9.2]{friz2021rough} and determines the symmetric part in \eqref{eq:lift decomp intro} as $\bY$ can then be interpreted as a geometric rough path (see \eqref{eq:sym from path} and Section \ref{sec: numerical examples}). Our approach of replacing It{\^o} by Stratonovich and subsequently rough integrals mirrors the construction in \cite{diehl2016pathwise} in the context of maximum likelihood parameter estimation for diffusions.}
\\

\textbf{Our contributions and structure of the paper.} Our main contributions are as follows:
\begin{itemize}
    \item 
    \textcolor{red}{Based on the prior works \cite{nusken2019state,pathiraja2020mckean,reich2019data},} we derive the McKean-Vlasov system \eqref{eq: true McKean Vlasov}-\eqref{eq:Poisson Gamma intro}, incorporating a stochastic innovation term as well as correlated model and observation noise. Crucially, the dynamics \eqref{eq: true McKean Vlasov}-\eqref{eq:Poisson Gamma intro} is given entirely in terms of Stratonovich integrals that allow the construction of robust filtering schemes \textcolor{red}{built on geometric rough paths}.
    \item
    We prove well-posedness of the rough McKean-Vlasov dynamics \eqref{eq: main system} as well as continuity in the rough driver $\Y$ (see Theorem \ref{thm: main}).
    \item
    We show well-posedness as well as propagation of chaos of the interacting particle approximation \eqref{eq: interacting particles}.
    \item We suggest the RP-EnKF scheme \eqref{eq:numerical scheme} and in particular devise a subsampling based method to estimate the lift components $\mathbb{Y}_k$ from data \textcolor{red}{(as corrections in a model misspecification scenario)}. The robustness of the RP-EnKF (and the nonrobustness of the EnKF) is demonstrated using numerical examples in the context of combined state-parameter estimation.
\end{itemize}

In Section \ref{sec: background in filtering} we review the most relevant results in filtering theory, we motivate the use of the McKean-Vlasov equation and explain the concept of a robust filter.
In Section \ref{sec: preliminaries} we introduce common notation and recall some background on rough paths. In Section \ref{sec: rough McKean-Vlasov} we present the analysis of the rough McKean-Vlasov equation \eqref{eq: main system}, which includes well-posedness and stability.
In Section \ref{sec: particles} we treat the interacting particles system \eqref{eq: interacting particles} and prove well-posedness and propagation of chaos. In Section \ref{sec: numerical examples} we detail the construction of the numerical scheme, including a presentation of our subsampling approaching towards estimating \textcolor{red}{the} L{\'e}vy area.  
Finally, we conclude the paper with some numerical experiments.

\section{Background in filtering, robust representations and McKean-Vlasov dynamics}

In this section we discuss essential background on robust filtering and put our work into perspective. Section \ref{sec:filtering} will summarise existing approaches towards solving the filtering problem posed by \eqref{eq:filtering measures}, both from a theoretical as well as from an algorithmic perspective. In Section \ref{sec: robust filtering} we review the challenges to these methods posed by perturbations in the observed data $(Y_s)_{0 \le  s \le t}$, leading 
to the concept of robustness. In Section \ref{sec:parameter estimation} we draw connections of the McKean-Vlasov approach considered in this paper to maximum likelihood based techniques for stochastic differential equations, in particular to the methods developed in \cite{diehl2016pathwise}. Finally in Section \ref{sec:McKean formulation} we make our McKean-Vlasov formulation as well as the ensemble Kalman approximation precise.

\label{sec: background in filtering}
\subsection{Solutions to the filtering problem and algorithms}
\label{sec:filtering}
 
It is well known that the measure $\pi_t$ defined in \eqref{eq:filtering measures} is a measurable function of the observation path $(Y_s)_{0 \le s \le t}$ and can be obtained as a solution to the \textbf{Kushner-Stratonovich SPDE} \cite[Section 3.6]{bain2008fundamentals},
\begin{equation}
\label{eq:KS SPDE}
\pi_t[\phi] = \pi_0[\phi] + \int_0^t \pi_s[\mathcal{L} \phi] \, \mathrm{d}s + \int_0^t \left( \pi_s[\phi h^\top] - \pi_s[h^\top]\pi_s[\phi] + \pi_s[(B \nabla \phi)^\top] \right) \left( C^{-1}\mathrm{d}Y_s - \pi_s[h] \, \mathrm{d}s\right),
\end{equation}
where
\begin{equation}
\mathcal{L} \phi  = f \cdot \nabla \phi + \frac{1}{2} \operatorname{Trace}(G \nabla^2\phi)
\end{equation}
denotes the infinitesimal generator associated to the signal process \eqref{eq:signal}.
From the computational viewpoint, numerically solving \eqref{eq:KS SPDE} directly (for instance, using grid-based methods) is usually infeasible, especially when the dimension $D$ is large (see \cite[Section 8.5]{bain2008fundamentals} for a discussion). Many algorithmic approaches therefore rely on the simulation of carefully constructed interacting particle systems, positing the corresponding (possibly weighted) empirical measures $\frac{1}{N}\sum_{i=1}^N \delta_{X_t^i}$ as approximations for $\pi_t$.
\\ 

\textbf{Sequential Monte Carlo methods} rely on Bayes' theorem in order to approximate the conditional expectations \eqref{eq:filtering measures}. More precisely, defining the likelihood 
\begin{equation}
\label{eq:likelihood}
l_t = \exp \left( \int_0^t h(X_s) \cdot C^{-1} \, \mathrm{d}Y_s - \frac{1}{2} \int_0^t h(X_s) \cdot C^{-1} h(X_s) \, \mathrm{d}s\right),
\end{equation}
the filtering measures admit the representation 
\begin{equation}
    \pi_t[\phi] = \frac{\mathbb{E}[\phi(X_t) l_t |\mathcal{Y}_t]}{\mathbb{E}[l_t|\mathcal{Y}_t]},
\end{equation}
according to the Kallianpur-Striebel formula \cite[Proposition 3.16]{bain2008fundamentals}. Consequently, approximations of $\pi_t$ can be obtained by sampling from the signal dynamics \eqref{eq:signal} in conjunction with appropriate weighting and/or resampling steps on the basis of \eqref{eq:likelihood}. For detailed accounts, we refer the reader to \cite{doucet2009tutorial, del2004feynman,reich2015probabilistic}. While sequential Monte Carlo methods reproduce the filtering measures exactly in the large-particle limit (see, for instance, \cite[Theorem 9.15]{bain2008fundamentals}), they tend to become unstable in high-dimensional settings due to weight collapse: the conditional and unconditional laws of $X_t$ are often so distinct as to render reweighting-based approaches infeasible due to low effective samples sizes. 
\\

\textbf{Ensemble Kalman filters (EnKFs)}
\cite[Section 7.1]{reich2015probabilistic} can be formulated in terms of interacting or mean-field (McKean-Vlasov) diffusions. In the case when $U=0$ (that is, when the signal and observation noises are uncorrelated  \cite[Section 7.1]{simon2006optimal}, \cite{berry2018correlation}) the basic EnKF due to Evensen \cite{burgers1998analysis,evensen2003ensemble,evensen2009data}  is given by

\begin{equation}
\label{eq:EnKF2}
\mathrm{d}\widehat{X}_t = f(\widehat{X}_t) \, \mathrm{d}t + G^{1/2} \, \mathrm{d}\widehat{W}_t + P(\pi_t) C^{-1} \left(\mathrm{d}Y_t -\left(h(\widehat{X}_t) \, \mathrm{d}t + R^{1/2}\,\mathrm{d}\widehat{V}_t \right) \right),  
\end{equation}
with 
\begin{equation}
\label{eq:Ph}
P(\pi) = \operatorname{Cov}_{\pi}(x,h),
\end{equation}
or a standard particle approximation thereof. The system \eqref{eq:EnKF2}-\eqref{eq:Ph} is motivated by the fact that the corresponding law reproduces $\pi_t$ exactly in the linear Gaussian case: If $\pi_0$ is Gaussian, $f(x) = Fx$ and $h(x) = Hx$ for appropriate matrices $H \in \mathbb{R}^{d \times D}$ and $F \in \mathbb{R}^{D \times D}$, then $\pi_t$ remains Gaussian for all $t \ge 0$, and $\text{Law}(\widehat{X}_t) = \pi_t$. In cases where the preceding conditions are not satisfied, the system \eqref{eq:EnKF2}-\eqref{eq:Ph} becomes an approximation, the accuracy of which is far from well understood theoretically. However, the ensemble Kalman approach has empirically proven to be both fairly reliable in nonlinear settings as well as scalable to high-dimensional scenarios, and therefore nowadays constitutes one of the workhorses in practical data assimilation tasks \cite{reich2015probabilistic}. We refer to \cite{bishop2020mathematical} for a recent review of its theoretical properties.
\\

The \textcolor{red}{more} recently proposed \textbf{feedback particle filters}
\textcolor{red}{
\cite{crisan2010approximate,taghvaei2018kalman,yang2013feedback}} rely on carefully designed McKean-Vlasov diffusions of the type \eqref{eq:EnKF2} such that the associated (conditional, nonlinear) Fokker-Planck equation coincides with the Kushner-Stratonovich SPDE \eqref{eq:KS SPDE}. By construction, such models are exact, and the conditional laws induced by the solutions to feedback particle filter dynamics provide the filtering measures \eqref{eq:filtering measures}.
As an illustration,
\begin{equation}
\label{eq:reich McKean}
    \mathrm{d}\widehat{X}_t = f(\widehat{X}_t) \, \mathrm{d}t + G^{1/2} \, \mathrm{d}\widehat{W}_t + K(\widehat{X}_t,\pi_t) C^{-1} \left(\mathrm{d}Y_t -\left(h(\widehat{X}_t) \, \mathrm{d}t + R^{1/2}\,\mathrm{d}\widehat{V}_t \right) \right) + \Xi(\widehat{X}_t,\pi_t) \, \mathrm{d}t,
\end{equation}
was suggested in \cite{reich2019data} \textcolor{red}{and extended in \cite{nusken2019state}},
where $K(\cdot,\pi) = \nabla \phi(\cdot,\pi)$ is determined from the elliptic PDE
\begin{equation}
\label{eq:simple poisson}
    \nabla \cdot(\pi \nabla \phi) = -\pi\left( h - \pi[h]\right),
\end{equation}
and $\Xi$ is an appropriate correction term (see Section \ref{sec:McKean formulation} for an in-depth discussion). Clearly, the systems \eqref{eq:EnKF2}-\eqref{eq:Ph} and \eqref{eq:reich McKean}-\eqref{eq:simple poisson} are strongly related in spirit, combining a replication of the signal dynamics \eqref{eq:signal} with a data-dependent nudging term so as to match the observations. Reiterating the discussion so far, solutions to \eqref{eq:reich McKean}-\eqref{eq:simple poisson} provide exact solutions to the filtering problem \eqref{eq:filtering measures}, while solutions to \eqref{eq:EnKF2}-\eqref{eq:Ph} lead to approximate ones (except in the linear Gaussian case). However, as $P$ is given explicitly in \eqref{eq:Ph}, the system \eqref{eq:EnKF2}-\eqref{eq:Ph} lends itself straightforwardly to efficient numerical integration, while the system \eqref{eq:reich McKean}-\eqref{eq:simple poisson} poses a formidable numerical challenge in the form of the high-dimensional PDE \eqref{eq:simple poisson}. What is more, well-posedness of systems of the type \eqref{eq:reich McKean}-\eqref{eq:simple poisson}, with coefficients that depend on the law through the solution of a PDE is currently not well understood.  Nevertheless, McKean-Vlasov formulations of the type \eqref{eq:reich McKean}-\eqref{eq:simple poisson} conveniently link between the theoretically optimal Kushner-Stratonovich SPDE \eqref{eq:KS SPDE} and the numerically tractable and practically relevant ensemble Kalman dynamics \eqref{eq:EnKF2}-\eqref{eq:Ph}. In this paper, we leverage this viewpoint in order to construct a robust version of \eqref{eq:EnKF}. 

\subsection{Robust filtering}
\label{sec: robust filtering}

In order to model and solve real-life problems it is highly desirable that the conditional law $\pi_t$ (or any numerical approximation thereof) depends continuously on the observation path $(Y_s)_{0 \le s \le t}$: This property would ensure robustness against misspecification of the underlying signal and observation dynamics (as is typical in reduced-order modeling) as well as against anomalies or artefacts in the collection of the data (such as discretisation errors or perturbation by noise), see \cite[Chapter 5]{bain2008fundamentals} for an overview,  
\cite{clark1978design,clark2005robust,davis1980multiplicative,kushner1980robust,davis1982pathwise,davis2011pathwise,davis1987pathwise}
and Section \ref{sec:parameter estimation} below. 
Unfortunately, however, the $\phi$-dependent measurable map $(Y_s(\cdot))_{s\in [0,t]} \mapsto \pi_t[\phi]$ provided by \eqref{eq:filtering measures} can be shown to be neither unique nor continuous \cite{crisan2013robust} in standard topologies. At a fundamental level, this problem is due to the appearance of stochastic integrals against $(Y_s)_{0 \le s \le t}$ in \eqref{eq:KS SPDE} and \eqref{eq:likelihood} which are well known to induce classically discontinuous maps (for instance, in the supremum norm), see \cite{friz2020course,friz2010multidimensional}. In order to address this issue and to obtain a continuous\footnote{As pointed out in \cite{crisan2013robust}, the continuity requirement restores the uniqueness of the map $(Y_s(\cdot))_{s\in [0,t]} \mapsto \pi_t[\phi]$.}  version of the process $(\pi_t)_{t \ge 0}$, Clarke suggested using (stochastic) integration by parts in \eqref{eq:likelihood} in order to eliminate the $\mathrm{d}Y$-dependence \cite{clark1978design}. Notably, this approach is restricted to the case when the signal and observation noises are independent (that is, $U=0$), see \cite{clark1978design,clark2005robust,kushner1980robust}, or the case when the observation is one-dimensional (that is, $Y_t \in \mathbb{R}$), see \cite{davis1980multiplicative,davis1982pathwise,davis2011pathwise,davis1987pathwise,florchinger1993zakai}. Addressing the situation of multidimensional  correlated observations, the authors of \cite{crisan2013robust} showed by means of a counter-example (see \cite[Example 1]{crisan2013robust}) that continuity as a function of $(Y_s)_{0 \le s \le t}$ is impossible to achieve. Instead, they
use rough path lifts $C([0,T];\mathbb{R}^d) \ni (Y_s)_{0 \le s \le T} \mapsto (\mathbf{Y}_s)_{0 \le s \le T} \in \mathscr{C}^{\alpha}([0,T];\R^{d})$ and establish continuity when $\pi_t$ is considered as a function of the augmented observation path. Similar ideas  have been pursued in \cite[Theorem 5.3]{gyongy1989approximation}, putting forward the notion of `good' approximations of the observation path. As the aforementioned works are concerned with the likelihood \eqref{eq:likelihood}, these lay the foundations for the development of robust sequential Monte Carlo methods as reviewed in Section \ref{sec:filtering}, and the recent preprint \cite{crisan2021pathwise} 
explores that direction. 
We would also like to mention the works \cite{diehl2017stochastic,hocquet2018energy} that allow treating the Zakai SPDE (governing the unnormalised filtering distribution \cite[Section 3.5]{bain2008fundamentals})  in a rough paths framework, however noticing that the numerical treatment of SPDEs is faced with enormous challenges, in particular in high-dimensional settings. Some other works addressing issues in robust or multiscale filtering include \cite{allan2020robust,allan2020pathwise} (assuming uncertainty in the coefficients) as well as a sequence of works by N. Perkowski and coworkers in the context of averaging and homogenisation \cite{beeson2018reduced,beeson2020approximation,beeson2020quantitative,imkeller2012homogenization,imkeller2013dimensional,lingala2012particle,lingala2014optimal,yeong2018dynamic,yeong2020particle}.
\\

In this paper, we instead construct a robust version of the ensemble Kalman filter \eqref{eq:EnKF2}-\eqref{eq:Ph} on the basis of its connections to feedback particle formulations as in \eqref{eq:reich McKean}-\eqref{eq:simple poisson}. Before describing our strategy, we review related work on maximum likelihood parameter estimation for stochastic differential equations.





\subsection{Parameter estimation and filtering in multiscale systems}
\label{sec:parameter estimation}

In this section we present a prototypical example that illustrates some of the challenges in robust filtering as well as the scope of the methods developed in this paper, \textcolor{red}{following \cite{nusken2019state}}. Consider the SDE
\begin{equation}
\label{eq:par est dynamics}
\mathrm{d}Z_t = F(Z_t,\theta) \, \mathrm{d}t + \mathrm{d}W_t,
\end{equation}
where $Z_t \in \mathbb{R}^d$, and $F:\mathbb{R}^d \times \Theta \rightarrow \mathbb{R}^d$ is a parameterised drift vector field, with parameter set $\Theta \subset \mathbb{R}^p$. The objective is to find the true parameter $\theta^* \in \Theta$ from a noisy realisation of $(Z_t)_{t \ge 0}$, that is, we assume that the observation process is given by 
\begin{equation}
\label{eq:par est observation}
\mathrm{d}Y_t = \mathrm{d}Z_t + R^{1/2} \, \mathrm{d}V_t.
\end{equation}
As before, $R \in \mathbb{R}^{d \times d}$ denotes the observation noise covariance, and $(V_t)_{t \ge 0}$ stands for a standard $d$-dimensional Brownian motion. The problem setting \eqref{eq:par est dynamics}-\eqref{eq:par est observation} can be brought into the form \eqref{eq:signal obs} by elevating $\theta$ to a time-dependent variable, that is, by setting $X_t = (Z_t,\theta_t) \in \mathbb{R}^{d + p}$, hence viewing \eqref{eq:par est dynamics}-\eqref{eq:par est observation} as a combined state-parameter estimation problem, see \cite{nusken2019state}. Accordingly, $f:\mathbb{R}^{d + p} \rightarrow \mathbb{R}^{d + p}$ and $h:\mathbb{R}^{d+p} \rightarrow \mathbb{R}^d$ are then given as $f(z,\theta) = (F(z,\theta), 0)$ and $h(z,\theta) = F(z,\theta)$, and the matrices $G \in \mathbb{R}^{(d+p)\times (d+p)} $ and $U \in \mathbb{R}^{d \times (d+p)}$ take the form
\begin{equation}
G = \begin{pmatrix}
I_{d \times d} & 0_{d \times p} 
\\
0_{p \times d} & 0_{p \times p}
\end{pmatrix}, \qquad \qquad U = \begin{pmatrix}
I_{d \times d} & 0_{p \times p}
\end{pmatrix}.
\end{equation}
Finally, the filtering formulation is completed by specifying a \emph{prior distribution} on the initial condition $(Z_0,\theta_0)$. The resulting filtering measures $\pi_t \in \mathcal{P}(\mathbb{R}^{d+p})$ encode the Bayesian posterior on the combined variable $(Z_t,\theta_t)$. Consequently, the $\theta$-marginals provide meam a posteriori estimates on the parameter of interest as well as corresponding bounds on Bayesian uncertainty. 
\\ 

In the particular case when the path $(Z_t)_{t \ge 0}$ is observed without contamination by noise, that is, $R=0$, and $F$ is linear\footnote{For simplicity of the presentation, we also assume here that $\theta$ is one-dimensional, that is, $\Theta \subset \mathbb{R}$.} in $\theta$, that is, $F(z,\theta) = \theta f(z)$ with $f$ satisfying appropriate nondegeneracy conditions \cite{diehl2016pathwise}, the parameter $\theta \in \Theta$ can be recovered from the maximum likelihood estimator
\begin{equation}
\label{eq:MLE}
\theta^*_T(Z) = \frac{\int_0^T f(Z_t) \,\mathrm{d}Z_t }{\int_0^T |f(Z_t)|^2\,\mathrm{d}t}    
\end{equation}
in the limit when $T \rightarrow \infty$, see \cite{kutoyants2013statistical,liptser2013statistics}. Furthermore, in this case the McKean-Vlasov dynamics suggested in this paper can be solved explicitly, and the corresponding means are directly related to \eqref{eq:MLE}, see Appendix \ref{app:MLE} and \cite{nusken2019state}. 
It is well known that the estimator \eqref{eq:MLE} can be inaccurate when evaluated on paths that only approximately satisfy \eqref{eq:par est dynamics}, for instance when \eqref{eq:par est dynamics} represents a reduced description of an underlying multiscale dynamics \cite{ait2005often,pavliotis2007parameter,zhang2005tale}. A common approach towards addressing this problem is to subsample the data, see \cite{abdulle2020drift,azencott2013sub,azencott2010adaptive,gailus2017statistical,gailus2018discrete,kalliadasis2015new,krumscheid2013semiparametric,krumscheid2015data,papavasiliou2009maximum,pavliotis2012parameter} for methodological aspects. For specific  applications see  \cite{nolen2012multiscale} (multiscale inverse problems), \cite{ait2005often,olhede2010frequency,zhang2005tale} (economics and finance), and \cite{cotter2009estimating,ying2019bayesian}  (ocean and atmospheric science). 

A different approach towards robustness of the estimator \eqref{eq:MLE} has been taken in \cite{diehl2016pathwise}, addressing the discontinuity of the It{\^o} integral $\int_0^T f(Z_t) \,\mathrm{d}Z_t$.
To resolve this issue, the authors suggest replacing It{\^o} by Stratonovich integration (motivated by the Wong-Zakai theorem \cite[Theorem 9.3]{friz2020course} and entailing a correction term involving $\nabla f$), and subsequently using rough paths integration instead of Stratonovich integration (relying on a suitable lift $C([0,T];\mathbb{R}^d) \ni (Z_s)_{0 \le s \le T} \mapsto (\mathbf{Z}_s)_{0 \le s \le T} \in \mathscr{C}^{\alpha}([0,T];\R^{d})$). 
The resulting estimator
\begin{equation}
\theta_T^{\mathrm{RP}}(\mathbf{Z}) = \frac{\int_0^T f(Z_t) \,\mathrm{d}\mathbf{Z}_t - \tfrac{1}{2} \int_0^T \operatorname{Trace}(\nabla f)(Z_t) \, \mathrm{d}t }{\int_0^T |f(Z_t)|^2\,\mathrm{d}t}
\end{equation}
can then be shown to be continuous as a map from $\mathscr{C}^{\alpha}([0,T];\R^{d})$ to $\mathbb{R}$. The construction of the RP-EnKF dynamics \eqref{eq:numerical scheme} follows a similar line of reasoning, but our approach is applicable to situations where the path $(Z_t)_{t \ge 0}$ is contaminated by noise ($R \neq 0$) and where $F(z,\theta)$ is nonlinear in $\theta$. The latter generalisation makes our method suitable to applications involving deep learning, that is, when the drift in \eqref{eq:par est dynamics} is parameterised by a neural network as in \cite{gottwald2020supervised}. 
We discuss the idea of subsampling the observed data path in the context of constructing an appropriate rough path lift in Remark \ref{rem:subsampling} below in Section \ref{sec: numerical examples}.

\subsection{From the filtering problem to the McKean-Vlasov equation}
\label{sec:McKean formulation}

In this section, we discuss the construction of the McKean-Vlasov system \eqref{eq: true McKean Vlasov}-\eqref{eq:Poisson Gamma intro} and the nature of the approximation in \eqref{eq: main system stochastic version} and \eqref{eq: main system}.
The general idea goes back to \textcolor{red}{\cite{crisan2010approximate,yang2013feedback}}, and various modifications have been proposed in \cite{nusken2019state,pathiraja2020mckean,reich2019data}. Our formulation combines Stratonovich integration (as in \cite{pathiraja2020mckean}) in order to later invoke Wong-Zakai type approximation results with a stochastic innovation term (as in \cite{nusken2019state,reich2015probabilistic}) as required for the case of correlated model and observation noise.
For the sake of clarity, we repeat the PDEs \eqref{eq:Poisson K intro} and \eqref{eq:Poisson Gamma intro} in their respective index forms (using Einstein's summation convention),
\begin{equation}
\label{eq:Poisson indices}
\partial_i \left(\widehat{\pi}_t \left(K_t^{ij} - (BC)^{ij} \right) \right) = - \widehat{\pi}_t \left(h^j - \widehat{\pi}_t[h^j] \right), \qquad j = 1, \ldots, d, 
\end{equation}
and 
\begin{equation}
\label{eq:Poisson indices2}
\partial_i (\widehat{\pi}_t\Xi_t^i) = \frac{1}{2} \widehat{\pi} \left( (K_tC^{-1})^{ij} \partial_i h^j - \widehat{\pi}_t [(K_tC^{-1})^{ij} \partial_i h^j]\right).
\end{equation}
\textcolor{red}{In the case when model and observation noise are uncorrelated ($U = 0$), testing \eqref{eq:Poisson indices} with $h^j$ reveals that the system \eqref{eq:Poisson indices}-\eqref{eq:Poisson indices2} is equivalent to the system (2.5)-(2.6) obtained
in \cite{pathiraja2020mckean}.} 
The McKean-Vlasov system \eqref{eq: true McKean Vlasov}-\eqref{eq:Poisson Gamma intro} solves the filtering problem in the following sense:
\begin{proposition}
\label{prop:McKean}
Let $T > 0$,
assume that the system \eqref{eq:signal obs} admits a unique solution $(X_t, Y_t)$ and that the Zakai equation associated to the corresponding filtering problem  is well posed. Let $\pi$ be the conditional law of $X$ given $Y$ as defined in \eqref{eq:filtering measures} and assume that $\pi_t$ admits a $C^1$-density with respect to the Lebesgue measure, $\mathbb{P}$-a.s., for all $t \in[0,T].$
Moreover, assume that the McKean-Vlasov equation \eqref{eq: true McKean Vlasov} admits a unique solution $(\widehat{X}_t)_{t \in [0,T]}$ such that its conditional law $\widehat{\pi}_t$ admits a $C^1$-density with respect to the Lebesgue measure, $\mathbb{P}$-a.s.. 
Assume that $K$ and $\Xi$ are predictable stochastic processes with values in $C^1(\mathbb{R}^D;\mathbb{R}^{D \times d})$ and $C^1(\mathbb{R}^D;\mathbb{R}^D)$, respectively, independent from $(\widehat{V}_t)_{t \ge 0}$ and $(\widehat{W}_t)_{t \ge 0}$, and such that the PDEs \eqref{eq:Poisson K intro} and \eqref{eq:Poisson Gamma intro} are satisfied, $\mathbb{P}$-a.s.. Then $\pi_t = \widehat{\pi}_t$, for all $t \in [0,T]$.
\end{proposition}
To prove Proposition \ref{prop:McKean},
we define the unnormalised conditional law associated to the McKean-Vlasov dynamics \eqref{eq: true McKean Vlasov},
\begin{equation}
\widehat{\rho}_t[\phi] := \mathbb{E}[\phi(X_t)l_t | \mathcal{Y}_t],
\end{equation}
where the likelihood $l_t$ is defined in \eqref{eq:likelihood}. Comparing the evolution of $\widehat{\rho}_t$ with the solution of the Zakai equation \cite[Section 3.5]{bain2008fundamentals} allows us to derive the PDEs \eqref{eq:Poisson K intro} and \eqref{eq:Poisson Gamma intro}. This approach allows us to circumvent the stringent regularity condition in \cite[Assumption 3.4]{pathiraja2020mckean}.
For details see Appendix \ref{app:proof McKean}.

One of the numerical challenges posed by the system of equations \eqref{eq: true McKean Vlasov}-\eqref{eq:Poisson Gamma intro} is to obtain (approximate) solutions $K$ and $\Xi$ to the PDEs \eqref{eq:Poisson K intro} and \eqref{eq:Poisson Gamma intro}. We sidestep this problem by using constant-in-space approximations leading to the system \eqref{eq: main system stochastic version} of ensemble Kalman filter type. A similar correspondence has been observed in \cite{taghvaei2018kalman} and is optimal in the following sense:

\begin{lemma}
\label{lem: approximate K}
Let $\pi \in \mathcal{P}(\R^D)$ and $K: \R^D \to \R^{D\times d}$ be a (weak) solution to \eqref{eq:Poisson K intro}. Denote by $\widetilde{K}$ the best constant approximation in least-squares sense, that is
\begin{equation}
\label{eq: least-squares}
    \widetilde{K} = \argmin_{\widetilde{K} \in \mathbb{R}^{D \times d}} \int_{\mathbb{R}^D} \Vert K(x) - \widetilde{K} \Vert_F^2 \, \mathrm{d}\pi(x),
\end{equation}
where 
$\Vert \cdot \Vert_F$ denotes the Frobenius norm. Then $\widetilde{K}$ is given by
\begin{equation}
\widetilde{K} = \int_{\mathbb{R}^D} x \left( h(x) - \pi[h] \,  \right)^{\top}\mathrm{d}\pi(x) - BC \in \mathbb{R}^{D \times d}.
\end{equation}
Moreover, let $\Xi$ be a solution to \eqref{eq:Poisson Gamma intro} with $K$ replaced by $\widetilde{K}$. Then the best constant approximation of $\Xi$ in least-squares sense (with respect to the Euclidean norm) is given by
\begin{equation*}
\widetilde{\Xi}^\gamma = -\frac{1}{2}\operatorname{Trace}(\widetilde{K}C^{-1}\pi[x^{\gamma}(Dh-\pi[Dh])), \qquad \qquad 1 \le \gamma \le D.
\end{equation*}
\end{lemma}

\begin{proof}
The variance \eqref{eq: least-squares} is minimised when $\widetilde{K} = \pi[K]$. To compute the expectation of $K$, we multiply equation \eqref{eq:Poisson K intro} on both sides by $x \in \R^D$ and integrate by parts to obtain
\begin{equation*}
    \pi[K] - BC = \int_{\R^D}(K(x) - BC) \,\mathrm{d}\pi(x) = \int_{\mathbb{R}^D} x \left( h(x) - \pi[h] \,  \right)^{\top}\mathrm{d}\pi(x).
\end{equation*}
Hence, $\tilde{K} = \pi[K] = \operatorname{Cov}_{\pi}(x,h) - BC$.

We now show the corresponding statement for $\widetilde{\Xi}$ when  $K$ is replaced by $\widetilde{K}$ in \eqref{eq:Poisson Gamma intro}. Again $\widehat{\Xi} = \mathbb{E}[\Xi]$, since the expectation minimises the variance. For $1\leq \gamma \leq D$, we input the test function $\phi(x) = x^{\gamma}$ into the weak formulation \eqref{eq: weak poisson gamma} of \eqref{eq:Poisson Gamma intro} to obtain the desired expression for $\tilde{\Xi}$.
\end{proof}

Constructing numerical approximations for \eqref{eq:Poisson K intro}-\eqref{eq:Poisson Gamma intro} beyond the constant-in-space approximation is a topic of ongoing research. We mention in particular the approach developed in \cite{taghvaei2020diffusion} and analysed in \cite{pathiraja2021analysis} based on diffusion maps as well as the method from \cite{nusken2021Stein} based on the Stein geometry \cite{duncan2019geometry,nusken2021steinLD}.

\subsection{Literature on McKean-Vlasov dynamics}
McKean-Vlasov equations are stochastic differential equations whose coefficients depend on the law of the solution. They are sometimes called law-dependent equations.
McKean-Vlasov equations have been the subject of several studies starting from the seminal work of McKean \cite{mckean1966} and Dobrushin \cite{dobrusin1979vlasov}.
McKean-Vlasov equations arise as limit of mean-field interacting particle systems, when the number of particles goes to infinity. 
For a general introduction on the topic we refer the reader to Sznitman \cite{sznitman1991topics}. 

In recent years there has been an increased interest in mean-field particles with common noise, see \cite{coghi2016propagation, kurtz1999particle, kurtz2001numerical}
or \cite{carmona2018probabilisticII} in the case of mean-field games. In these type of systems the particles are subject to the same random perturbation and possibly additional independent noises. There is no averaging effect of the common perturbation when the number of particles increases. The limit object is again a law-dependent SDE, but this time the coefficient depends on the conditional law of the solution given the common noise. This is the case for equation \eqref{eq: main system}, where the coefficients $P$ and $\Gamma$ depend on the law of $X$ given $Y$.

McKean-Vlasov equations from a rough path perspective were studied for the first time in \cite{cass2015evolving} and more recently in the twin papers \cite{ bailleul2019propagation,bailleul2020solving}. 
In both of these works the equation is driven by a random rough path  that is quite general and can describe the independent noise, the common noise or both. This gives the additional difficulty of needing to keep track of the rough path as a $L^p$-valued path. 
In \cite{cass2015evolving} only the drift of the equation depends on the law of the solution, the coefficients in front of the noise depend only on the state of the solution. The more recent work \cite{bailleul2020solving} generalises that approach to include law-dependent coefficients. The authors use the approach by Gubinelli on controlled rough paths (see Section \ref{sec: background in rough paths} for a brief introduction on the topic). In order to do this, they need Lions' approach to calculus in  measure spaces endowed with the Wasserstein metric. The equation is then solved as a fixed-point in the mixed $\R^d$ and $L^p$-space.
In \cite{coghi2020pathwise} the case of pathwise McKean-Vlasov equation with additive noise is considered. The basic techniques used are similar to the ones used in the rough-path case, but the need for rough paths is removed thanks to the additive noise. See also \cite{tanaka1982limit}.

McKean-Vlasov equations with a rough common noise have been studied recently in \cite{coghi2019rough}. This is the first time that a rough McKean-Vlasov equation is studied when there is a clear separation between independent Brownian motions and a common deterministic noise. In \cite{coghi2019rough}, the common noise coefficient depends on both the state of the solution and its law. In equation \eqref{eq: main system}, the coefficient $P$ only depends on the law of $\widehat{X}$, which simplifies the problem to some extent, but also allows us to use a different approach, where we treat the stochastic and the rough integrals in two separated steps. As it will be clear from the proofs in Section \ref{sec: rough McKean-Vlasov}, there is no need to create a joint rough path (or rough driver).
All previous works on rough McKean-Vlasov equations deal with bounded coefficients, which cannot be applied here, as the ceofficient $P$ has linear growth in the measure of the solution.

Very recently the authors of \cite{friz2021rough} developed a theory of mixed rough and stochastic differential equations under Lipschitz and boundedness conditions on the coefficients and they plan to address the application to McKean-Vlasov equations with common noise in a forthcoming paper.

\section{Preliminaries}
\label{sec: preliminaries}
\subsection{Notation}
Given a metric space $(S,d_S)$, we call $\mathcal{P}(S)$ the space of probability measures on $S$. For $\pi\in \mathcal{P}(S)$ we denote by $\pi[x] = \int_{S}x\pi(\mathrm{d}x)$ the mean of $\pi$ and by $\pi[\phi] = \int_{S}\phi(x)\pi(\mathrm{d}x)$ the integral of a measurable function $\phi:S \to \R$ in $\pi$.

Let $\rho >0$, if $S$ is a normed space with norm $|\cdot|$ and $\pi \in \mathcal{P}(S)$, we denote by
\begin{equation*}
    M^{\rho}(\pi) = \int_{S} |x|^\rho \pi(\mathrm{d}x),
    \qquad
    \overline{M}^{\rho}(\pi) = \int_{S}\big|x - \pi[x]\big|^\rho \pi(\mathrm{d}x),
\end{equation*}
 the $\rho$-moment of $\pi$ and the $\rho$-central moment of $\pi$, respectively.
For $\rho \geq 1$, we call $\mathcal{P}_{\rho}(S)\subset \mathcal{P}(S)$ the space of probability measures $\pi$ on $S$ such that $M^{\rho}(\pi) < \infty$. We endow this space with the $\rho$-Wasserstein metric
\begin{equation*}
W^\rho_{\rho,S}(\mu,\nu):= \inf_{m\in \Gamma{(\mu,\nu)}} \iint_{S\times S}d^\rho_S(x,y) m(\mathrm{d} x,\mathrm{d}y),
\end{equation*}
where $\Gamma(\mu,\nu)$ is the set couplings between $\mu$ and $\nu$. 
For ease of notation, we denote by $W_{\rho}$ the Wasserstein metric on $\R^D$ and by $W_{\rho,[0,T]}$ the Wasserstein metric on $(C([0,T],\R^D), \|\cdot \|_{\infty})$.
Given a function $\varphi \in C(\R^m, \R^n)$ and a path $x \in C([0,T], \R^m)$, we define
\begin{equation*}
[\varphi]_{s,t}^{k,x} := \int_{0}^{1} (1-\theta)^{k-1} \varphi(x_s + \theta \delta x_{s,t}) \, \mathrm{d}\theta,
\end{equation*}
where $\delta x_{s,t} := x_t - x_s$.
If $\varphi \in C^1(\R^m, \R^n)$, we use the following notation for the Taylor expansion
\begin{equation*}
\delta \varphi(x_{\cdot})_{s,t} = [D\varphi]_{s,t}^{1,x} \delta x_{s,t},
\qquad
[\varphi]_{s,t}^{1,x} - \varphi(x_s) = [D \varphi]_{s,t}^{2,x} \delta x_{s,t}.
\end{equation*}
Throughout the paper we use $D\varphi$ for the usual Fr\'echet derivative and $\nabla \varphi = D\varphi^\top$.
We sometimes use the notation $L^{\rho}_{\omega}:=L^{\rho}(\Omega,\R^d)$. For a stochastic process $X$ on a filtered probability  space $(\Omega, \mathcal{F}, (\mathcal{F}_t)_{t \ge 0}, \mathbb{P})$, we denote the conditional expectation by $\mathbb{E}_s[X_t]:= \mathbb{E}[X_t \mid \mathcal{F}_s]$.

\subsection{Background in rough paths}
\label{sec: background in rough paths}

The theory of rough paths is a framework that allows well-posedness and stability properties for equations of the form 
\begin{equation} \label{differential equation}
    \dot{X}_t = \xi(X_t) \dot{Y}_t , \quad X_0 \in {\R}^{d},
\end{equation}
where $Y$ is a path of regularity lower than the regularity assumptions amenable to classical calculus. For $\alpha > 0$, we denote by $C_2^{\alpha}([0,T];\R^{d})$ the set of all continuous functions
$$
g : \{ (s,t) \in [0,T]^2 : s < t \} \rightarrow {\R}^{d}
$$
such that there exists a constant $C$ with $|g_{s,t}| \leq C |t-s|^{\alpha}$ and we denote by $[g]_{\alpha}$ the infimum over all such constants. We denote by $\|g\|_{\alpha} := [g]_{\alpha}+|g_0|$ the $\alpha$-Hölder norm. We write $C^{\alpha}([0,T];\R^d)$ for the set of paths $f : [0,T] \rightarrow \R^d$ such that $\delta f \in C_2^{\alpha}([0,T];\R^d)$, where we have defined $\delta f_{s,t} := f_t - f_s$.

A \emph{rough path} is a pair $\bY = (Y,\mathbb{Y}) \in \mathscr{C}^{\alpha}([0,T];\R^{d}) \subset C^{\alpha}([0,T];\R^{d}) \times C^{2 \alpha}_2([0,T];\R^{d \times d})$ such that 
\begin{align}
     \mathbb{Y}_{s,t} - \mathbb{Y}_{s,u} - \mathbb{Y}_{u,t} & = Y_{s,u} \otimes Y_{u,t} \label{Chen's relation}  .
\end{align}
We equip $\mathscr{C}^{\alpha}([0,T];\R^{d})$ with its subset topology which we shall call the rough path topology. Relation \eqref{Chen's relation}, commonly referred to as Chen's relation, encodes the algebraic property between a path and its \emph{iterated integral}, viz the formal equality
$$
\mathbb{Y}_{s,t} = \int_s^t Y_{s,r} \otimes \mathrm{d} Y_r.
$$
When $\alpha \in (\frac13,\frac12]$ the above integral is in general not canonically defined using functional analysis. However, in the case of $Y$ being a sample path of the Brownian motion, $B$, one can use probability theory to define iterated integrals using e.g. Itô integration or Stratonovich integration. We denote by $\bB := (B,\mathbb{B}^{Ito}) := (B,\int B \otimes \mathrm{d} B)$ and $\bB^{Strat} := (B,\mathbb{B}^{Strat}) := (B, \int B \otimes \circ \mathrm{d} B)$ these (random) rough paths, respectively. It is classical that 
\begin{equation} \label{ito to strat}
    \mathbb{B}_{s,t}^{Ito} = \mathbb{B}_{s,t}^{Strat} \textcolor{red}{-} \frac12(t-s) I_{d \times d}.
\end{equation}
The Stratonovich rough path is an example of a \emph{geometric} rough path, that is to say it is in the closure in the rough path topology of the image of the mapping 
$$
Y \mapsto (Y, \int Y \otimes \mathrm{d} Y)
$$
defined on $BV([0,T]; \R^d)$.

Given two rough paths $\bY^1, \bY^2 \in \mathscr{C}^{\alpha}([0,T],\R^d)$ we define the following distance,
\begin{equation*}
    \rho_{\alpha}(\bY^1, \bY^2) := \|Y^1-Y^2\|_{\alpha} + \|\YY^1 - \YY^2\|_{2\alpha}.
\end{equation*}
We refer the reader to \cite{friz2020course} for a more comprehensive discussion of the rough paths notations and concepts used in this paper.

\subsection{Controlled rough paths and rough differential equations}
\label{sec: controlled rough-paths}
To use rough paths for a solution theory of equations of the form \eqref{differential equation}, which we rewrite with the formal expression
\begin{equation} \label{rough differential equation}
    \mathrm{d} X_t = \xi(X_t) \,\mathrm{d} \bY_t,
\end{equation}
we start with the ansatz that the solution $X$ takes the form of a Taylor-like expansion 
\begin{equation} \label{controlled}
\delta X_{s,t} = X_s' \delta Y_{s,t} + X_{s,t}^{\sharp},    
\end{equation}
where $X^{\sharp}$ is of higher regularity than $X$, and $X'$ is the so-called Gubinelli derivative. We denote by $\mathscr{D}^{2\alpha}_Y([0,T];\R^n)$ the set of all pairs $(X,X')$ such that $X^{\sharp}$ implicitly defined via \eqref{controlled} satisfies $(X',X^{\sharp}) \in C^{\alpha}([0,T];\mathcal{L}(\R^d;\R^n)) \times C^{2 \alpha}_{2}([0,T];\R^n)$, which also induces the topology on $\mathscr{D}^{2\alpha}_Y([0,T];\R^n)$. 

The sewing lemma provides a continuous integration mapping 
$$
\begin{array}{ccc}
\mathscr{D}^{2\alpha}_Y([0,T];\mathcal{L}({\R}^n,{\R}^d)) & \longrightarrow & \mathscr{D}^{2\alpha}_Y([0,T];{\R}^n) \\
(X,X') & \longmapsto & \left( \int X_r \, \mathrm{d} \bY_r, X\right), 
\end{array}
$$
where 
$$
X_{s,t}^{\natural} := \int_s^t X_r \mathrm{d} \bY_r - X_s Y_{s,t} - X_s' \mathbb{Y}_{s,t}
$$
satisfies $|X_{s,t}^{\natural}| \leq C|t-s|^{3 \alpha}$ for some constant $C$ only depending on $\alpha$. A solution of \eqref{rough differential equation} can now be defined as a fixed point of the composition of the mappings
$$
\begin{array}{ccccc}
\mathscr{D}^{2\alpha}_Y([0,T];{\R}^n) & \longrightarrow & \mathscr{D}^{2\alpha}_Y([0,T];\mathcal{L}({\R}^n,{\R}^d)) & \longrightarrow & \mathscr{D}^{2\alpha}_Y([0,T];{\R}^n) \\
(X,X')  & \longmapsto & (\xi(X),\xi(X)') = (\xi(X),\nabla \xi(X) X') & \longmapsto & \left( \int \xi(X_r) \mathrm{d} \bY_r, \xi(X)\right) .
\end{array}
$$
From the sewing lemma and the definition of the integration mapping we see that we could equivalently define the solution of \eqref{rough differential equation} as a path $X : [0,T] \rightarrow \R^n$ such that 
$$
X^{\natural}_{s,t} : = \delta X_{s,t} - \xi(X_s) Y_{s,t} - \nabla \xi(X_s) \xi(X_s) \mathbb{Y}_{s,t}
$$
satisfies $|X^{\natural}_{s,t}| \leq C|t-s|^{3 \alpha}$. The latter formulation is usually referred to as Davie's expansion/solution. 

One of the remarkable properties of rough path equations is the continuity of the Itô-Lyons map,
$$
\begin{array}{ccc}
\mathscr{C}^{\alpha}([0,T];{\R}^d) & \longrightarrow & C^{\alpha}([0,T];{\R}^n) \\
\bY & \longmapsto & X 
\end{array}
$$
where $X= X^{\bY}$ denotes the solution of \eqref{rough differential equation}, provided $f$ is regular enough. In fact, Theorem \ref{thm: main} is an analogous result for the McKean-Vlasov dynamics of the ensemble Kalman filter  treated in this paper.

\section{Stochastic rough McKean-Vlasov equations}
\label{sec: rough mckean}

In this section, unless otherwise specified we fix $T>0$ and $\rho \geq 1$. 
Moreover, $(\Omega, \mathcal{F}, (\mathcal{F}_t)_{t \ge 0}, \mathbb{P})$ is a complete filtered probability space that supports a standard $m$-dimensional Brownian motion $W$.
Recall that for $\pi \in \mathcal{P}_{\rho}(\R^D)$, we call $M^{\rho}(\pi)$ the $\rho$-moment of $\pi$ and $\overline{M}^{\rho}(\pi)$ the $\rho$-central moment of $\pi$.

\subsection{McKean-Vlasov SDEs with linear growth in the measure}

Consider the measurable functions $b:[0,T]\times \R^D \times \mathcal{P}_{\rho}(\R^D) \to \R^D$ and $\sigma : [0,T] \times \R^D \times \mathcal{P}_{\rho}(\R^D) \to \R^{D\times m}$ satisfying the following assumptions:
\begin{assumption}
	\label{assumptions coefficients time dependent}
	Assume that there exists a constant $C > 0$ such that
	\begin{enumerate}[label=(\roman*), ref= \ref{assumptions coefficients time dependent} (\roman*)]
		\item \label{assumpitons coefficients time dependent: linear growth} 
		(linear growth)  $\forall t \in [0,T]$, $\forall x \in \R^D$,$ \forall \pi \in \mathcal{P}_{\rho}(\R^D)$,
		\begin{equation}
		\label{eq: linear growth time dependent}
		|b(t,x,\pi) | , | \sigma(t,x, \pi) | \leq C (1 + M^{\rho}(\pi)  )^{\frac{1}{\rho}} ,
		\end{equation}
		\item  \label{assumpitons coefficients time dependent: lipschitz}
		(locally Lipschitz) $\forall t \in [0,T]$, $\forall x,y \in \R^D$, $\forall \pi, \nu \in \mathcal{P}_{\rho}(\R^D)$, 
		\begin{equation}
		\label{eq: lipschitz condition time dependent}
		| b(t,x, \pi) - b(t,y, \nu) | , 
		| \sigma(t,x, \pi) - \sigma(t,y, \nu) | 
		\leq C ( 1 + M^{\rho}(\pi))^{\frac{1}{\rho}} \left(|x - y| + W_{\rho,\R^D}(\pi, \nu)\right).
		\end{equation}
	\end{enumerate}
\end{assumption}
\begin{remark}
Notice that, if $\rho \geq \rho^{\prime}$, Assumption \ref{assumptions coefficients time dependent} with $\rho^{\prime}$ implies Assumption \ref{assumptions coefficients time dependent} with $\rho$.
\end{remark}
Consider the following McKean-Vlasov equation,
	\begin{equation}
	\label{eq: mckean with linear growth}
	\mathrm{d} X_t = b(t, X_t, \mathcal{L}(X_t)) \,\mathrm{d} t + \sigma(t, X_t, \mathcal{L}(X_t)) \,\mathrm{d} W_t,
	\qquad
	 X_t|_{t=0} = X_0 \in L^2_{\omega},
\end{equation}
where $W$ is an $m$-dimensional Brownian motion.

\begin{lemma}
	\label{lem: well-posedness mckean-vlasov linear growth}
	Under Assumption \ref{assumptions coefficients time dependent} with $\rho \geq 2$, equation \eqref{eq: mckean with linear growth} admits a unique strong solution in the classical sense. Moreover, the solution has continuous sample paths.
\end{lemma}
	
\begin{proof}
	To prove well-posedness we use a Picard-type argument. 
	For $T\geq 0$ and $K > 2 M^2(\pi_0) \vee 1$, we define
	$$
	B_K := \left\{\pi \in \mathcal{P}(C([0,T],{\R}^D) \mid M^{2}(\pi_t) \leq K  \,\, \forall t \in [0,T],\; M^{2}(\pi_0) \leq K/2 \right\}.
	$$
	Let $\pi \in B_K$. We define the following stochastic differential equation
	\begin{equation}
	\label{eq: standard sde}
	\mathrm{d} X_t = b(t, X_t, \pi_t) \,\mathrm{d} t + \sigma(t, X_t, \pi_t) \,\mathrm{d} W_t,
	\qquad
	\tilde X_t|_{t=0} = X_0 \sim \pi_0.
	\end{equation}
	By our choice of $\pi$ and Assumption \ref{assumptions coefficients time dependent}, the coefficients $b$ and $\sigma$ are bounded and Lipschitz. By standard theory this equation admits a unique strong solution on the interval $[0,T]$, which we denote by $X^{\pi}$. This solution has continuous sample paths, $\mathbb{P}$-a.s..
	We are ready to define the map
	\begin{equation*}
	\begin{array}{cccc}
	\Phi  : & B_K & \to & \mathcal{P}(C([0,T],\R^D))\\
	& \pi & \mapsto & \mathcal{L}(X^{\pi}).
	\end{array}
	\end{equation*}
	Using standard stochastic calculus estimates, Assumption \ref{assumptions coefficients time dependent} and Gronwall's lemma we obtain the following upper bound on the second moment,
	\begin{equation}
		\| \sup_{t\in[0,T]} |X_t| \|_{L^2_{\omega}} \leq K/2 + TC(1+K),
	\end{equation}
	where $C>0$ is a generic constant that only depends on the coefficients $b$ and $\sigma$.
	Let us now take $T = 1/(4C)$, with this choice we have $\| \sup_{t\in[0,T]} |X_t| \|_{L^2_{\omega}} \leq K$. This implies $\Phi(B_K) \subset B_K$. Using again stochastic calculus and Gronwall's lemma, we have that for each $t\in [0,T]$,
	\begin{equation}
	\label{eq: lipschitz bound linear growth}
	W_{2,C([0,t],\R^D)}(\Phi(\pi),\Phi(\pi^{\prime}))^2
	\leq \| \sup_{s\in[0,t]} |X^{\pi}_s - X^{\pi^{\prime}}_s | \; \|_{L^2_{\omega}}
	\leq C(1+K) e^{C(1+K)t} \int_{0}^{t} W_{2,\R^D}(\pi_s,\pi^{\prime}_s)^2 \mathrm{d} s,
	\qquad
	\pi,\pi^{\prime} \in B_K.
	\end{equation}
	Here $C>0$ is again a generic constant depending only on $b$ and $\sigma$, possibly different than before. Also notice that $\pi, \pi^{\prime}$ are measures on the path space $C([0,T],\R^D)$ and in the left-hand side of \eqref{eq: lipschitz bound linear growth} we are considering, with an abuse of notations, their projections on $C([0,t],\R^D)$.
	
	Iterating $n$ times the inequality \eqref{eq: lipschitz bound linear growth}, we obtain
	\begin{equation*}
	W_{2,C([0,T],\R^D)}(\Phi(\pi), \Phi(\pi^\prime)) \leq \frac{e^{(n+1)C(1+K)T}}{n!}  W_{2,C([0,T],\R^D)}(\pi,\pi^{\prime}),
	\qquad
	\pi,\pi^{\prime} \in B_K.
	\end{equation*}
	For the choice of $T=1/(4C)$, if we take $n$ large enough, we have that the map $\Phi$ is a contraction on $B_K$. By the Banach fixed point theorem, $\Phi$ has a unique fixed point on $B_K$, which is the unique solution to equation \eqref{eq: mckean with linear growth} up to time $T$.
	
	Global existence and uniqueness follow by iterating this argument on intervals of fixed length $1/(4C)$, which does not depend on the value of the initial condition, but only on the assumptions on the coefficients of the equation.
\end{proof}

\subsection{The common noise case}
Consider measurable functions $b$ and $\sigma$ as in the previous section and $\beta: [0,T] \times \mathcal{P}_{\rho}(\R^D) \to \R^{D\times m_1}$, each of which satisfying the following assumption on their respective domain.

\begin{assumption}
	\label{assumptions coefficients common noise}
		 Let $(V, |\cdot|)$ be a Banach space, for $f: [0,T] \times \R^D \times \mathcal{P}_{\rho}(\R^D) \to V$ assume that there exists a constant $C > 0$ such that
		\begin{enumerate}[label=(\roman*), ref= \ref{assumptions coefficients common noise} (\roman*)]
			\item \label{assumpitons coefficients common noise: linear growth} 
			(linear growth)  $\forall x \in \R^D, \forall \pi \in \mathcal{P}_{\rho}(\R^d)$,
			\begin{equation}
			\label{eq: linear growth common noise}
			|f(t, x,\pi) | \leq C (1 + \overline{M}^{\rho}(\pi)  )^{\frac{1}{\rho}} ,
			\end{equation}
			\item  \label{assumpitons coefficients common noise: lipschitz}
			(locally Lipschitz) $\forall x,y \in \R^D$, $\forall \pi, \nu \in \mathcal{P}_{\rho}(\R^D)$, 
			\begin{equation}
			\label{eq: lipschitz condition common noise}
			| f(t, x, \pi) - f(t, y, \nu) |
			\leq C ( 1 + \overline{M}^{\rho}(\pi))^{\frac{1}{\rho}} \left(|x - y| + W_{\rho, \R^D}(\pi, \nu)\right).
			\end{equation}
		\end{enumerate}
\end{assumption}
Notice that, since $\beta$ is independent of the space variable $x \in \R^D$, condition \ref{assumpitons coefficients common noise: lipschitz} reduces to local Lipschitz continuity in the measure variable.

Consider the following McKean-Vlasov equation,
	\begin{equation}
	\label{eq: mckean with common noise}
	\mathrm{d} X_t = b(t, X_t, \mathcal{L}(X_t \mid \mathcal{B}_t)) \,\mathrm{d} t + \sigma(t, X_t, \mathcal{L}(X_t\mid \mathcal{B}_t)) \,\mathrm{d} W_t + \beta(t, \mathcal{L}(X_t \mid \mathcal{B}_t)) \,\mathrm{d}B_t,
	\qquad
	 X_t|_{t=0} = X_0 \in L^2_{\omega},
\end{equation}
where $W$ is the $m$-dimensional Brownian motion fixed at the beginning of the section and $B$ is an $m_1$-dimensional Brownian motion adapted to $(\mathcal{F}_t)_{t\geq 0}$. Assume that $X_0, W, B$ are independent.
In \eqref{eq: mckean with common noise}, $\mathcal{L}(X \mid \mathcal{B}_t)$ is the conditional law of the solution $X$ given the filtration $\mathcal{B}_t := \sigma(B_s \mid 0\leq s\leq t)$ generated by the common noise $B$.

\begin{lemma}
\label{lem: uniqueness mckean common noise}
Assume Assumption \ref{assumptions coefficients common noise}. Let $X$ and $Z$ be any two solutions to equation \eqref{eq: mckean with common noise}. Then $X$ and $Z$ are indistinguishable.
\end{lemma}
\begin{proof}
    It follows from the independence of $W$, $B$ and $X_0$ that
    \begin{equation*}
        \mathbb{E}[X_t \mid \mathcal{B}_t] = \mathbb{E}[X_0] + \int_0^t  \mathbb{E}[b(s,X_s,\mathcal{L}(X_s \mid \mathcal{B}_s)) \mid \mathcal{B}_s] \,\mathrm{d} s
        + \int_0^{t} \mathbb{E}[\beta(s,X_s,\mathcal{L}(X_s \mid \mathcal{B}_s)) \mid \mathcal{B}_s] \,\mathrm{d} B_s.
    \end{equation*}
 Let $\overline{X}_t := X_t - \mathbb{E}[X_t \mid \mathcal{B}_t]$, we have
 \begin{equation*}
        \overline{X}_t = \overline{X}_0 + \int_0^t \left(b(s,X_s,\mathcal{L}(X_s \mid \mathcal{B}_s)) - \mathbb{E}[b(s,X_s,\mathcal{L}(X_s \mid \mathcal{B}_s)) \mid \mathcal{B}_s] \right) \mathrm{d} s
         + \int_{0}^t \sigma (s, X_s, \mathcal{L}(X_s \mid \mathcal{B}_s)) \,\mathrm{d} W_s.
    \end{equation*}
    Set $N_t := \int_{0}^t \sigma(s,X_s,\mathcal{L}(X_s \mid \mathcal{B}_s)) \,\mathrm{d}W_s$, as an application of It\^o's formula we have
    \begin{align*}
        \mathbb{E}\left[\left|N_t\right|^2 \mid \mathcal{B}_t \right]
        = & 2 \mathbb{E}\left[\int_{0}^t N_s \cdot \sigma(s,X_s,\mathcal{L}(X_s \mid \mathcal{B}_s)) \mathrm{d}W_s \mid \mathcal{B}_t \right]
        + \int_{0}^t \mathbb{E}\left[ \operatorname{Trace}(\sigma\sigma^{\top})(s,X_s,\mathcal{L}(X_s \mid \mathcal{B}_s)) \mid \mathcal{B}_t \right] \mathrm{d}s \\
        \leq & C \int_{0}^{t}(1+ \overline{M}^2(\mathcal{L}(X_s \mid \mathcal{B}_s) ) \,\mathrm{d} s.
    \end{align*}
    By using Assumption \ref{assumpitons coefficients common noise: linear growth} also on the drift, we obtain
    \begin{align*}
        \overline{M}^2(\mathcal{L}(X_t \mid \mathcal{B}_t)) 
        = \mathbb{E}\left[\left |\overline{X}\right |^2 \mid \mathcal{B}_t\right]
        \leq C \int_{0}^{t}(1+ \overline{M}^2(\mathcal{L}(X_s \mid \mathcal{B}_s) ) \,\mathrm{d} s.
    \end{align*}
    Gronwall's lemma gives $\overline{M}^2(\mathcal{L}(X_t \mid \mathcal{B}_t))  \leq C_T$, where $C_T > 0$ is a constant depending on $T$. A similar bound can be obtained for $Z$.
    
    Combining the upper bounds on the central moment with Assumption \ref{assumptions coefficients common noise}, we have global Lipschitz continuity and boundedness of the coefficients $b,\sigma, \beta$. With standard estimates and Growall's lemma one can show that $\mathbb{E}[\sup_{t\leq T} |X_t - Z_t|^2] = 0$, which implies intistinguishability of the processes $(X_t)_{t\in[0,T]}$ and $(Z_t)_{t\in [0,T]}$.
\end{proof}

\begin{remark}
Weak existence follows from \cite{hammersley2021weak} and one can apply Yamada-Watanabe to obtain well-posedness of \eqref{eq: mckean with common noise}. However, well-posedness will also follow from our results on mixed rough and stochastic McKean-Vlasov equations in the special case $\beta(t,\pi) = P(\pi)$ defined in \eqref{def: P}
below.
\end{remark}

\subsection{McKean-Vlasov with continuous deterministic forcing}

Let $b: \mathcal{P}_{\rho}(\R^D) \to \R^{D}$ and $\sigma: \mathcal{P}_{\rho}(\R^D) \to \R^{D\times m}$ be measurable functions satisfying the following assumptions:
\begin{assumption}
	\label{assumptions coefficients}
		Assume that there exists a constant $C > 0$ such that
		\begin{enumerate}[label=(\roman*), ref= \ref{assumptions coefficients} (\roman*)]
			\item \label{assumpitons coefficients: linear growth} 
			(linear growth)  $\forall x \in \R^D, \forall \pi \in \mathcal{P}_{\rho}(\R^d)$,
			\begin{equation}
			\label{eq: linear growth}
			|b(x,\pi) | , | \sigma(x, \pi) | \leq C (1 + \overline{M}^{\rho}(\pi)  )^{\frac{1}{\rho}} ,
			\end{equation}
			\item  \label{assumpitons coefficients: lipschitz}
			(Lipschitz continuity) $\forall x,y \in \R^d$, $\forall \pi, \nu \in \mathcal{P}_{\rho}(\R^D)$, 
			\begin{equation}
			\label{eq: lipschitz condition}
			| b(x, \pi) - b(y, \nu) | , 
			| \sigma(x, \pi) - \sigma(y, \nu) | 
			\leq C ( 1 + \overline{M}^{\rho}(\pi))^{\frac{1}{\rho}} \left(|x - y| + W_{\rho,\mathbb{R}^D}(\pi, \nu)\right).
			\end{equation}
		\end{enumerate}
\end{assumption}
Let $F: [0,T] \to \mathbb{R}^D$ be a continuous bounded function, $X_0 \in L^\rho_{\omega}$ and consider the following stochastic differential equation,
\begin{equation}
\label{eq: sde with forcing}
\mathrm{d}X_t = b(X_t, \mathcal{L}(X_t)) \,\mathrm{d} t + \sigma(X_t, \mathcal{L}(X_t)) \,\mathrm{d} W_t + \mathrm{d} F_t,
\qquad
X_t|_{t=0} = X_0.
\end{equation}
\begin{definition}
\label{def: sol eq with forcing}
A stochastic process $(X_t)_{t\in [0,T]}$ on $\R^D$ is a solution for equation \eqref{eq: sde with forcing} with initial condition $X_0$ if $(\sigma(X_t, \mathcal{L}(X_t)))_{t\in [0,T]}$  is predictable and for every $t\in [0,T]$, the following integral equation is satisfied $\mathbb{P}$-a.s., 
\begin{equation*}
    X_t - F_t = \int_{0}^{t} b(X_s, \mathcal{L}(X_s))\, \mathrm{d} s
    + \int_{0}^{t} \sigma(X_s, \mathcal{L}(X_s))\, \mathrm{d} W_s,
    \qquad
    X_t|_{t=0} = X_0.
\end{equation*}
\end{definition}
\begin{remark}
In the following, we will construct an adapted solution $(X_t)_{t\in [0,T]}$ with continuous sample paths. Since $\sigma$ is Lipschitz continuous, we immediately have that $(\sigma(X_t, \mathcal{L}(X_t)))_{t\in [0,T]}$  is predictable and the It\^o integral is well defined.
\end{remark}

Let $X_t$ be a solution to equation \eqref{eq: sde with forcing}, we define $\overline{X}_t := X_t - \mathbb{E}[X_t]$. We start with some preliminary expansions and estimates for $\overline{X}$. We have, for $s,t \in [0,T]$,
	\begin{equation}
	\label{eq: X centered}
	\delta \overline{X}_{s,t} 
	= \int_{s}^{t} \left( b(X_r, \mathcal{L}(X_r) )- \mathbb{E}[b(X_r, \mathcal{L}(X_r))] \right) \mathrm{d} r
	+  \int_{s}^{t} \sigma(X_r, \mathcal{L}(X_r))  \,\mathrm{d} W_r.
	\end{equation}
	
\begin{lemma}
	\label{lem: bounds X centered}
Let $\rho \geq 2$ and $X_0 \in L^{\rho}_{{\omega}}$. There exists a constant $C_T$ such that $C_T \to 0$ as $T\to 0$ and

\begin{equation*}
	\| \sup_{t\in [0,T]}  |\overline{X}_t| \|_{L_{\omega}^\rho}
	\leq
	e^{C_T} [C_T + \|\overline{X}_0\|_{L_{\omega}^\rho}].
	\end{equation*}
	Moreover, for $\alpha < \frac{1}{2}$ and $\rho > \frac{2}{1-2\alpha}$, we have
	\begin{equation*}
	    \|[\overline{X}]_{\alpha}\|_{L^{\rho}_{\omega}}
	\leq 
	C_T ( 1 + \|\overline{X}_0\|_{L_{\omega}^{\rho}}).
	\end{equation*}
\end{lemma}	
\begin{proof}
 Using standard estimates and  the Burkholder-Davis-Gundy (BDG) inequality on equation \eqref{eq: X centered} we have
 \begin{equation*}
	\| \sup_{t\in [0,T]}  |\overline{X}_t| \|_{L_{\omega}^\rho}^{\rho}
	\leq
	 \|\overline{X}_0\|_{L_{\omega}^\rho}^{\rho} + C \int_{0}^{T}(1+\overline{M}^{\rho}(\pi_t)) \,\mathrm{d} t.
	\end{equation*}
	The first inequality follows from a standard application of Gronwall's lemma. For the second inequality we use again BDG and Jensen's inequality to obtain
	\begin{equation*}
	    \mathbb{E}[|\delta \overline{X}_{s,t}|^{\rho}]
	    \leq |t-s|^{\rho/2-\epsilon} C_T (1 + \| \overline{X}_0 \|_{L^{\rho}_{\omega}}).
	\end{equation*}
	The second inequality as well as the condition on $\rho$ follows from the Kolmogorov continuity theorem \cite[Theorem 3.1]{friz2020course}.
\end{proof}
	\begin{remark}
	\label{rmk: global bounds}
		The bounds in Lemma \ref{lem: bounds X centered} do not depend on the forcing term $F$. This gives us a good a-priori bound on the solution. Let $\overline{T}>0$ be fixed and arbitrarily large and assume that there exists $M_0 >0$ such that $\|\overline{X}_0\|_{L^\rho} \leq M_0$. Then there exists a global constant $M(M_0, \overline{T})$ such that
		\begin{equation}
		\label{eq: global bounds}
		   \| \sup_{t\in [0,T]}  |\overline{X}_t| \|_{L_{\omega}^\rho},
		   \;
		   \|[\overline{X}]_{\alpha}\|_{L^{\rho}_{\omega}} 
		   \leq
		   M.
		\end{equation}
	\end{remark}
We have the following a priori estimates.
\begin{lemma}
	\label{lem: a priori sde}
	Let $\rho \geq 2$. Given $X_0 \in L^{\rho}_{\omega}$ and $F \in C_b([0,T], \R^D)$, we call $X(F,X_0)$ a solution to equation \eqref{eq: sde with forcing} with forcing $F$ and initial condition $X_0$.
	 
	Let $b, \sigma$ satisfy Assumption \ref{assumptions coefficients}.
	Then there exists a constant $C_T:= C(T,\rho) >0$ such that $C_T \to 0$ as $T\to 0$ and
	\begin{equation}
	\label{eq: Lp bound sde}
	\| \sup_{t\in [0,T]}  |X_t(F,X_0)| \|_{L_{\omega}^\rho}
	\leq
	 e^{C_T} [C_T + \| F \|_{C_b} + \|X_0\|_{L_{\omega}^\rho}].
	\end{equation}
	Given $X_0, Y_0 \in L^{\rho}_{\omega}$ and $F,G \in C_b([0,T], \R^D)$, there exists a positive constant $C>0$ such that
	\begin{equation}
	\label{eq: Lp lipschitz sde}
	\| \sup_{t\in [0,T]} |X_t(F,X_0) - X_t(G,Y_0)|  \|_{L_{\omega}^{\rho}}
	\leq
	e^{C (1 + \|\overline{X}_0\|_{L_{\omega}^\rho}) }
	\left(
	\| F - G \|_{C_b} + \|X_0 - Y_0\|_{L_{\omega}^\rho}
	\right).
	\end{equation}
\end{lemma}

\begin{proof}
	We write $X$ and omit here the dependence of the process on $F$ and $X_0$ as there is no possibility of confusion in the first part of the proof.
	Using Jensen's inequality, the Burkholder-Davis-Gundy inequality and Assumptions \ref{assumptions coefficients} we have the following estimate for $s,t \in [0,T]$, 
	\begin{align*}
	\mathbb{E} \sup_{t\in[0,T]} | X_{t} |^{\rho}
	\leq & \mathbb{E}|X_0|^{\rho} 
	+ C T^{\rho-1} \int_0^T \mathbb{E}|b(X_r, \mathcal{L}(X_r))|^{\rho} \,\mathrm{d} r 
	+ CT^{\rho/2-1} \int_0^T \mathbb{E}|\sigma(X_r, \mathcal{L}(X_r))|^{\rho} \,\mathrm{d} r 
	+ C\sup_{t\in[0,T]}|F_{t}|^{\rho}\\
	\leq &\mathbb{E}|X_0|^{\rho} 
	+ C\left(
	T^{\rho-1} + T^{\rho/2-1}
	\right)
	\int_0^T (1+ \mathbb{E}|\overline{X}_{r}|^{\rho}) \,\mathrm{d} r
	+ C\sup_{t\in[0,T]}|F_{T}|^{\rho},
	\end{align*}
	where the constant $C$ depends only on $\rho, b$ and $\sigma$.
	Equation \eqref{eq: Lp bound sde} follows immediately using the bounds in Lemma \ref{lem: bounds X centered}.
	
	We now prove  inequality \eqref{eq: Lp lipschitz sde}. Applying a similar reasoning as before, we obtain
	\begin{align*}
	\mathbb{E}&\sup_{t\in [0,T]} |X_t(F,X_0) - X_t(G,Y_0)|^{\rho}
	\leq C  (T^\rho + T^{\rho/2}) \int_{0}^{T} (1 + \mathbb{E}|\overline{X}_t(F,X_0)|^{\rho})\\
	& \cdot
	\left(
		\mathbb{E}\sup_{r\in [0,t]} |X_r(F,X_0) - X_r(G,Y_0)|^{\rho}
		+ W_{\rho,\R^D}(\mathcal{L}(X_t(F,X_0)), \mathcal{L}(X_t(G,Y_0)))^{\rho}
	\right)
	\mathrm{d} t
	+ \| F-G\|_{C_b}.
	\end{align*}
	The $\rho$-Wasserstein distance is controlled by the $L^{\rho}_{\omega}$ norm of the difference of the processes, so we can apply Gronwall's lemma and Lemma \ref{lem: bounds X centered} to obtain the desired inequality \eqref{eq: Lp lipschitz sde}.
\end{proof}
In addition the the previous bounds, if the forcing term is $\alpha$-H\"older continuous, we have that also the solution is $\alpha$-H\"older continuous.
\begin{lemma}
	\label{lem: a priori sde heolder}
	Let $\alpha  \in (\frac{1}{3}, \frac{1}{2})$ and $\rho > 2/(1-2\alpha)$. Assume that $F,G \in C^{\alpha}([0,T], \R^D)$ and $X_0, Y_0 \in L_{\omega}^{\rho}$. Then there exist $C_T, C:= C(\rho, T, \alpha)>0$ such that $C_T \to 0$ as $T\to 0$ and
	\begin{equation} 
	\label{eq: hoelder bound sde}
	\|[X(F,X_0)]_{\alpha}\|_{L^{\rho}_{\omega}}  
	\leq 
	C_T (
	1 + \|\overline{X}_0\|_{L_{\omega}^{\rho}} ) + C [F]_{\alpha}.
	\end{equation}
	Moreover, 
	\begin{align}
	\label{eq: contraction hoelder bound}
	|[X(F,X_0) - X(G,Y_0)]_{\alpha}\|_{L^{\rho}_{\omega}}  
	\leq & C_T
	e^{C (1  + \|\overline{X}_0\|_{L_{\omega}^\rho}) }
	\left(
	\| F-G \|_{C_b} + \|X_0 - Y_0\|_{L_{\omega}^\rho}
	\right)
	+ C [F-G]_{\alpha}.
	\end{align}
\end{lemma}
\begin{proof}
	We write $X$ and omit here the dependence of the process on $F,X_0$ as there is no possibility of confusion in the first part of the proof.
	Using Jensen's inequality, the Burkholder-Davis-Gundy inequality and Assumption \ref{assumptions coefficients} we have the following estimate for $s,t \in [0,T]$, 
	\begin{align*}
	\mathbb{E}[ |\delta (X-F)_{s,t} |^{\rho}] 
	\leq & C |t-s|^{\rho-1} \int_s^t \mathbb{E}|b(X_r, \mathcal{L}(X_r))|^{\rho} \,\mathrm{d} r 
	+ C|t-s|^{\rho/2-1} \int_s^t \mathbb{E}|\sigma(X_r, \mathcal{L}(X_r))|^{\rho} \,\mathrm{d} r \\
	\leq & C\left(
	|t-s|^{\rho-1} + |t-s|^{\rho/2-1}
	\right)
	\int_s^t (1 + \mathbb{E}|\overline{X}_{r}|^{\rho}) \,\mathrm{d} r\\
	\leq & C\left(
	|t-s|^{\rho-1} + |t-s|^{\rho/2-1}
	\right)|t-s| (1 + \mathbb{E}|\sup_{r\in[0,T]}\overline{X}_{r}|^{\rho})\\
	\leq & C |t-s|^{\rho/2-\epsilon} \left[C_T (1+ \mathbb{E}[\sup_{r\in[0,T]}|\overline{X}_{r}|^{\rho}]) \right].
	\end{align*}
	The constant $C_T$ is such that $C_T \to 0$, as $T\to 0$. We apply Lemma \ref{lem: bounds X centered} to obtain the following estimate,
	\begin{align*}
	\mathbb{E}[ |\delta (X-F)_{s,t} |^{\rho}] 
	\leq & 	C_T|t-s|^{\rho/2 - \epsilon}
	\left( 
	1 + \|\overline{X}_0\|_{L_{\omega}^{\rho}} 
	\right).
	\end{align*}
	Using the Kolmogorov continuity theorem we obtain that $X-F$ is $\alpha$-Hölder continuous for $\rho > 2/(1-2\alpha)$ and
	\begin{equation*}
	    \|[X]_{\alpha}\|_{L^{\rho}_{\omega}} \leq \|[X-F]_{\alpha}\|_{L^{\rho}_{\omega}} + C [F]_{\alpha}
	    \leq 
	    C_T (
	    1 + \|\overline{X}_0\|_{L_{\omega}^{\rho}} ) + [F]_{\alpha}.
	\end{equation*}
	
	We now prove  inequality \eqref{eq: contraction hoelder bound}. Arguing as in the first half of the proof we obtain
	\begin{align*}
	\mathbb{E}&|\delta (X(F,X_0) - X(G,Y_0)- (F-G))_{s,t}|^{\rho} 
	\leq C
	\left(
	|t-s|^{\rho-1}  + |t-s|^{\rho/2-1}
	\right) \\
	&\cdot
	\int_{s}^{t} (1+E|\overline{X}_r|^{\rho})\left(\mathbb{E}|X_r(F,X_0) - X_r(G,Y_0))|^{\rho} + W_{\rho}(\mathcal{L}(X_r(F,X_0)), \mathcal{L}(X_r(G,Y_0))^{\rho} \right)
	\mathrm{d}r
	\end{align*}
	We apply Lemma \ref{lem: bounds X centered} and \eqref{eq: Lp lipschitz sde} and use the fact that the $L^{\rho}$-norm of the difference of the process controls the $\rho$-Wasserstein distance, to obtain
	\begin{align*}
	\mathbb{E}&|\delta (X(F,X_0) - X(G,Y_0)-F+G)_{s,t}|^{\rho} 
	\leq C_T |t-s|^{\alpha\rho}
	e^{C_T (1 + \|\overline{X}_0\|_{L_{\omega}^\rho}) }
	\left(
	\| F-G \|_{C_b} 
	+ \|X_0 - Y_0\|_{L_{\omega}^\rho} \right)^{\rho}.
	\end{align*}
	Equation \eqref{eq: contraction hoelder bound} follows as before from the Kolmogorov continuity theorem, for $\rho > 2/(1-2\alpha)$.	
\end{proof}

\begin{lemma}
	\label{lem: sdes with forcing}
	Let $F \in C_b([0,T], \R^D)$ and $X_0 \in L_{\omega}^2$. Let $b$ and $\sigma$ satisfy Assumption \ref{assumptions coefficients}. 
	Then equation \eqref{eq: sde with forcing} admits a unique strong solution.
\end{lemma}
\begin{proof}
	To prove existence we define
	\begin{equation*}
	\tilde b(t, x, \mu) := b(x + F_t, (\tau_{F_t})_{\#}\mu), \qquad t\in[0,T], x \in {\R}^D, \mu \in \mathcal{P}({\R}^D),
	\end{equation*}
	where $\tau_{z} : \R^D \to \R^D$ is the translation by $z \in \R^D$. Similarly we define $\tilde \sigma$. It is immediate to see that, if $b,\sigma$ satisfy Assumption \eqref{assumptions coefficients}, then $\tilde{b}, \tilde{\sigma}$ satisfy Assumption \eqref{assumptions coefficients time dependent}.
	
	By Lemma \ref{lem: well-posedness mckean-vlasov linear growth} the following equation admits a unique global strong solution $\tilde X$,
		\begin{equation*}
		\mathrm{d}\tilde X_t = \tilde b(t, \tilde X_t, \mathcal{L}(\tilde X_t)) \,\mathrm{d} t + \tilde \sigma(t, \tilde X_t, \mathcal{L}(\tilde X_t)) \, \mathrm{d} W_t,
		\qquad
		\tilde X_t|_{t=0} = X_0 - F_0.
		\end{equation*}
	The stochastic process $X = \tilde X + F$ solves equation \eqref{eq: sde with forcing}.
	Uniqueness follows from the a priori estimates given in Lemma \ref{lem: a priori sde}.
\end{proof}

\subsection{Rough McKean-Vlasov}
\label{sec: rough McKean-Vlasov}

\begin{assumption}
	\label{assumpion: h}
	Let $h\in C^2_b(\R^D,\R^d)$ such that $x\mapsto xh^T(x) \in C^2_b(\R^D, \R^{D\times d})$.
\end{assumption}
For a given probability measure $\pi \in \mathcal{P}(\R^D)$, we define 
\begin{equation}
\label{def: P}
P(\pi) := \operatorname{Cov}_{\pi}(x,h) A + B,
\qquad
\operatorname{Cov}_{\pi}(x,h)^{l,j} =\int_{\R^D} (x^{l}-\pi[x^{l}]) h^{j}(x) \pi(\mathrm{d} x),
\qquad
1\leq l \leq D, 
\;
1\leq j \leq d,
\end{equation}
where $A\in \R^{d\times d}$ and $B\in \R^{D\times d}$ are given matrices.

Let $\bY:= (Y,\mathbb{Y}) \in \mathscr{C}^{\alpha}(\R^d)$ and $X_0 \in L_{\omega}^2$. We study the following equation,
\begin{subequations}
\label{eq: rough McKean}
\begin{align}
\mathrm{d}X_t & = b(X_t, \mathcal{L}(X_t))\,\mathrm{d}t + \sigma(X_t, \mathcal{L}(X_t)) \,\mathrm{d} W_t + \mathrm{d} F_t,
\qquad
X_t|_{t=0} = X_0,\label{eq: rough McKean 1}\\
\mathrm{d} F_t & = P(\mathcal{L}(X_t)) \,\mathrm{d} \bY_t.
\label{eq: rough McKean 2}
\end{align}
\end{subequations}
\begin{definition}
\label{def: solution}
A couple $(X,F): [0,T] \times \Omega \to \R^D \times \R^D$ is a solution to equation \eqref{eq: rough McKean} if
\begin{itemize}
    \item $F$ is a continuous path and $X$ solves equation \eqref{eq: rough McKean 1} in the sense of Definition \ref{def: sol eq with forcing};
    \item $P(\mathcal{L}(X_t)) \in \mathscr{D}_{\bY}^{2\alpha}([0,T],\mathscr{L}({\R}^{d},{\R}^{D}))$ and $F$ is the rough integral in \eqref{eq: rough McKean 2}. 
\end{itemize}
\end{definition}
Next we prove that, if $X$ solves \eqref{eq: rough McKean 1} for a fixed controlled $F$, then $P(\mathcal{L}(X_t))$ is a controlled path, which makes the rough integral \eqref{eq: rough McKean 2} well defined.

\begin{lemma}
\label{lem: P controlled path}
    Let $\alpha  \in (\frac{1}{3}, \frac{1}{2})$ and $\rho > 2/(1-2\alpha)$.
	Let $\bY:= (Y,\mathbb{Y}) \in \mathscr{C}^{\alpha}(\R^d)$, $F\in \mathscr{D}_{\bY}^{2\alpha}([0,T],\R^D)$ and $X_0 \in L_{\omega}^{\rho}(\R^D)$.
	
	If $\pi_t$ is the law of the solution process $X_t(F,X_0)$ to equation \eqref{eq: sde with forcing} with forcing $F$ and initial condition $X_0$, then 
	$$P(\pi_{\cdot}) \in \mathscr{D}_{\bY}^{2\alpha}([0,T],\mathscr{L}({\R}^{d},{\R}^{D})),$$
	with Gubinelli derivative
	\begin{equation}
	\label{eq: Gubinelli derivative of P}
	    P(\pi_{s})^{\prime}
	    = (F^{\prime}_s)^{\top}\pi[(x-\pi[x]) D h^{\top}] A .
	\end{equation}
	Moreover, we have the bound
	\begin{equation*}
	    \| P(\pi_{\cdot})\|_{\mathscr{D}^{2\alpha}_{\bY}}
	\leq  C \|h\|_{C^2_b}
	(1 + \| \overline{X}_0\|_{L^{\rho}_{\omega}} + [F]_{\alpha}) \left(C_T( 1 + \|\overline{X}_0\|_{L^{\rho}_{\omega}}) + \| F^{\prime} \|_{C_b} [Y]_{\alpha}
	+ [R^F]_{2\alpha}
	\right),
	\end{equation*}
	where $C_T \to 0$ as $T\to0$.
\end{lemma}

\begin{proof}
	To simplify the notation, we write $X_t$ for $X(F,X_0)$. Let $\overline{X}_t := X_t - \mathbb{E}[X_t]$. 
		We have the following expansion for $h(X)$, $s,t \in [0,T]$,
	\begin{align}
	\label{eq: expansion h}
	\delta h(X_{\cdot})_{s,t}
	= & [Dh]_{s,t}^{1,x} \left(
	\int_{s}^{t}  b(X_r, \mathcal{L}(X_r) ) \,\mathrm{d} r
	+R^{F}_{s,t}
	\right)\\
	&+ ([D^2h]_{s,t}^{2,x} \delta X_{s,t} + Dh(X_s) ) \left(
	\int_{s}^{t}  \sigma(X_r, \mathcal{L}(X_r) ) \,\mathrm{d} W_r
	+ F^{\prime}_{s}\delta Y_{s,t}
	\right).\nonumber
	\end{align}
	Let $s,t\in[0,T]$. From the definition of $P(\pi)$, equation \eqref{def: P}, we have
	\begin{equation*}
	\delta P(\pi_{\cdot})_{s,t} 
	= \delta \mathbb{E}[\overline{X}_{\cdot}h(X_{\cdot})^{\top}]_{s,t} A
	= \mathbb{E}[\delta \overline{X}_{s,t} h(X_s)^{\top}]A
	+\mathbb{E}[\delta \overline{X}_{s,t} \delta h(X_{\cdot})_{s,t}^{\top}]A
	+\mathbb{E}[\overline{X}_{s} \delta h(X_{\cdot})_{s,t}^{\top}]A
	=:I_1+I_2+I_3.
	\end{equation*}
	We expand $I_3$ even further using \eqref{eq: expansion h}
	\begin{align*}
	I_3 
	= & \mathbb{E}[\overline{X}_s (Dh(X_s)F^{\prime}_{s}\delta Y_{s,t})^{\top}] A
	+ \mathbb{E}\left[\overline{X}_s \left(Dh(X_s)	\int_{s}^{t}  \sigma(X_r, \mathcal{L}(X_r)) \mathrm{d} W_r\right)^{\top}\right]A \\
	& + \mathbb{E}\left[ \overline{X}_s \left(
	[Dh]_{s,t}^{1,x} \left(
	\int_{s}^{t}  b(X_r, \mathcal{L}(X_r) ) \mathrm{d} r
	+ R^{F}_{s,t}
	\right)
	+ [D^2h]_{s,t}^{2,x} \delta X_{s,t} \left(
	\int_{s}^{t}  \sigma(X_r, \mathcal{L}(X_r)) \mathrm{d} W_r
	+ F^{\prime}_{s}\delta Y_{s,t}
	\right)
	\right)^{\top}
	\right] A \\
	=&: \mathbb{E} \left[ \overline{X}_s (Dh(X_s)F^{\prime}_{s}\delta Y_{s,t})^{\top} \right]A + I_4.
	\end{align*}	
	We can write
	\begin{equation*}
	\delta P(\pi_{\cdot})_{s,t} = \mathbb{E}[\overline{X}_s (Dh(X_s)F^{\prime}_{s}\delta Y_{s,t})^T] A + R^{P}_{s,t},
	\end{equation*}
	where $	R^{P}_{s,t} = I_1 + I_2 + I_4$.
	For $1 \leq l \leq D$ and $1 \leq j \leq d$ we write the $(l, j)$ entry of the matrix $P$ as
	\begin{equation*}
	    \delta P^{l,j}(\pi_{\cdot})_{s,t} =
	    (P^{l,j}(\pi_{\cdot}))^{\prime}_{s} \cdot \delta Y_{s,t} + R^{P^{l,j}}_{s,t},
	\end{equation*}
	where, for $s\in[0,T]$, $P(\pi_s)^{\prime} \in \mathscr{L}({\R}^{d}, \mathscr{L}({\R}^{d},{\R}^{D})) \cong \R^{D\times d\times d}$ is the Gubinelli derivative given as
	\begin{equation*}
	    (P^{l,j}(\pi_{s}))^{\prime}
	    := \sum_{k = 1}^{D} \mathbb{E}[\overline{X}^{l}_s (Dh(X_s)^{\top}A)^{j}](F^{\prime}_s)^{k} 
	    = \sum_{k = 1}^{D}\sum_{i=1}^{d} \mathbb{E}[\overline{X}^{l}_s \partial_k h^i(X_s)]A^{i,j}(F^{\prime}_s)^{k}\in {\R}^{d}.
	\end{equation*}
	
	We check now that $R^P$ is a remainder with regularity $2\alpha$.
	Using the martingale property of the stochastic integral in \eqref{eq: X centered}, we have
	\begin{equation*}
	I_1 
	= \mathbb{E}[\mathbb{E}_s[\delta \overline{X}_{s,t}] h(X_s)^T]A
	= \mathbb{E}\left[\int_{s}^{t} \left( b(X_r, \mathcal{L}(X_r))- \mathbb{E}[b(X_r, \mathcal{L}(X_r))] \right) \mathrm{d} r \; h(X_s)^T\right] A.
	\end{equation*}
	We apply Lemma \ref{lem: bounds X centered} to obtain
	\begin{equation*}
	| I_1 |\leq |t-s| C \| h \|_{\infty} (C_T + \| X_0\|_{L^{\rho}_{\omega}}),
	\end{equation*}
	where $C>0$ is a global constant.
	For $I_2$ we use Lemma \ref{lem: bounds X centered} and \eqref{eq: hoelder bound sde} to obtain, for any $\rho > 3$,
	\begin{align*}
	|I_2 |
	= &| \mathbb{E}[\delta \overline{X}_{s,t} ([Dh]_{s,t}^{1,x} \delta X_{s,t} )^T] |
	\leq \| Dh \|_{\infty} \|[\overline{X}]_{\alpha}\|_{L^{\rho}_{\omega}}  \|[X]_{\alpha}\|_{L^{\rho}_{\omega}} 
	\\
	\leq & \| Dh \|_{\infty} 
	C_T(1+ \|\overline{X}_0\|_{L^{\rho}_{\omega}})
	[C_T(1 + \|\overline{X}_0\|_{L^{\rho}_{\omega}}) + [F]_{\alpha}] |t-s|^{2\alpha}.
	\end{align*}
	We proceed by finding a bound for $I_4$. Notice that the term $E[\overline{X}_s \left(Dh(X_s)	\int_{s}^{t}  \sigma(X_r, \mathcal{L}(X_r) \mathrm{d} W_r\right)^T]$ vanishes thanks to the martingale property of the stochastic integral.
	Using Lemma \ref{lem: bounds X centered} and inequality \eqref{eq: hoelder bound sde}, we get
	\begin{align*}
	I_4 \leq & C |t-s|^{2\alpha} 
	\|Dh\|_{\infty}
	\| \sup_{t\in[0,T]}\overline{X}_t\|_{L^{\rho}_{\omega}} 
	\left(
	C_T (1+ \| \sup_{t\in[0,T]}\overline{X}_t\|_{L^{\rho}_{\omega}} )
	+ [R^F]_{2\alpha}
	\right)\\
	& + C |t-s|^{2\alpha} \|D^2h\|_{\infty} \|[X]_{\alpha} \|_{L^{\rho}_{\omega}} \left(C_T( 1 + \|\sup_{t\in[0,T]}\overline{X}_t\|_{L^{\rho}_{\omega}}) + \| F^{\prime} \|_{C_b} [Y]_{\alpha}
	\right)
    \\
    \leq & C |t-s|^{2\alpha} 
	\|Dh\|_{\infty}
	(C_T + 
	\| \overline{X}_0\|_{L^{\rho}_{\omega}} )
	\left(
	C_T (1+ \| \overline{X}_0\|_{L^{\rho}_{\omega}} )
	+ [R^F]_{2\alpha}
	\right)\\
	& + C |t-s|^{2\alpha} \|D^2h\|_{\infty} (1 + \| \overline{X}_0\|_{L^{\rho}_{\omega}} + [F]_{\alpha}) \left(C_T( 1 + \|\overline{X}_0\|_{L^{\rho}_{\omega}}) + \| F^{\prime} \|_{C_b} [Y]_{\alpha}
	\right)
    \\
	\leq & C |t-s|^{2\alpha} (\|Dh\|_{\infty} + \|D^2h\|_{\infty} )
	(1 + \| \overline{X}_0\|_{L^{\rho}_{\omega}} + [F]_{\alpha}) \left(C_T( 1 + \|\overline{X}_0\|_{L^{\rho}_{\omega}}) + \| F^{\prime} \|_{C_b} [Y]_{\alpha}
	+ [R^F]_{2\alpha}
	\right).
	\end{align*}
\end{proof}

We now set up a contraction argument that we will use to prove the well-posedness of equation \eqref{eq: rough McKean}. Let $f \in \mathscr{D}^{2\alpha}_{\bY}([0,T],\mathcal{L}(\R^d,\R^D))$. We define, for $t \in [0,T]$,
\begin{equation}
\label{eq: F from f}
    (F_{t}, F^{\prime}_{t}) := \left(\int_{0}^{t} f_{r} \mathrm{d} \bY_{r}, f_t\right)
    \; 
    \in \mathscr{D}^{2\alpha}_{\bY}([0,T],{\R}^D).
\end{equation}
Now we plug this as the forcing into equation \eqref{eq: sde with forcing} with initial condition $X_0 \in L_{\omega}^{\rho}({\R}^D)$ and call $X$ the solution with law $\pi$. We define the map
\begin{equation*}
    \begin{array}{cccc}
       \Gamma:  &  \mathscr{D}^{2\alpha}_{\bY}([0,T],\mathcal{L}(\R^d,\R^D))
       & \to &
       \mathscr{D}^{2\alpha}_{\bY}([0,T],\mathcal{L}(\R^d,\R^D)) \\
         & f
         & \mapsto &
         P(\pi_{\cdot}).
    \end{array}
\end{equation*}
From \cite{friz2020course} we have the following estimates,
\begin{equation*}
    \|f\|_{C_b}, \; \|f^{\prime}\|_{C^b}
    \leq
    |f_0| + \| f \|_{\mathscr{D}^{2\alpha}_{\bY}},
\end{equation*}
\begin{equation*}
    [F]_{\alpha},\; [R^{F}]_{\alpha}
    \leq C \| \bY \|_{\alpha}
    (|f_0| + \| f \|_{\mathscr{D}^{2\alpha}_{\bY}}).
\end{equation*}
It follows from Lemma \ref{lem: P controlled path} that $\Gamma$ is well defined. 
Let us prove that there exists a small time, such that the following set is invariant under $\Gamma$,
\begin{equation*}
    D(\Gamma) := \{ f\in \mathscr{D}^{2\alpha}_{\bY} 
    \mid
    |f_0| + \| f \|_{\mathscr{D}^{2\alpha}_{\bY}}
    \leq
    2\|h\|_{\infty} M
    \},
\end{equation*}
where $M$ is the global constant of Remark \ref{rmk: global bounds}.

We have the following estimates,
\begin{equation*}
|\Gamma(f)_0|
=
|\mathbb{E}[\overline{X}_0 h(X_0)]A + B|
\leq \|h\|_{\infty} M,
\end{equation*}
	\begin{align*}
	   \| \Gamma(f) \|_{\mathscr{D}^{2\alpha}_{\bY}} =
	    \| P(\pi_{\cdot})\|_{\mathscr{D}^{2\alpha}_{\bY}}
	\leq  & C \|h\|_{C^2_b}
	(1 + \| \overline{X}_0\|_{L^{\rho}_{\omega}} + \| \bY \|_{\alpha}
    (|f_0| + \| f \|_{\mathscr{D}^{2\alpha}_{\bY}})) \left(C_T( 1 + \|\overline{X}_0\|_{L^{\rho}_{\omega}}) + \| \bY \|_{\alpha}
    (|f_0| + \| f \|_{\mathscr{D}^{2\alpha}_{\bY}})
	\right)\\
	\leq & \|h\|_{C^2_b}
	C_T (1+M + 2 \|h\|_{\infty}
	(\|\bY\|_{\alpha} \vee \|\bY\|_{\bar{\alpha}}))^2\\
	\leq & \|h\|_{\infty} M,
	\end{align*}
where in the last step we chose $T$ small enough such that
\begin{equation*}
    C_T \leq 
    \frac{\|h\|_{\infty}M}{\|h\|_{C^2_b}
	(1+M + 2 M\|h\|_{\infty}
	(\|\bY\|_{\alpha} \vee \|\bY\|_{\bar{\alpha}}))^2}.
\end{equation*}
Adding the previous estimates we proved that $\Gamma(D(\Gamma)) \subset D(\Gamma)$.
Since $C_T$ only depends on global quantities, we can divide each interval $[0,\overline{T}]$ into smaller intervals of length $T$ and in each one apply the contraction argument.
We prove in the following that $\Gamma$ is a contraction on $B_T$. We start by showing that $\Gamma$ is Lipschitz continuous.

\begin{lemma}
\label{lem: contractive estimates}
Let $1/3 < \alpha < \bar{\alpha} < 1/2$. Let $\rho > 2/(1-2\alpha)$ and $X_0, Z_0 \in L_{\omega}^{\rho}$.  Let $\bY, \hat{\bY} \in \mathscr{C}^{\bar{\alpha}}$ and $F \in \mathscr{D}^{2\alpha}_{\bY}$, $G \in \mathscr{D}^{2\alpha}_{\hat{\bY}}$. Assume that there exists a universal constant $C>0$ such that
\begin{equation*}
\| F \|_{\mathscr{D}^{2\alpha}_{\bY}} \leq C \| \bY \|_{\alpha},
    \qquad
\| G \|_{\mathscr{D}^{2\alpha}_{\hat{\bY}}} \leq C \| \overline{\bY} \|_{\alpha},
\qquad
\| \overline{X}_0 \|_{L^{\rho}_{\omega}},
\;
\| \overline{Z}_0 \|_{L^{\rho}_{\omega}} \leq C.
\end{equation*}
We call $X_t(F,X_0)$ (resp. $X_t(G,Z_0)$) the solution to equation \eqref{eq: sde with forcing} with inputs $F$ and $X_0$ (resp. $G$ and $Z_0$).
We call $R$ the difference of the remainders of $P(\pi^x)$ and $P(\pi^z)$.  There exists a global constant $C_T>0$ such that $C_T\to 0$ as $T\to 0$ and
\begin{equation}
    [R]_{2\alpha}
    \leq
    C_T
    \left(
	\| F-G \|_{C_b} + \|X_0 - Z_0\|_{L_{\omega}^\rho}
	+ [F-G]_{\alpha}
	    \right)
    +
    \|F^{\prime} - G^{\prime}\|_{C_b} [\bY]_{\alpha}
    + \|G^{\prime} \|_{C_b}
     [\bY - \hat{\bY}]_{\alpha}.
\end{equation}
A similar estimate can be obtained for the difference of the Gubinelli derivatives.
\end{lemma}
\begin{proof}
    To simplify the notation, we write $X_t$ for $X(F,X_0)_t$ and $Z_t$ for $X(G,Z_0)_t$. Let $\overline{X}_t := X_t - E[X_t]$, similarly we define $\overline{Z}_t$.
    We have the following expansion for $h(X)-h(Z)$, $s,t \in [0,T]$, 
    $\mathbb{P}$-a.s.,
    \begin{subequations}
    \label{eq: expansion difference h}
	\begin{align}
	\delta (h(X_{\cdot}) - h(Z_{\cdot}))_{s,t} 
	= & Dh(X_s) F^{\prime} \delta Y_{s,t}
	- Dh(X_s) G^{\prime} \delta \hat{Y}_{s,t}\label{eq: expansion difference h: gubinelli derivative}\\
	& + Dh(X_s)
	\left(
	\int_{s}^{t}  \sigma(X_r, \mathcal{L}(X_r) ) \,\mathrm{d} W_r
	\right)
	 - Dh(Z_s)
	 \left(
	\int_{s}^{t}  \sigma(Z_r, \mathcal{L}(Z_r) ) \,\mathrm{d} W_r
	\right)
	\label{eq: expansion difference h: martingale}
	\\
	& + [Dh]_{s,t}^{1,z} \left(
	\int_{s}^{t}  
	\left[
	b(X_r, \mathcal{L}(X_r) ) -  b(Z_r, \mathcal{L}(Z_r) )
	\right] \mathrm{d} r
	+ R^{F}_{s,t} - R^{G}_{s,t}
	\right)\label{eq: expansion difference h: reminder first}\\
	&+ \left(
	[Dh]_{s,t}^{1,x} - [Dh]_{s,t}^{1,z}
	\right)
	\left(
	\int_{s}^{t}  
	b(X_r, \mathcal{L}(X_r) ) \,\mathrm{d} r
	+ R^{F}_{s,t}
	\right)\\
	& + [D^2h]_{s,t}^{2,z} \delta Z_{s,t} \left(
	\int_{s}^{t}  \left[\sigma(X_r, \mathcal{L}(X_r) ) 
	- \sigma(Z_r, \mathcal{L}(Z_r) )
	\right]\mathrm{d} W_r
	\right)\\
	& + [D^2h]_{s,t}^{2,z} \delta Z_{s,t} \left(F^{\prime}_{s}\delta Y_{s,t} - G^{\prime}_{s}\delta \hat{Y}_{s,t} \right)\\
	&+\left(
	[D^2h]_{s,t}^{2,z}\delta(X-Z)_{s,t}
	+([D^2h]_{s,t}^{2,x} - [D^2h]_{s,t}^{2,z})\delta X_{s,t}
	\right)
	\left(
	\int_{s}^{t}  \sigma(X_r, \mathcal{L}(X_r) ) \,\mathrm{d} W_r
	\right).
	\label{eq: expansion difference h: reminder last}
	\end{align}
	\end{subequations}
	Let $s,t\in[0,T]$. We call $\pi^x := \mathcal{L}(X)$ and $\pi^z := \mathcal{L}(Z)$. From the definition of $P(\pi)$, equation \eqref{def: P}, we have
	\begin{align*}
	\delta (P(\pi^x_{\cdot})-P(\pi^z_{\cdot}))_{s,t} 
	= & \mathbb{E}[\delta (\overline{X}_{\cdot}-\overline{Z}_{\cdot})_{s,t}  h(Z_s)^T]A
	+\mathbb{E}[\delta (\overline{X}_{\cdot}-\overline{Z}_{\cdot})_{s,t}\delta h(Z_{\cdot})_{s,t}^T]A
	+\mathbb{E}[(\overline{X}_{s} - \overline{Z}_{s}) \delta h(Z_{\cdot})_{s,t}^T]A\\
	& + \mathbb{E}[\delta \overline{X}_{s,t} 
	(h(X_s)-h(Z_s))]A
	+\mathbb{E}[\delta \overline{X}_{s,t}
	\delta (h(X_{\cdot})-h(Z_{\cdot}))_{s,t}]A
	+\mathbb{E}[\overline{X}_{s}
	\delta (h(X_{\cdot})-h(Z_{\cdot}))_{s,t}] A\\
	=: & I_{11}+I_{12}+I_{13}
	 + I_{21}+I_{22}+I_{23}.
	\end{align*}
Using the same expansion for $h(Z)$ as in \eqref{eq: expansion h}, we obtain
\begin{align*}
    I_{13} = & \mathbb{E}[(\overline{X}_s - \overline{Z}_s)
    (Dh(Z_s)G^{\prime}\delta \hat{Y}_{s,t})^{T}]A
    + 0
    + \mathbb{E}[(\overline{X}_s - \overline{Z}_s)
    ([Dh]_{s,t}^{1,z}
    (\int_{s}^{t} b(Z_r, \mathcal{L}(Z_r)) dr 
    + R^{G}_{s,t}
    ))^T
    ] A\\
    & + \mathbb{E}[(\overline{X}_s - \overline{Z}_s) 
    ([D^2]_{s,t}^{2,z} \delta Z_{s,t}
    (\int_{s}^{t} \sigma(Z_r, \mathcal{L}(Z_r))\mathrm{d} W_r + G^{\prime}_s\delta \hat{Y}_s))^T]A\\
    = & \mathbb{E}[(\overline{X}_s - \overline{Z}_s)
    (Dh(Z_s)G^{\prime}\delta \hat{Y}_{s,t})^{T}] A + I_{14} + I_{15}.
\end{align*}
The second term in the first line vanishes thanks to the martingale property of the stochastic integral.
Similarly, we expand $I_{23}$ using \eqref{eq: expansion difference h}. Notice that the term \eqref{eq: expansion difference h: martingale} vanishes because the stochastic integral is a martingale. The term \eqref{eq: expansion difference h: gubinelli derivative} produces a term of regularity $\alpha$, the others have regularity $2\alpha$.
\begin{equation*}
    I_{23} = \mathbb{E}[\overline{X}_s(Dh(X_s) F^{\prime} \delta Y_{s,t}
	- Dh(Z_s) G^{\prime} \delta \hat{Y}_{s,t})^T] A + I_{24},
\end{equation*}
where $I_{24} := \mathbb{E}[\overline{X}_s( 
\eqref{eq: expansion difference h: reminder first} + \dots +
\eqref{eq: expansion difference h: reminder last})]A$.
We thus obtain that
\begin{align*}
    \delta (P(\pi^x_{\cdot})-P(\pi^z_{\cdot}))_{s,t} 
    = \mathbb{E}[(\overline{X}_s - \overline{Z}_s)
    (Dh(Z_s)G^{\prime}\delta \hat{Y}_{s,t})^{T}]A
    + \mathbb{E}[\overline{X}_s(Dh(X_s) F^{\prime} \delta Y_{s,t}
	- Dh(X_s) G^{\prime} \delta \hat{Y}_{s,t})] A
	+ R_{s,t}
\end{align*}
with $R_{s,t} = I_{11} + I_{12} + I_{14} + I_{15} + I_{21} + I_{22} + I_{24}$.
For $1 \leq l \leq D$ and $1 \leq j \leq d$ we write the $(l, j)$ entry of the matrix as
	\begin{align*}
	    \delta (P(\pi^x) -P(\pi^z))^{l,j}_{s,t} 
	    = &
	    \sum_{k=1}^{D}\Big(\mathbb{E}[(\overline{X}_s - \overline{Z}_s)^{l}
    (Dh(Z_s)^{\top}A)^{i,k}(G^{\prime})^{k}]
    + \mathbb{E}[(\overline{X}_s ((Dh(X_s) -
    Dh(Z_s))^{\top}A)^{i,k}(G^{\prime})^{k}]\\
    & + \mathbb{E}[\overline{X}_s (Dh(X_s)^{\top}A)^{i,k}(F^{\prime} - G^{\prime})^{k}]
    \Big)
    \cdot\delta \hat{Y}_{s,t}
    + \sum_{k=1}^{D} \mathbb{E}[\overline{X}_s^{l}(Dh(X_s)^{\top}A)^{i,k} (F^{\prime})^{k} ]
    \cdot \delta (Y - \hat{Y})_{s,t}
    + R^{P^{l,j}}_{s,t}.
	\end{align*}
We must now find  estimates for $[R]_{2\alpha}$. We start by some estimates of the processes in $L^{\rho}$.
From equation \eqref{eq: Lp lipschitz sde} we have
\begin{equation}
\label{eq: Lp lipschitz no x0}
    	\| \sup_{t\in [0,T]} |X_t - Z_t|  \|_{L_{\omega}^{\rho}}
	\leq
	C
	\left(
	\| F - G \|_{C_b} + \|X_0 - Y_0\|_{L_{\omega}^\rho}
	\right).
	\end{equation}
It follows from the equations for $\overline{X}$ and $\overline{Z}$ as well as estimate \eqref{eq: Lp lipschitz no x0} that
\begin{align*}
    \| [\overline{X} - \overline{Z}]_{\alpha} \|_{L^{\rho}_{\omega}}
    \leq & C_T (1+ \|\overline{X}_0 \|_{L^{\rho}_{\omega}})
    \| \sup_{t\in[0,T]} (X_t - Z_t ) \|_{L^{\rho}_{\omega}}
    \leq C_T
	\left(
	\| F - G \|_{C_b} + \|X_0 - Y_0\|_{L_{\omega}^\rho}
	\right).
\end{align*}
Using Lemma \ref{lem: bounds X centered} we also have
\begin{equation*}
    \| [\overline{X} - \overline{Z}]_{\alpha}\|_{L^{\rho}_{\omega}}
    \leq \| [\overline{X}]_{\alpha}\|_{L^{\rho}_{\omega}}
    +
    \| [\overline{Z}]_{\alpha}\|_{L^{\rho}_{\omega}}
    \leq C_T.
\end{equation*}
From \eqref{eq: contraction hoelder bound} we have
\begin{equation*}
    	\|[X-Z]_{\alpha}\|_{L^{\rho}_{\omega}}
	\leq  C_T
	\left(
	\| F-G \|_{C_b} + \|X_0 - Y_0\|_{L_{\omega}^\rho}
	\right)
	+ C [F-G]_{\alpha}.
	\end{equation*}
From \eqref{eq: hoelder bound sde} we obtain
\begin{equation*}
    \| [X]_{\alpha} \|_{L^{\rho}_{\omega}},
    \;
    \| [Z]_{\alpha} \|_{L^{\rho}_{\omega}}
    \leq C_T.
\end{equation*}
Using the previous estimates we obtain
\begin{equation*}
    \delta (I_{11}+I_{21})_{s,t}
    \leq C_T \| h\|_{C^1_b}
    \| \sup_{t\in[0,T]} |X_t-Z_t| \|_{L^{\rho}_{\omega}} |t-s|^{2\alpha},
\end{equation*}
\begin{equation*}
    \delta (I_{12}+I_{22})_{s,t}
    \leq C_T \| h\|_{C^1_b}
    \| [X-Z]_{\alpha}
    \|_{L^{\rho}_{\omega}} |t-s|^{2\alpha},
\end{equation*}
\begin{equation*}
    \delta (I_{14}+I_{15})_{s,t}
    \leq \| h\|_{C^2_b}
    \| [\sup_{t\in[0,T]} |X_t-Z_t|
    \|_{L^{\rho}_{\omega}}
    (
    C_T + [R^G]_{2\alpha}
    + \|G^{\prime}\|_{C_b} [\hat{Y}]_{\alpha}
    )
    |t-s|^{2\alpha}.
\end{equation*}
For $I_{24}$ we have a combination of the above and an extra term
\begin{equation*}
    [I_{24}]_{2\alpha} \leq \dots + \|F^{\prime} - G^{\prime}\|_{C_b} [\bY]_{\alpha}
    + \|G^{\prime} \|_{C_b}
     [\bY - \hat{\bY}]_{\alpha}.
\end{equation*}
Summing up, we get
\begin{equation*}
    [R]_{2\alpha} 
    \leq
    C
    \left(
	\| F-G \|_{C_b} + \|X_0 - Y_0\|_{L_{\omega}^\rho}
	+ [F-G]_{\alpha}
	\right)
	(
    C_T + [R^G]_{2\alpha}
    + \|G^{\prime}\|_{C_b} [\hat{Y}]_{\alpha}
    )
    +
    \|F^{\prime} - G^{\prime}\|_{C_b} [\bY]_{\alpha}
    + \|G^{\prime} \|_{C_b}
     [\bY - \hat{\bY}]_{\alpha},
\end{equation*}
which concludes the proof using the assumption $\| G \|_{\mathscr{D}^{2\alpha}_{\hat{\bY}}} \leq C \| \overline{\bY} \|_{\alpha}$
\end{proof}

We are now ready to prove the main well-posedness result.
\begin{theorem}
\label{thm: well-posedness rough McKean}
Let $1/3 < \alpha < \bar{\alpha} < 1/2$. Let $\rho > 2/(1-2\alpha)$ and $X_0, Z_0 \in L^{\rho}$. Let $X_0 \in L^\rho$ and $\bY \in \mathscr{C}^{\bar{\alpha}}(\R^d)$.  If $b, \sigma$ and $h$ satisfy Assumptions \ref{assumptions coefficients} and \ref{assumpion: h}, then equation \eqref{eq: rough McKean} admits a unique solution $(X,F) \in L^\rho(C([0,T],\R^D))\times \mathscr{D}^{2\alpha}_{\bY}([0,T],\R^D)$.
\end{theorem}
\begin{proof}
A process $X: \Omega \to C([0,T],\R^D)$ is a solution to equation \eqref{eq: rough McKean} if and only if $P(\mathcal{L}(X))$ is a fixed point of $\Gamma$.

We want to prove that $\Gamma: D(\Gamma)\to D(\Gamma)$ is a contraction.
If $f,g \in D(\Gamma)$ we have that $F, G$ defined as in \eqref{eq: F from f} satisfy the assumptions of Lemma \ref{lem: contractive estimates} with a global constant $C = C(h,M)$. By taking $\bY = \hat{\bY}$ and $X_0 = Z_0$ the estimate in Lemma \ref{lem: contractive estimates} reduces to
\begin{equation*}
    \| \Gamma(f) - \Gamma(g) \|_{\mathscr{D}^{2\alpha}_{\bY}} 
    \leq C_T(M, \| \bY \|_{\bar{\alpha}})
    (\| f - g \|_{\mathscr{D}^{2\alpha}_{\bY}}
    -\| f_0 - g_0 \|_{C_b}),
    \qquad
    f,g \in D(\Gamma).
\end{equation*}
Choosing now $T_0$ small enough such that $C_{T_0} < 1$, we have that $\Gamma$ is a contraction on the closed subset $D(\Gamma) \subset \mathscr{D}^{2\alpha}_{\bY}([0,T_0],\mathcal{L}(\R^d,\R^D))$. Hence, $\Gamma$ admits a unique fixed point.

Since the small constant $C_{T_0}$ in the definition of the domain $D(\Gamma)$ of $\Gamma$ and in the contraction argument only depends on the global quantities $h,M,\bY$ and not on the initial condition $X_0$, we can construct a finite family of subsequent time intervals of size $T_0$ that covers $[0,T]$. On each of these time interval we construct the solution as a fixed point of $\Gamma$.

\end{proof}



\begin{corollary} \label{cor:Wasserstein lipschitz}
Let $1/3 < \bar{\alpha} < 1/2$ and $\rho > 2/(1-2\bar{\alpha})$. Let $X_0 \in L^{\rho}$ and $\bY^1, \bY^2 \in \mathscr{C}^{\bar{\alpha}}$. For $i=1,2$, let $X^i$ be the solution to equation \eqref{eq: rough McKean} with driver $\bY^i$ and initial condition $X_0$. Call $\pi^i := \mathcal{L}(X^i) \in \mathcal{P}_{\rho}(C([0,T], \R^D))$. There exists a positive constant $C >0$ such that
\begin{equation*}
    W_{\rho, C([0,T], \R^D)}(\pi^1,\pi^2)
    \leq
    C \rho_{\bar{\alpha}}(\bY^1, \bY^2).
\end{equation*}
\end{corollary}
\begin{proof}
 It follows form Lemma \ref{lem: a priori sde} and Lemma \ref{lem: contractive estimates} that the solution map $\bY \mapsto X$ of equation \eqref{eq: rough McKean} is a Lipschitz-continuous function between the spaces $\mathscr{C}^{\bar{\alpha}}(\R^d)$ and $L^{\rho}(C([0,T],\R^D))$.
 The corollary follows as the Wasserstein distance between $\pi^1$ and $\pi^2$ is always less or equal to the $L^{\rho}$ distance between $X^1$ and $X^2$.
\end{proof}
We now proceed to prove a Wong-Zakai type result when $\bY$ is the It\^{o} lift of a Brownian motion. We introduce the approximation $X^n$ as the solution of 
\begin{subequations}
\label{eq: WZ approximation equation}
\begin{align}
\mathrm{d}X^n_t & = b(X^n_t, \mathcal{L}(X^n_t)) \,\mathrm{d} t + \sigma(X^n_t, \mathcal{L}(X^n_t)) \,\mathrm{d} W_t + \mathrm{d} F^n_t,
\qquad
X_t|_{t=0} = X_0,\label{eq: WZ approximation equation 1}\\
\dot{F}^n_t & = P(\mathcal{L}(X^n_t)) \dot{Y}_t^n + \frac12 P(\mathcal{L}(X_t^n))^{\top}\mathbb{E}[\overline{X}_t^n Dh(X_t^n)^{\top}]A   .
\label{eq: WZ approximation equation 2}
\end{align}
\end{subequations}
where $Y^n$ is a piecewise linear approximation of $Y$. We have the following result.

\begin{theorem}
\label{thm: convergence geometric}
The solution $X^n$ converges to $X$ in the Wasserstein distance, viz $\mathbb{P}$-a.s. we have
$$
W_{\rho,C([0,T],\R^D)}(\mathcal{L}(X^n), \mathcal{L}(X)) \rightarrow 0
$$
as $n \rightarrow \infty$.
\end{theorem}

\begin{proof}
 We start by noting that $Y^n$ solves \eqref{eq: WZ approximation equation} if and only if it solves the rough path equation 
\begin{align*}
\mathrm{d}X^n_t & = b(X^n_t, \mathcal{L}(X^n_t)) \,\mathrm{d} t + \sigma(X^n_t, \mathcal{L}(X^n_t)) \,\mathrm{d} W_t + \mathrm{d} F^n_t,
\qquad
X_t|_{t=0} = X_0,\\
\mathrm{d} F^n_t & = P(\mathcal{L}(X^n_t)) \,\mathrm{d} \bY_t^{n,Ito}, 
\end{align*}
 where 
 $$
 \bY^{n,Ito}_{st} = \left( \delta Y^n_{s,t} , \int_s^t \delta Y^n_{s,r} \otimes \dot{Y}_r^n \mathrm{d}r \textcolor{red}{-} \frac12 (t-s) I_{d \times d} \right).
 $$
 It is well known that the canonical lift of $Y^n$, 
  $$
 \bY^{n,Str}_{st} = \left( \delta Y^n_{s,t} , \int_s^t \delta Y^n_{s,r} \otimes \dot{Y}_r^n \mathrm{d}r \right),
 $$
 converges $\mathbb{P}$-a.s. in the rough path topology to the rough path
   $$
 \bY^{Str}_{st} = \left( \delta Y_{s,t} , \int_s^t \delta Y_{s,r} \otimes \circ \mathrm{d}Y_r \right).
 $$
 where the latter integral is the Stratonovich integral. From this it follows immediately that 
 $$
 \bY^{n,Ito}_{st} \rightarrow \left( \delta Y_{s,t} , \int_s^t \delta Y_{s,r} \otimes \circ \mathrm{d}Y_r \textcolor{red}{-} \frac12 (t-s) I_{d \times d} \right)= \left( \delta Y_{s,t} , \int_s^t \delta Y_{s,r} \otimes  \mathrm{d}Y_r \right)
 $$
 where the latter integral is the It\^{o} integral. 
 The result now follows from Corollary \ref{cor:Wasserstein lipschitz}.
\end{proof}


\subsection{Proof of Theorem \ref{thm: main}}
\label{section: proof main theorem}
We can now proceed with the proof of the main theorem.

\textbf{Well-posedness of the rough stochastic McKean-Vlasov equation \eqref{eq: main system}.}
Given the coefficients of equation \eqref{eq: main system} we transform it into equation \eqref{eq: rough McKean} by defining the following coefficients
\begin{equation*}
    b(x,\pi) := f(x) - P(\pi)h(x),
\end{equation*}
\begin{equation*}
        \sigma(x,\pi) := \left(
    \begin{array}{cc}
    G^{\frac12} -P(\pi)U, &
    -P(\pi) R^{\frac{1}{2}}
    \end{array}
    \right),
\end{equation*}
\begin{equation*}
    P(\pi) := \operatorname{Cov}_{\pi}(x,h)C^{-1} + B.
\end{equation*}
The $m$-Brownian motion $W$ in \eqref{eq: rough McKean} stands for the paired independent Brownian Motions $(\widehat{W},\widehat{V})$ of equation \eqref{eq: main system}, with $m = D + d$.
Moreover, if $\bY = (Y, \YY)$ is any given rough path, we can modify it by a bounded variation term to obtain
\begin{equation*}
    \overline{\bY}_{s,t} := (Y_{s,t}, \YY_{s,t} \textcolor{red}{-} \frac{1}{2}(t-s)\textcolor{red}{C}).
\end{equation*}
We have that equation \eqref{eq: rough McKean} driven by $\overline{\bY}$ corresponds to equation \eqref{eq: main system} driven by $\bY$ and the term $\frac{1}{2}(t-s)\textcolor{red}{C}$ generates the correction term $\Gamma(\pi)$. Hence, the proof of well-posedness and stability of equation \eqref{eq: main system} follows from Theorem \ref{thm: well-posedness rough McKean} and Corollary \ref{cor:Wasserstein lipschitz}.

\textbf{Well-posedness of the McKean-Vlasov equation with common noise \eqref{eq: main system stochastic version}.} The uniqueness of equation \eqref{eq: main system stochastic version} follows from Lemma \ref{lem: uniqueness mckean common noise}, as the coefficents of equation \eqref{eq: main system stochastic version} satisfy Assumption \ref{assumptions coefficients common noise}. The following lemma gives existence.

\begin{lemma}
\label{lem:wellposedness McKean common noise}
Let $X^{\mathbf{y}}$ denote the solution of \eqref{eq: main system} and let $\mathbf{Y}(\omega)$ denote the Stratonovich rough path lift of the Brownian motion $Y$. Then $X^{\mathbf{Y}}$ is the solution of \eqref{eq: main system stochastic version} and we have 
\begin{equation} \label{eq:rough conditional expectation}
\mathbb{E}\left[ \phi(X^{\mathbf{Y}}_t) | \mathcal{Y}_t \right] (\omega) = \mathbb{E}\left[ \phi(X_t^{\mathbf{y}}) \right] |_{\mathbf{y} = \mathbf{Y}(\omega)|_{[0,t]} } , \quad \mathbb{P}_{Y} - a.e. \omega.
\end{equation}
\end{lemma}

\begin{proof}
Since the Stratonovich rough path lift of $Y|_{[0,t]}$ coincides with $(Y,\mathbb{Y})|_{[0,t]}$, we notice that by uniqueness, we have 
$$
X_t^{\mathbf{Y}} = X_t^{\mathbf{Y}|_{[0,t]}} .
$$
Since $\mathbf{Y}|_{[0,t]}$ is independent of $\sigma(W_s,X_0, 0 \leq s \leq t)$, then it is also independent of $X^{\mathbf{y}}_t$ for every $\mathbf{y} \in \mathscr{C}^{\alpha}([0,t])$. Using the monotone class theorem for functions we get \eqref{eq:rough conditional expectation}.

From \cite[Corollary 5.2]{friz2020course} we get that 
$$
\int P(\hat{\pi}) \mathrm{d} \mathbf{Y} = \int P(\hat{\pi}) \circ \mathrm{d} Y,
$$
and hence $X^{\mathbf{Y}}$ satisfies \eqref{eq: true McKean Vlasov}. The result follows by uniqueness. 
\end{proof}

\section{The interacting particle system}
\label{sec: particles}

In this section we prove well-posedness and convergence of the interacting particle system \eqref{eq: interacting particles} in the case when $Y$ is a path of bounded variation.

To compress the notation and have a slightly more general result we study the following mean-field system
\begin{equation}
\label{eq: mean-field}
\mathrm{d}X^{i,N}_t 
= [b(X^{i,N}_t, \mu^N_t) 
+ P(\mu^N_t) \dot{Y}_t + \Gamma(\mu^N_t) ] \,\mathrm{d} t
+ \sigma(X^{i,N}_t ,\mu^N_t) \,\mathrm{d} W^i_t,
\quad
X_t^{i,N} |_{t=0} = X^{i}_0 ,
\qquad i = 1, \dots, N.
\end{equation}
The variable $X^{i,N}_t $ is in the state space $\R^D$. $(W^i)_{i\in\N}$ is a family of independent $m$-dimensional Brownian motions and $(X_0^i)_{i\in\N}$ is a family of independent and identically distributed initial conditions with law $\pi_0 \in \mathcal{P}(\R^D)$. Moreover, assume that $\dot{Y}: [0,T] \to \R^D$ is cadlag and bounded. We assume that the coefficients $b$, $\sigma$ and $P$ satisfy Assumption \ref{assumptions coefficients}, but notice that $P$ only depends on the measure. For the coefficient $\Gamma$ we assume the following

\begin{assumption}
	\label{assumptions Gamma}
		Let $\rho \geq 1$, assume that there exists a constant $C > 0$ such that
		\begin{enumerate}[label=(\roman*), ref= \ref{assumptions coefficients} (\roman*)]
			\item \label{assumpitons Gamma: linear growth} 
			(linear growth)  $\forall \pi \in \mathcal{P}_{\rho}(\R^d)$,
			\begin{equation*}
			|\Gamma(\pi) | \leq C (1 + \overline{M}^{\rho}(\pi)  )^{\frac{2}{\rho}} ,
			\end{equation*}
			\item  \label{assumpitons Gamma: lipschitz}
			(Lipschitz continuity) $\forall \pi, \nu \in \mathcal{P}_{\rho}(\R^D)$, 
			\begin{equation*}
			| \Gamma(\pi) - \Gamma(\nu) | \leq C ( 1 + \overline{M}^{\rho}(\pi)^2 + \overline{M}^{\rho}(\nu))^{\frac{1}{\rho}} W_{\rho}(\pi, \nu).
			\end{equation*}
		\end{enumerate}
\end{assumption}

\begin{remark}
Notice that the growth is quadratic in the central moments of the measure. Also, the best Lipschitz constant we can hope for is the square of the central moment of one measure and the central moment of the other.
\end{remark}

\begin{remark}
\label{rmk: particles coefficient conversion}
By doing the same substitution as in Section \ref{section: proof main theorem} and taking \eqref{eq: approx Gamma} for $\Gamma$ we recover the interacting particle system \eqref{eq: interacting particles}.
\end{remark}

\begin{lemma}
\label{lem: well posedness particles}
Let $T>0$ and $\rho \geq 1$. If the initial distribution $\pi_0$ has finite $\rho$-moment, then  equation \eqref{eq: mean-field} admits a pathwise unique strong solution on $[0,T]$. Moreover, there exists $C = C(\rho, \pi_0, T) > 0$ such that
\begin{equation}
\label{eq: bounds empirical measure}
\frac{1}{N}\sum_{j=1}^N\mathbb{E}[\sup_{t\in [0,T]}|\overline{X}^{j,N}_t|^{\rho} ]\leq C,
\qquad
\max_{i=1,\dots,N}\mathbb{E}[\sup_{t\in[0,T]}|X^{i,N}_t|^{\rho}] \leq C (1+\|\dot{Y}\|_{\infty}),
\end{equation}
where $\overline{X}^{i,N}= X^{i,N} - \frac{1}{N}\sum_{j=1}^N X^{j,N}$, for $i= 1,\dots, N$.
\end{lemma}

\begin{proof}
Under Assumptions \ref{assumptions coefficients} and \ref{assumptions Gamma}, the coefficients are locally Lipschitz and there is classically strong existence and pathwise uniqueness for the SDE \eqref{eq: mean-field} up to an explosion time $\tau$.  We want to prove that $\tau < T$. Let us call the solution $X^{(N)}_t = (X^{i,N}_t, \dots , X^{N,N}_t)$. Since the coefficents $P$ and $\Gamma$ do not depend on the state variable $x$ we have the following identity for $\overline{X}^{i,N}$, $i=1,\dots, N$,
\begin{align*}
    d \overline{X}^{i,N}_t = (b(X^{i,N}_t, \mu^{N}_t) - \frac{1}{N}\sum_{j=1}^N b(X^{j,N}, \mu^N_t)) \,\mathrm{d} t
    + \sigma(X_t^{i,N}, \mu^N_t) \, \mathrm{d} W_t^{i} - \frac{1}{N} \sum_{j=1}^N \sigma(X^{j,N}, \mu^N_t)\,\mathrm{d} W^j_t.
\end{align*}
Notice that
\begin{equation*}
    \overline{M}^{\rho}(\mu^N_t)
    = \mathbb{E}_{\mu^N_t}[|X-\mathbb{E}_{\mu^N_t}[X]|^{\rho}]
    =
    \frac{1}{N}\sum_{i=1}^N|\overline{X}^{i,N}_t|^{\rho}.
\end{equation*}
Using Assumption \ref{assumptions coefficients} we obtain that for every $\rho \geq 2$ there exists a constant $C > 0$ independent of $\dot{Y}$ and $N$ such that
\begin{equation*}
\mathbb{E}[\sup_{s\leq t}|\overline{X}^{i,N}_s|^\rho] \leq \mathbb{E}[|\overline{X}^{i,N}_0|^\rho] + C \int_{0}^t (1+ \frac{1}{N}\sum_{j=1}^N \mathbb{E}[|\overline{X}^{j,N}_s|^\rho] ) \,\mathrm{d} s.
\end{equation*}
Taking the mean over $i$ and using Gronwall's lemma we obtain the first estimate in \eqref{eq: bounds empirical measure}. We can now estimate the moments of $X^{i,N}$ as follows,
\begin{align*}
    \mathbb{E}[\sup_{t\in[0,T]}|X^{i,N}_t|^{\rho}]
    \leq C(1+\|\dot{Y}\|_{\infty}) \int_0^{T}\mathbb{E}\left[(1+ \overline{M}^{\rho}(\mu^N_s))\right]^{2} \mathrm{d} s
    \leq (1+\|\dot{Y}\|_{\infty}) C,
\end{align*}
where $C = C(\rho, \pi_0, T)$ can change from one inequality to the next. Since $X^{(N)}$ has finite $L_{\omega}^{\rho}$-norm on the interval $[0,T]$, it means that the explosion time satisfies $\tau > T$. 
\end{proof}
\begin{remark}
Notice that in Lemma \ref{lem: well posedness particles} the choice of $T>0$ was arbitrary, which implies that the system of interacting particles is well posed on intervals of any length.
\end{remark}



\subsection{Propagation of chaos}
For $i = 1, \dots, N$, we introduce the following McKean-Vlasov equation
\begin{equation}
\label{eq: regular McKean-Vlasov}
\mathrm{d}X^{i}_t 
= [b(X^{i}_t, \pi_t) 
+ P(\pi_t) \dot{Y}_t + \Gamma(\pi_t) ] \,\mathrm{d} t
+ \sigma(X^{i}_t ,\pi_t) \,\mathrm{d} W^i_t,
\quad
\pi_t = \mathcal{L}(X^i_t),
\quad
X_t^{i} |_{t=0} = X^{i}_0 \sim \pi_0.
\end{equation}
By construction the random variables $X^1, \dots, X^N$ are independent and identically distributed.

Equation \eqref{eq: regular McKean-Vlasov} is well-posed because it corresponds to equation \eqref{eq: rough McKean} driven by the \textit{rough} path 
\begin{equation*}
    \bY_{st} := (Y_{s,t}, \int_s^r Y_{s,r}\dot{Y}_r \,\mathrm{d} r - \frac{1}{2}(t-s)I_{d\times d}),
\end{equation*}
where $Y_t := \int_0^t \dot{Y}_{r} \,\mathrm{d}r$. Notice that, since $\dot{Y}$ is cadlag and bounded, the path $\bY$ is Lipschitz continuous which is much more regular than $\alpha$-Hölder with $\alpha \in (\frac{1}{3}, \frac{1}{2}]$.

Let $\overline{X}^i := X^i - \mathbb{E}[X^i]$ and notice that, by definition $\overline{M}^{\rho}(\pi_t) = \|\overline{X}^i_t\|_{L_{\omega}^\rho}^{\rho}$. Using the a priori estimate in Lemma \ref{lem: bounds X centered} and Assumptions \ref{assumptions coefficients} and \ref{assumptions Gamma} we can see that, for any $\rho \geq 2$ we have an estimate on the central moments of $\pi_t$,
\begin{equation*}
\overline{M}^{\rho}(\pi_t) \leq \|\sup_{t\in[0,T]}|\overline{X}^i_t|\|_{L_{\omega}^\rho}^{\rho}
	\leq
	C [1 + \overline{M}^{\rho}(\pi_0)]^{\rho}.
\end{equation*}
From this one can easily recover the following estimate for the moments of $\pi_t$,
\begin{equation}
    \label{eq: moments bounds}
	M^{\rho}(\pi_t)^{\frac{1}{\rho}} \leq C(1+\|\dot{Y}\|_{\infty})(M^{\rho}(\pi_0)+\overline{M}^{2\rho}(\pi_t))^{\frac{1}{\rho}}
	\leq C (M^{\rho}(\pi_0)+\|\dot{Y}\|_{\infty})(1 + \overline{M}^{2\rho}(\pi_0))^{2},
\end{equation}
where $C = C(\rho, T)$.
\begin{remark}
\label{rmk: second moments initial condition}
Notice that the $\rho$-moment of $\pi_t$ is only bounded by the $2\rho$-central moment of $\pi_0$ because of the quadratic growth of $\Gamma$ from Assumption \ref{assumptions Gamma}. 
\end{remark}
In the following we will hide the dependence on $\pi_0$ in the constant $C$ and only focus on the explicit dependence on $\|\dot{Y}\|_{\infty}$.
We define the empirical measure associated with the independent particles $X^i$ as $\overline{\mu}^N_t := \frac{1}{N}\sum_{i=1}^N \delta_{X^i_t}$. It follows from \cite[Theorem 1]{fournier2015rate} that for any $\overline{\rho} > \rho$ there exists an explicit rate of convergence $H(N) = H(N,\rho, \overline{\rho})$ such that $H(N) \to 0$ as $N\to \infty$ and
\begin{equation}
    \label{eq: rate iid particles}
    \mathbb{E}[W_{\rho,\R^D}(\overline{\mu}^N_t, \pi_t)^{\rho}] \leq   M^{\overline{\rho}}(\pi_t)^{\frac{\rho}{\overline{\rho}}}H(N)
    \leq C (1+\|\dot{Y}\|_{\infty}) H(N),
\end{equation}
where in the last inequality we used \eqref{eq: moments bounds} and $C$ depends on the moments of $\pi_0$.

Moreover, the rate is optimal and explicitly given as
\begin{equation*}
    H(N) = \left\{
    \begin{array}{ll}
      N^{-\frac{1}{2}} + N^{1-\frac{\rho}{\overline{\rho}}}  & \rho > \frac{D}{2}, \; \overline{\rho} \neq 2\rho \\
      N^{-\frac{\rho}{D}} + N^{1-\frac{\rho}{\overline{\rho}} }  & 1 \leq \rho < \frac{D}{2},\; \overline{\rho} \neq \frac{D}{D-2}.
    \end{array}
    \right.
\end{equation*}
Notice that if we have enough moments the optimal rate of convergence for independent particles is $H(N) = \frac{1}{\sqrt{N}}$.
Furthermore, we have the following estimate, because $\mu^N$ and $\overline{\mu}^N$ are both empirical measures,
\begin{equation*}
    W_{\rho,\R^D}(\mu^N_t, \overline{\mu}^N_t)^{\rho}
    \leq \frac{1}{N} \sum_{i=1}^N|X^{i,N}_t - X^{i}_t|^{\rho},
\end{equation*}
so that we have from the triangular inequality
\begin{equation}
    \label{eq: wasserstein empirical measures}
    \sup_{t\in[0,T]}\mathbb{E}[W_{\rho,\R^D}(\mu^N_t, \pi_t)^{\rho}] \leq  C (1+\|\dot{Y}\|_{\infty})H(N) + \sup_{t\in[0,T]} \frac{1}{N} \sum_{i=1}^N\mathbb{E}|X^{i,N}_t - X^{i}_t|^{\rho},
\end{equation}
where $C$ depends on the $2\overline{\rho}$-moment of $\pi_0$, for $\overline{\rho}>\rho$.
Now that we stated most of the preliminaries we can prove the following convergence result.
\begin{proposition}
    \label{pro: regular propagation of chaos}
    Let $\overline{\rho} > \rho \geq 2$. If $X_0 \in L_{\omega}^{2\overline{\rho}}$, there is a function $J(N) = O(\log(N)^{-1})$ such that
    \begin{equation*}
        \max_{i=1,\dots,N} \mathbb{E}[\sup_{t\in[0,T]}|X^{i,N}_t - X^{i}|^{\rho}] \lesssim e^{ \|\dot{Y}\|_{\infty}} J(N).
    \end{equation*}
\end{proposition}
\begin{proof}
We call $\overline{X}^{i,N} := X^{i,N} - \frac{1}{N}\sum_{j=1}^N X^{j,N}$, for $i= 1,\dots, N$ and recall that the central moment for the empirical measure is $\overline{M}^{\rho}(\mu_t^N) = \frac{1}{N}\sum_{i=1}^N |\overline{X}^{i,N}_t|^{\rho}$. For $R > 0$, we define the stopping time $T_R 
:= \inf\{t\geq0:\overline{M}^{\rho}(\mu_t^N) \geq R\}$. 

We set $Z^i:=X^{i,N} - X^{i}$ and using Assumptions \ref{assumptions coefficients} and \ref{assumptions Gamma} as well as inequality \eqref{eq: wasserstein empirical measures} and Lemma \ref{lem: bounds X centered} we compute the following,
\begin{align*}
    \max_{i=1,\dots,N} \mathbb{E}[|Z^i_t|^{\rho}1_{\{T_R > t\}}]
    \leq & C \mathbb{E}\left[
    \int_{0}^{t\wedge T_R} (1 + \overline{M}^{2\rho}(\pi_s) \|\dot{Y}\|_{\infty} + \overline{M}^{\rho}(\mu^N_s) )
    \left(|Z_s^i| + W_{\rho,\R^D}(\pi_s, \mu^N_s)\right)^{\rho} \mathrm{d} s
    \right]\\
    \leq & C (1 + \|\dot{Y}\|_{\infty} + R)\left[
    \int_{0}^{t} \max_{i=1,\dots,N} \mathbb{E}[|Z^i_s|^{\rho}1_{\{T_R\geq s\}}] \mathrm{d} s + H(N)\right].
\end{align*}
Using Gronwall's inequality we obtain
\begin{equation*}
    \max_{i=1,\dots,N} \mathbb{E}[|Z^i_t|^{\rho}1_{\{T_R > t\}}]
    \leq C (1 + \|\dot{Y}\|_{\infty}) e^{C(1 + \|\dot{Y}\|_{\infty})} R e^{CR} H(N).
\end{equation*}
Now we compute the following using Cauchy-Schwarz and Markov inequalities as well as \eqref{eq: bounds empirical measure},
\begin{align*}
\mathbb{E}[|Z_t^i|^\rho 1_{\{T_R \leq t\}}]
\leq C \left( \mathbb{E}[|Z_t^i|^{2\rho}] 
\right)^{\frac{1}{2}}
\mathbb{P}(T_R \leq t)^{\frac{1}{2}}
\leq C \mathbb{P}(\sup_{s\leq t}\overline{M}^{\rho}(\mu_s^N) \geq R)^{\frac{1}{2}}
\leq \frac{C}{R^{\frac{1}{2}}} \left(\frac{1}{N}\sum_{i=1}^N \mathbb{E}[|\sup_{t\in[0,T]}\overline{X}^{i,N}_t|^{\rho}]\right)^{\frac{1}{2}}
\leq \frac{C}{R^{\frac{1}{2}}}.
\end{align*}
Notice that we used $\mathbb{E}[|Z_t^i|^{2\rho}]\leq C (\mathbb{E}[|X^{i,N}_t|^{2\rho}] + \mathbb{E}[|X^{i}_t|^{2\rho}]) \leq C(1 + \|\dot{Y}\|_{\infty})$, hence we need from Remark \ref{rmk: second moments initial condition} that the initial measure $\pi_0$ has finite $2\rho$ moments.
We can put together the estimates to obtain
\begin{equation*}
    \max_{i=1,\dots,N}\mathbb{E}[|Z_t^i|^\rho]
\leq \frac{C}{R^{\frac{1}{2}}} + C (1 + \|\dot{Y}\|_{\infty} + R) e^{C(1 + \|\dot{Y}\|_{\infty})} e^{CR} H(N)
\leq C e^{C(1 + \|\dot{Y}\|_{\infty})} (\frac{1}{R^{\frac{1}{2}}} + e^{CR} H(N)),
\end{equation*}
where we changed the constants from one line to the next. Now choose $R = R(N)$ such that 
\begin{equation*}
    R(N) \to \infty, \qquad
    e^{CR(N)} H(N) \to 0, \qquad
    \mbox{as } N \to \infty.
\end{equation*}
Remember that $H(N) \approx N^{-\gamma}$ with $\gamma < \frac{1}{2}$ so that choosing $R(N)\approx \log(N^{\overline{\gamma}})$ for some $0 <\overline{\gamma} < \frac{\gamma}{C}$ we have a rate of convergence 
\begin{equation*}
    J(N) \approx (\overline{\gamma}\log(N))^{-1} + N^{\overline{\gamma}C - \gamma}
    \approx \log(N)^{-1}.
\end{equation*}
\end{proof}

\begin{remark}
Notice that the rate of convergence is far from the optimal $\frac{1}{\sqrt{N}}$ of a sample of independent and identically distributed random variables. This is due to the non-local Lipschitz condition on $\Gamma$ in Assumption \ref{assumptions Gamma}.
\end{remark}

Finally, we can put together Corollary \ref{cor:Wasserstein lipschitz} and Proposition \ref{pro: regular propagation of chaos}.

\begin{theorem}
\label{thm: particles convergence}
    Let $\delta > 0$ and $\bY^{\delta}$ be a bounded differentiable approximation of a geometric rough path $\bY = (Y, \mathbb{Y})\in \mathscr{C}^{\overline{\alpha}}$. Let $\mu^{N,\delta}$ be the empirical measure of the system of mean-field particles \eqref{eq: mean-field} with input $\bY^{\delta}$. Moreover, let $\pi$ be the law of the solution to equation \eqref{eq: rough McKean} driven by $(Y, \mathbb{Y} - \frac{1}{2}(t-s) \mathrm{Id}) \in \mathscr{C}^{\overline{\alpha}}$.
    
    Then there exists $\rho >0$ and a sequence $\delta(N)$ such that $\delta(N) \to 0$ and
    \begin{equation*}
       \sup_{t\in[0,T]}\mathbb{E}[W_{\rho, \R^D}(\mu^{N,\delta(N)}_t, \pi_t)^{\rho}] \to 0,
       \qquad
       \mbox{as } N \to \infty.
    \end{equation*}
\end{theorem}
\begin{proof}
Let $\pi^{\delta}$ be the law of a solution to \eqref{eq: regular McKean-Vlasov}. By the triangular inequality, Proposition \ref{pro: regular propagation of chaos} and Corollary \ref{cor:Wasserstein lipschitz}
\begin{align*}
\sup_{t\in[0,T]}\mathbb{E}[W_{\rho, \R^D}(\mu^{N,\delta(N)}_t, \pi_t)^{\rho}]
\leq &
\sup_{t\in[0,T]}\mathbb{E}[W_{\rho, \R^D}(\mu^{N,\delta(N)}_t, \pi_t^{\delta})^{\rho}]
+ W_{\rho, C([0,T], \R^D)}(\pi^{\delta},\pi)\\
    \lesssim &
    e^{\|\dot{Y}^{\delta}\|_{\infty}} J(N) 
    + \rho_{\overline{\alpha}}(\bY^{\delta}, \bY).
\end{align*}
Choosing $\delta(N)$ such that $\|\dot{Y}^{\delta(N)}\|_{\infty} = o(\log(J(N)^{-1}))) = o(\log(\log(N))$ we have that the right hand side vanishes as $N\to \infty$.
\end{proof}

\section{The numerical scheme, construction of the lift and examples}
\label{sec: numerical examples}

In this section we derive the RP-EnKF (rough path ensemble Kalman filter) numerical scheme alluded to in the introduction (see equation \eqref{eq:numerical scheme}), discuss the construction of appropriate rough path lifts, and provide details concerning the implementation. Furthermore, we demonstrate its effectiveness in the context of misspecified and multiscale models by means of a few examples in the context of parameter estimation. This setting provides a convenient testbed for the scenario where the model and observation noises are correlated, and we expect our conclusions regarding (non-)robustness to be relevant more generally.
\\

\noindent \textbf{Discretising in time.} A natural scheme for approximating the dynamics of the interacting particle system \eqref{eq: interacting particles} is given by
\begin{subequations}
\label{eq:ns}
\begin{align}
X^i_{k+1} & = X^i_{k} + f(X_k^i)\Delta t  + G^{1/2} \sqrt {\Delta t}\,\xi_k^i + \widehat{P}_k\left( \Delta Y_k - (h(X^i_k)\Delta t + U \sqrt{\Delta t} \,  \xi_k^i + R^{1/2}  \sqrt{\Delta t }\,\eta_k^i) \right)
\\
\label{eq:RP correction}
 & +  \widehat{\text{Cov}}_k(x, Dh) \widehat{P}_k   \Delta \mathbb{Y}_k + \widehat{\Gamma}_k \Delta t,
\end{align} 
\end{subequations}
where $\Delta t > 0$ is the step size, and $(\xi_n^i)$ and $(\eta_n^i)$ denote independent zero mean Gaussian random variables with unit variance of dimensions $D$ and $d$, respectively. In the above display, $\widehat{P}_k := \widehat{\Cov}_k(x,h) C^{-1}+ B$ and $\widehat{\Cov}_k(x, Dh)$ refer to the standard unbiased empirical estimators of the covariance, that is, 
\begin{subequations}
\label{eq:covariances}
\begin{align}
    \widehat{\Cov}_k(x,h) & = \frac{1}{N-1} \sum_{i=1}^N \left( X_k^i - \bar{X}_k \right) \otimes h(X_k^i) \in \mathbb{R}^{D \times d},
    \\
\widehat{\Cov}_k(x,Dh) & = \frac{1}{N-1} \sum_{i=1}^N (X_k^i - \bar{X}_k) \otimes D h(X_k^i) \in \mathbb{R}^{D \times d \times D},
\end{align}
\end{subequations}
where 
\begin{equation}
\bar{X}_k = \frac{1}{N} \sum_{i=1}^N X_k^i
\end{equation}
refers to the empirical mean. With \eqref{eq:covariances} in place, the standard empirical estimator for $\Gamma$ as defined in \eqref{eq: approx Gamma} is given by
\begin{equation}
\label{eq: gamma k}
    \widehat{\Gamma}_{k,\gamma} = -\frac{1}{2} \operatorname{Trace} \left(\widehat{\Cov}_{k,\gamma}(x,Dh) \widehat{P}_k \right),
\end{equation}
where $ \widehat{\Gamma}_{k,\gamma}$ denotes the $\gamma^{\text{th}}$ component of $\widehat{\Gamma}_k \in \mathbb{R}^D$, and using a similar convention for $\widehat{\Cov}_{k,\gamma}(x,Dh)$. 
The precise meaning of the first term in equation \eqref{eq:RP correction} is
\begin{equation*}
    (\widehat{\text{Cov}}_k(x,Dh) \widehat{P}_k   \Delta \mathbb{Y}_k)_{\gamma}
    = \sum_{j,q=1}^d \sum_{r=1}^D\widehat{\text{Cov}}_k(x,Dh)_{\gamma,j,r} (\widehat{P}_k)_{r,q} 
    (\Delta \mathbb{Y}_k)_{q,j}
    \qquad
    \gamma = 1,\dots, D,
\end{equation*}
and analogously in equation \eqref{eq: gamma k} with $\Delta \YY_k$ replaced by the $d$-dimensional identity matrix.

\begin{remark}[Gubinelli derivative]
The first term in \eqref{eq:RP correction} is modelled after the Gubinelli derivative $P(\widehat{\pi}_{s})^{\prime}$ in \eqref{eq: Gubinelli derivative of P intro}. We would like to stress that a standard time discretisation of the interacting particle system \eqref{eq: interacting particles} according to Davie \cite{davie2008differential} would involve further contributions accounting for correlations between the particles as well as for cross terms induced by the joint lift $(Y,W,V) \mapsto ((Y,W,V),(\mathbb{Y},\mathbb{W},\mathbb{V}))$. At least formally, these additional terms vanish in the limit as $N \rightarrow \infty$, and our numerical experiments have not shown noticeable benefits of including them.
\end{remark}

\noindent \textbf{Constructing the lift $\mathbb{Y}$.}
In order to implement the RP-EnKF scheme defined in \eqref{eq:ns}, we need to posit the discrete-time second order increments $\Delta \mathbb{Y}_k \in \mathbb{R}^{d \times d}$, given discrete-time samples $y_0, y_1, \ldots,y_n \in \mathbb{R}^d$ from $(Y_t)_{0 \le t \le T}$. In what follows we will denote the piecewise-linear interpolation of $y_0, \ldots,y_n$ by $(\widetilde{Y}_t)_{0 \le t \le T}$ and consider the decomposition of $\Delta \mathbb{Y}_k$ into symmetric and skew-symmetric parts,
\begin{equation}
\label{eq:lift decomp}
\Delta \mathbb{Y}_k = \Delta \mathbb{Y}^{\sym}_k +\Delta \mathbb{Y}^{\skewp}_k,     
\end{equation}
where 
\begin{equation}
\Delta \mathbb{Y}^{\sym}_k := \sym(\Delta Y_k) := \frac{1}{2} \left( \Delta \mathbb{Y}_k + \Delta \mathbb{Y}_k^\top \right), \qquad \Delta \mathbb{Y}^{\skewp}_k := \skewp(\Delta Y_k) := \frac{1}{2} \left( \Delta \mathbb{Y}_k - \Delta \mathbb{Y}_k^\top \right). 
\end{equation}
For the symmetric part, we set
\begin{equation}
\label{eq:sym from data}
\Delta \mathbb{Y}_k^{\sym} = \sym \left(\int_{t_k}^{t_{k+1}} \widetilde{Y}_{t_k,r} \otimes \mathrm{d}\widetilde{Y}_r \right) = \frac{1}{2} (y_{k+1} - y_k) \otimes (y_{k+1} - y_k), 
\end{equation}
maintaining structural similarities to the defining algebraic relation of weakly geometric rough paths \cite[Section 2.2]{friz2020course}. Defining the skew-symmetric part is more challenging since\footnote{This relation expresses the fact that the area enclosed by a straight line with itself is zero.} 
\begin{equation}
\skewp \left(\int_{t_k}^{t_{k+1}} \widetilde{Y}_{t_k,r} \otimes \mathrm{d}\widetilde{Y}_r \right) = 0,
\end{equation}
indicating that information on the enclosed area between two neighbouring points is inevitably lost by the discretisation. \textcolor{red}{However, Wong-Zakai type results on piecewise linear interpolations of semimartingales \cite[Proposition 2]{coutin2005semi} as well as the convergence of simplified Euler schemes \cite{deya2012milstein,friz2014convergence} suggest that the scheme \eqref{eq:ns} with $\Delta \mathbb{Y}_k^{\mathrm{skew}} = 0$ and $\Delta \mathbb{Y}_k^{\sym}$ as defined in \eqref{eq:sym from data}, based on data $(Y_t)_{0 \le t \le T}$ obtained from \eqref{eq:obs}, recovers \eqref{eq: main system} in the limit $\Delta t \rightarrow 0$, and our numerical experiments support this conjecture. We leave a detailed analysis for future work and refer to  \cite{bailleul2015inverse,flint2016discretely,levin2013learning} for specific approaches towards estimating L{\'e}vy areas. We next discuss the setting when the data is only approximately obtained from \eqref{eq:signal}, in which case the skew-symmetric contributions $\Delta \mathbb{Y}^{\mathrm{skew}}_k$ will play a crucial role:}

\textcolor{red}{\textbf{Correcting $\mathbb{Y}$ in the context of model misspecification.} Let us consider the case when instead of observations $(Y_t)_{0 \le t \le T}$ from \eqref{eq:signal obs}, we have access to perturbed or modified data $(Y_t^\varepsilon)_{0 \le t \le T}$, so that $Y^\varepsilon \rightarrow Y$ in $C^{\alpha}([0,T];\mathbb{R}^d)$ for some $\alpha \in (\frac13,\frac12]$, almost surely (say). Particular instances of this situation occur when the filtering model \eqref{eq:signal obs} arises as a simplified description of a more elaborate (possibly multiscale) model in the limit as $\varepsilon \rightarrow 0$ (with $\varepsilon$ referring to a scale separation parameter in the multiscale scenario). For specific examples we refer to Sections \ref{subsec:magnetic field}-\ref{sec:two-scale potential} below. Assuming that $(Y^\varepsilon_t)_{0 \le t \le T}$ is in $C^\gamma([0,T];\mathbb{R}^d)$ for $\gamma > \frac{1}{2}$, the canonical lift $\mathbb{Y}^{\varepsilon}_{s,t} = \int_s^t Y^{\varepsilon}_{s,r} \, \mathrm{d}Y_r^{\varepsilon}$ is well defined, and it might seem natural to use $\mathbf{Y}^\varepsilon := (Y^{\varepsilon},\mathbb{Y}^\varepsilon)$ in the rough McKean-Vlasov dynamics \eqref{eq: main system}, or a discretised version thereof in the scheme \eqref{eq:numerical scheme}. However, it is then clearly possible that $\mathbf{Y}^\varepsilon \not\to \mathbf{Y} := (Y,\mathbb{Y}^{Strat})$ in $\mathscr{C}^{\alpha}([0,T];\mathbb{R}^d)$, with $\mathbb{Y}^{Strat}$ denoting the Stratonovich lift associated to \eqref{eq:obs}. In this case, the continuity statement of Theorem \ref{thm: main} suggests that the solutions to \eqref{eq: main system} driven by $\mathbf{Y}$ and $\mathbf{Y}^\varepsilon$ may be substantially different, in general, and hence inference based on $\mathbf{Y}^\varepsilon$ may be erroneous.\footnote{\textcolor{red}{We remind the reader that the system \eqref{eq: main system} has been derived on the basis of the approximation stated in Lemma \ref{lem: approximate K} which is only exact in the case when $f$ and $h$ are affine, and $\pi_0$ is Gaussian. For the sake of discussion in this section, we assume that the incurred error is negligible and the solution to \eqref{eq: main system} driven by $\mathbf{Y}$ provides sufficiently accurate estimates.}} 
\textcolor{red}{In this case, we say that the model providing $(Y^\varepsilon_t)_{0 \le t \le T}$ is \emph{misspecified} with respect to the filtering model \eqref{eq:signal obs} on which the schemes developed in this paper are built. The degree of misspecification can be quantified using the rough path metric $\rho_\alpha (\bY, \bY^\varepsilon)$, see Section \ref{sec: background in rough paths}.}}		

\textcolor{red}{In order to devise a practical method to address the problem exposed in the preceding paragraph, let us assume that there exists $\overline{\mathbf{Y}} = (\overline{Y}, \overline{\mathbb{Y}}) \in \mathscr{C}^{\alpha}([0,T];{\R}^d)$ such that $\bY^\varepsilon \rightarrow \overline{\bY}$. The rough path $\overline{\bY}$ is then geometric and lies above the same path as $\bY$, hence $Y = \overline{Y}$ as well as $\mathrm{sym}(\mathbb{Y}^{Strat}_{s,t}) = \mathrm{sym}(\overline{\mathbb{Y}}_{s,t})$, for all $s,t \in [0,T]$. We conclude that the (model misspecification) discrepancy
\begin{equation}
\label{eq:discrepancy}
\mathbb{A}_{s,t} := \mathbb{Y}^{Strat}_{s,t} - \overline{\mathbb{Y}}_{s,t} 
\end{equation}
is skew-symmetric, and indeed setting $\Delta \mathbb{Y}_k^{\mathrm{skew}} = \mathbb{A}_{t_k,t_{k+1}}$ is expected to correct the model misspecification error (see also \cite[Section 8.2]{diehl2016pathwise} and \cite{reich2022frequentist}). For later use, we note that Chen's relation together with $Y = \overline{Y}$ implies 
\begin{equation}
\label{eq:Chen for A}
\mathbb{A}_{s,t} = \mathbb{A}_{0,t} - \mathbb{A}_{0,s}, \qquad s,t \in [0,T]. 
\end{equation}
In practical scenarios where only $(Y^\varepsilon_t)_{0 \le t \le T}$ is available, we are left with the challenge of estimating the correction term $\mathbb{A}_{s,t}$. Fixing a sequence of (determinstic, equidistant) partitions $(p_\delta)$ with $\mathrm{mesh}(p_\delta) \rightarrow 0$ as $\delta \rightarrow 0$, we denote the piecewise linear interpolations associated to $p_\delta$ by $ Y^{(\delta)}$ and $Y^{\varepsilon,(\delta)}$, respectively, and similarly the corresponding (canonical) second-order integrals by $\mathbb{Y}^{(\delta)}$ and $\mathbb{Y}^{\varepsilon,(\delta)}$. In $C^{2\alpha}_2([0,T];\mathbb{R}^{d \times d})$, we see that
\begin{subequations}
\label{eq:eps delta limits}
	\begin{align}
	&\mathbb{Y}^{\varepsilon,(\delta)} \xrightarrow{\varepsilon \rightarrow 0} \mathbb{Y}^{(\delta)} \xrightarrow{\delta \rightarrow 0} \mathbb{Y}^{Strat}  
	\\
	&\mathbb{Y}^{\varepsilon,(\delta)} \xrightarrow{\delta \rightarrow 0} \mathbb{Y}^{\varepsilon} \xrightarrow{\varepsilon \rightarrow 0} \overline{\mathbb{Y}} = \mathbb{Y}^{Strat} - \mathbb{A},
	\end{align}   
\end{subequations}
where the limits as $\delta \rightarrow 0$ follow from \cite[Proposition 2]{coutin2005semi} and \cite[Theorem 5.33]{friz2010multidimensional}. As a consequence, $\mathbb{A}$ can be obtained from $(Y^\varepsilon_t)_{0 \le t \le T}$ as the `commutator'
\begin{equation}
\label{eq:commutator}
\mathbb{A}_{s,t} = \lim_{\varepsilon \rightarrow 0} \lim_{\delta \rightarrow 0} \mathbb{Y}_{s,t}^{\varepsilon,(\delta)} - \lim_{\delta \rightarrow 0} \lim_{\varepsilon \rightarrow 0} \mathbb{Y}_{s,t}^{\varepsilon,(\delta)}, \qquad s,t \in [0,T].
\end{equation}
In practice, a realisation of $(Y_t^\varepsilon)_{0 \le t \le T}$ will only be available for one specific (small) value of $\varepsilon$. To mimic \eqref{eq:commutator}, we may however compute the difference between $\mathbb{Y}^{\varepsilon,(\delta_1)}$ and $\mathbb{Y}^{\varepsilon,(\delta_2)}$ for $\delta_1 \ll \delta_2$, so that $\delta_1$ is small and $\delta_2$ is large in comparison with $\varepsilon$. 
}

\textcolor{red}{Recalling that $(Y_t)_{0 \le t \le T}$ is typically available in the form of a discrete time series $y_0,\ldots,y_n$ (to which we can associate a grid $p_{\delta_1}$ with mesh size $\delta_1 = \Delta t$ coinciding with the grid for the numerical scheme \eqref{eq:ns} and with piecewise linear interpolation $(\widetilde{Y}_t)_{0 \le t \le T}$), we are naturally led to the idea of subsampling the data in order to obtain a coarser grid $p_{\delta_2}$ with mesh size $\delta_2 = \tau \delta_1 = \tau \Delta t$. More precisely, }for a specific time-lag $\tau \in \mathbb{N}_{\ge 1}$, consider the subsampled sequence $y_0, y_\tau, y_{2\tau}, \ldots$, as well as the associated piecewise linear interpolation $(\widetilde{Y}^\tau_t)_{0 \le t \le T}$. The time-lag $\tau$ shall be chosen in such a way that the corresponding area paths
\begin{equation}
    t \mapsto \skewp \left(\int_{0}^{t} \widetilde{Y}_{0,r} \otimes \mathrm{d}\widetilde{Y}_r \right) =: \widetilde{\mathbb{Y}}^{\skewp}_{0,t} \qquad \text{and} \qquad t \mapsto \skewp \left(\int_{0}^{t} \widetilde{Y}^\tau_{0,r} \otimes \mathrm{d}\widetilde{Y}^\tau_r \right) =: \widetilde{\mathbb{Y}}^{\skewp, \tau}_{0,t}
\end{equation}
are `as distinct as possible' \textcolor{red}{(attempting to realise the limiting regimes in \eqref{eq:commutator})}, while maintaining $(\widetilde{Y}_t)_{0 \le t \le T} \approx (\widetilde{Y}^\tau_t)_{0 \le t \le T}$. \textcolor{red}{The latter desideratum is motivated by the fact that the limits in \eqref{eq:eps delta limits} are only valid provided that $\bY$ and $\overline{\bY}$ are over the space path $Y$.}  \textcolor{red}{The comparisons between $\widetilde{\mathbb{Y}}^{\skewp, \tau}$ and $\widetilde{\mathbb{Y}}^{\skewp}$ as well as between $\widetilde{Y}^\tau$ and $\widetilde{Y}$} can be made in supremum norm, for instance. We then set
\textcolor{red}{
\begin{equation}
\label{eq:discrete lift}
\Delta \mathbb{Y}^{\skewp}_k =
(\widetilde{\mathbb{Y}}^{\skewp}_{0,t_{k+1}} - \widetilde{\mathbb{Y}}^{\skewp}_{0,t_{k}}) - 
(\widetilde{\mathbb{Y}}^{\skewp,\tau}_{0,t_{k+1}} - \widetilde{\mathbb{Y}}^{\skewp,\tau}_{0,t_{k}}) \approx \mathbb{A}_{t_k,t_{k+1}}, 
\end{equation}
for the correction in the numerical scheme \eqref{eq:ns}, relying on \eqref{eq:Chen for A} and \eqref{eq:commutator}.}
Equation \eqref{eq:discrete lift} compares the area contributions associated to the original interpolation $\widetilde{Y}$ and the subsampled interpolation $\widetilde{Y}^\tau$. The requirement $(\widetilde{Y}_t)_{0 \le t \le T} \approx (\widetilde{Y}^\tau_t)_{0 \le t \le T}$ is meant to ensure that $\widetilde{Y}$ and $\widetilde{Y}^\tau$ mainly differ at the second-order level $\mathbb{Y}$; visually, the subsampling operation may be understood as `straightening out' $\widetilde{Y}$ and measuring the area \textcolor{red}{difference} accumulated thereby.
\begin{remark}[Relationship to subsampling in multiscale parameter estimation]
As mentioned in Section \ref{sec:parameter estimation}, ideas related to subsampling 
have been considered extensively
in the context of multiscale parameter estimation (without observational noise), see, for instance, 
\cite{abdulle2020drift,azencott2013sub,azencott2010adaptive,gailus2017statistical,gailus2018discrete,kalliadasis2015new,krumscheid2013semiparametric,krumscheid2015data,papavasiliou2009maximum,pavliotis2012parameter}. The method proposed in this section is different in that we use the subsampled paths in order to estimate the L{\'e}vy area \textcolor{red}{correction}, but otherwise input the original data path into the RP-EnKF dynamics \eqref{eq:numerical scheme}.  Moreover, the motivations are distinct: While subsampling in the aforementioned works is used in order to eliminate small-scale fluctuations, our method is specifically designed to estimate L{\'e}vy area \textcolor{red}{correction} terms. \textcolor{red}{Our method may constitute a convenient alternative in complex settings where multiple effects would require competing subsampling frequencies, posing a challenge to tradional subsampling strategies. As our approach ultimately uses the resolution of the original data, we would expect to be able to include effects on multiple scales rather seamlessly into the procedure put forward in this section.} We leave a \textcolor{red}{detailed} exploration of the connection between both methods for future work \textcolor{red}{and refer the reader to the follow-up paper \cite{reich2022frequentist}.}
\label{rem:subsampling}
\end{remark}

\subsection{Physical Brownian motion in a magnetic field}
\label{subsec:magnetic field}

In a first example, we consider a parameter estimation problem where the dynamics of interest is driven by a physical Brownian motion subject to a magnetic field. More precisely, physical Brownian motion $(W^\varepsilon_t)_{t \ge 0}$ is defined in terms of the unique strong solution to the following system of SDEs,
\begin{subequations}
\begin{align}
\mathrm{d}W^\varepsilon_t & = \frac{1}{\varepsilon}  M P^\varepsilon_t \, \mathrm{d}t, \qquad \qquad \qquad  && W^{\varepsilon}_0 = 0, \\
\mathrm{d}P^\varepsilon_t & = - \frac{1}{\varepsilon} MP^\varepsilon_t \, \mathrm{d}t + \mathrm{d}W^0_t, \qquad && P^\varepsilon_0  = 0,
\end{align}
\end{subequations}
where $W_t^\varepsilon, P^\varepsilon_t \in \mathbb{R}^2$, $(W^0_t)_{t \ge 0}$ is a standard (mathematical) two-dimensional Brownian motion, and $\varepsilon \ll 1$ is a small parameter, the limit $\varepsilon \rightarrow 0$ corresponding to the regime of negligible particle mass. Furthermore, the matrix $M$ is given by
\begin{equation}
M = \begin{pmatrix}
1 & \gamma \\
-\gamma & 1
\end{pmatrix},
\end{equation}
with $\gamma \in \mathbb{R}$ being a real-valued parameter associated to the strength of the magnetic field. For fixed $\alpha \in (1/3,1/2)$ and $T>0$, it is known that  $(W^\varepsilon, \mathbb{W}^\varepsilon) \rightarrow (W^0,\mathbb{W}^{\mathrm{phys}}(\gamma ))$ in $\mathscr{C}^\alpha([0,T],\mathbb{R}^d)$ and $L^1$ as $\varepsilon \rightarrow 0$,
where 
\begin{equation}
\mathbb{W}^\varepsilon_{s,t} = \int_s^t W^\varepsilon_{s,r} \otimes \mathrm{d}W^\varepsilon_r 
\end{equation}
denotes the canonical lift, and 
\begin{equation}
\label{eq:pBM RP limit}
    \mathbb{W}^{\mathrm{phys}}_{s,t}(\gamma) = \int_s^t W^{0}_{s,r} \otimes \circ \mathrm{d}W^0_r + (t-s)D,
\end{equation}
with area correction 
\begin{equation}
\label{eq:area correction pBM}
    D = \frac{1}{2} 
    \begin{pmatrix}
    0 & \gamma 
    \\
    - \gamma & 0
    \end{pmatrix},
\end{equation}
see 
\cite{friz2015physical} and \cite[Section 3.4]{friz2020course}. The setting is reminiscent of the passage between underdamped and overdamped Langevin dynamics, see \cite[Section 6.5.1]{pavliotis2014stochastic}
and \cite[Section 2.2]{pavliotis2012parameter}.
Similarly to \cite[Section 8.2]{diehl2016pathwise},
we consider the problem of  estimating the parameter  $\theta \in \mathbb{R}$ in 
\begin{equation}
\label{eq:pBM Z}
\mathrm{d}Z^\varepsilon_t=  \theta f(Z^\varepsilon_t) \, \mathrm{d}t +  \mathrm{d}W^\varepsilon_t, \qquad \qquad Z^\varepsilon_0 = 0,
\end{equation}
given noisy observations of the path $(Z_t^\varepsilon)_{0 \le t \le T}$, that is, given a path $(Y^\varepsilon_t)_{0 \le t \le T}$ of the solution to
\begin{equation}
\label{eq:pBM Y}
\mathrm{d}Y^\varepsilon_t =  \mathrm{d}Z_t^\varepsilon + R^{1/2} \, \mathrm{d}V_t, \qquad Y_0 = 0,
\end{equation}
see Section \ref{sec:parameter estimation}. In \eqref{eq:pBM Z} and \eqref{eq:pBM Y}, we allow for both $\varepsilon = 0$ and $\varepsilon > 0$, that is, we consider the dynamics driven by both mathematical and physical Brownian motion. Note that in the noiseless case $R = 0$ our setting coincides with the one discussed in \cite{diehl2016pathwise,reich2022frequentist}, see also Appendix \ref{app:MLE}. Standard arguments show that both $(Z^\varepsilon, \mathbb{Z}^\varepsilon)$ and $(Y^\varepsilon, \mathbb{Y}^\varepsilon)$ converge in 
$\mathscr{C}^\alpha([0,T],\mathbb{R}^d)$ with a nontrivial area correction akin to \eqref{eq:pBM RP limit}, where in the latter case, $\mathbb{Y}^\varepsilon$ refers to the Stratonovich iterated integrals. As an illustration, we plot sample paths of $t \mapsto Z_t^\epsilon$ and an off-diagonal component of $t \mapsto \int_0^t Z^\varepsilon_s \otimes \mathrm{d}Z_s^\varepsilon$ in Figure \ref{fig:pBM paths}, comparing the cases $\varepsilon = 0$ and $\varepsilon = 10^{-2}$, for the same realisation of (a discretised version of) $(W^0_t)_{0 \le t \le T}$. Crucially, the convergence towards a nontrivially lifted path as expressed in \eqref{eq:pBM RP limit} manifests itself in the offset between the paths in Figure (\ref{fig:area pBM}), \textcolor{red}{corresponding to a component of $\mathbb{A}$ in \eqref{eq:discrepancy}}. Throughout this section, we choose a fine time step of $\Delta t = 10^{-4}$ for all the involved approximations, $\theta_{\mathrm{true}} = 1/2$ for the parameter to be recovered, $\gamma = -2.0$ for the strength of the magnetic field, $R = 0.1$ for the variance of the observation noise, and $f(z_1,z_2) = \textcolor{red}{-}(z_1 - z_2, z_1 + z_2)^\top$ for the drift in \eqref{eq:pBM Z}. 

\begin{figure}
\caption{\small $t \mapsto (Z^\varepsilon_t)_{0 \le t \le T}$ and $t \mapsto \int_0^t Z_s^\varepsilon \otimes \mathrm{d}Z_s^\varepsilon$ for mathematical and physical Brownian motion.}
\label{fig:pBM paths}
     \centering
     \begin{subfigure}[t]{0.45\textwidth}
        \includegraphics[width=\textwidth,valign=t]{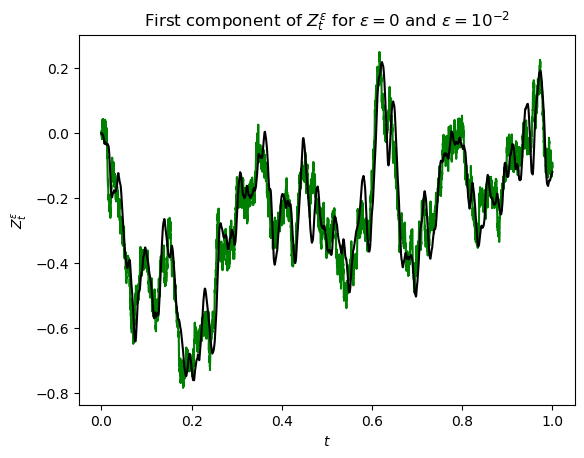}
        \caption{\footnotesize Comparison of the paths $t \mapsto (Z_t^\varepsilon)_{0 \le t \le T}$, for mathematical Brownian motion ($\varepsilon = 0$, green) and physical Brownian motion ($\varepsilon = 10^{-2}$, black). }
     \label{fig:paths pBM}
     \end{subfigure}
     \begin{subfigure}[t]{0.42\textwidth}
    \includegraphics[width=\textwidth,valign=t]{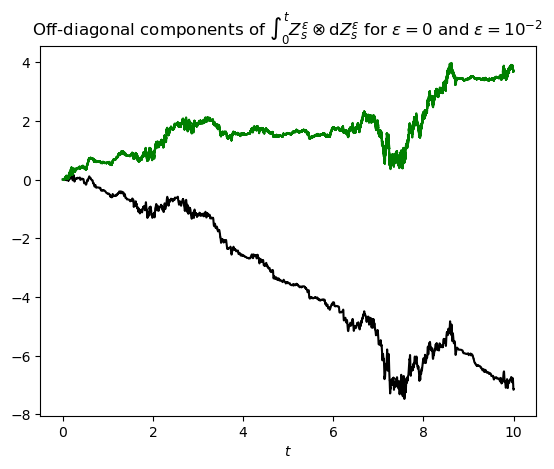}
     \caption{\footnotesize Comparison of the iterated integrals (area processes) $t \mapsto \int_0^t Z_s^\varepsilon \otimes \mathrm{d}Z_s^\varepsilon$.}
    \label{fig:area pBM}
    \end{subfigure}
\end{figure}

To test the robustness of the EnKF scheme \eqref{eq:EnKF} and the RP-EnKF scheme \eqref{eq:numerical scheme}, we generate data according to \eqref{eq:pBM Z} and \eqref{eq:pBM Y} for both $\varepsilon = 0$ (mathematical Brownian motion) and $\varepsilon = 10^{-2}$ (physical Brownian motion). We would like to stress that the filtering methodology (expressed in terms of the schemes \eqref{eq:EnKF} and  \eqref{eq:numerical scheme}) is however based on the model \eqref{eq:signal obs} and therefore tailored to the case $\varepsilon = 0$. 
In Figure \ref{fig:pBM EnKF}, we show the empirical mean of the $\widehat{\theta}$-components for the output of the EnKF-dynamics \eqref{eq:EnKF}, considering both mathematical and physical Brownian motion as drivers in \eqref{eq:pBM Z}. Evidently, the EnKF is not robust, in the sense that it fails to recover the true parameter $\theta_{\mathrm{true}}$ in the case when $\varepsilon = 10^{-2}$, see Figure (\ref{fig:EnKF pBM}).

\begin{figure}
\captionsetup{format=plain}
\caption{\small Output (empirical mean of the $\widehat{\theta}$-components) of the EnKF-scheme \eqref{eq:EnKF}, for data obtained from the dynamics \eqref{eq:pBM Z}  driven by either mathematical of physical Brownian motion. The blue line indicates the true value $\theta_{\mathrm{true}}=\frac{1}{2}$. To suppress sampling error, we plot averaged results computed from $5$ independent runs.} 
\label{fig:pBM EnKF}
     \centering
     \begin{subfigure}[t]{0.45\textwidth}
         \centering \includegraphics[width=\textwidth]{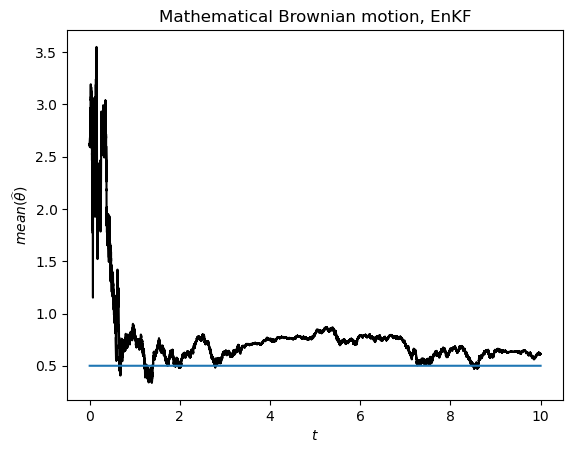}
     \caption{mathematical Brownian motion, $\varepsilon = 0$.}
     \label{fig:EnKF mPM}
     \end{subfigure}
     \begin{subfigure}[t]{0.45\textwidth}
         \centering
    \includegraphics[width=\textwidth]{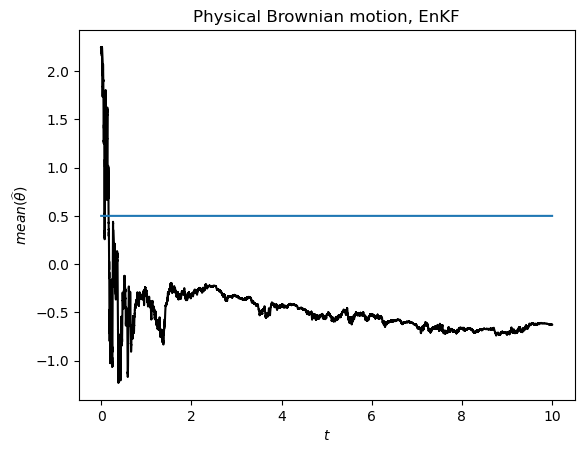}
    \caption{physical Brownian motion, $\varepsilon = 10^{-2}$.}
    \label{fig:EnKF pBM}
    \end{subfigure}
\end{figure}
We proceed by showing the corresponding results for the RP-EnKF scheme defined by \eqref{eq:ns} in Figure \ref{fig:RP-EnKF}, demonstrating the robustness promised by Theorem \ref{thm: convergence geometric}. For the required (discrete-time) lift $\Delta \mathbb{Y}$, we use the construction detailed in Section \ref{sec: numerical examples}. More precisely, in the case of physical Brownian motion, we set the time-lag to $\tau = 700$. In the case of mathematical Brownian motion, we set the skew-symmetric part in \eqref{eq:lift decomp} to zero, $\Delta \mathbb{Y}_k^{\mathrm{skew}} = 0$, corresponding to the choice $\tau = 1$. These choices have been made on the basis of the  subsampled area processes $t \mapsto \mathrm{skew}\left(\int_0^t \widetilde{Y}^{\varepsilon,\tau}_s \otimes \mathrm{d}\widetilde{Y}^{\varepsilon,\tau}_s \right)$ depicted in Figure \ref{fig:pBM subsampling}. More precisely, Figures  (\ref{fig:subsampling pBM area}) and (\ref{fig:mBM area subsampling}) show the dependence $t \mapsto \mathrm{skew}\left(\int_0^t \widetilde{Y}^{\varepsilon,\tau}_s \otimes \mathrm{d}\widetilde{Y}^{\varepsilon,\tau}_s \right)$ with different values of the time-lag $\tau$, for physical Brownian motion ($\varepsilon = 10^{-2}$, Figure (\ref{fig:subsampling pBM area})) and mathematical Brownian motion ($\varepsilon = 0$, Figure (\ref{fig:mBM area subsampling})). While the subsampling only minimally effects the area process associated to the process driven by mathematical Brownian motion (Figure \ref{fig:mBM area subsampling}), we observe a systematic shift in the case of physical Brownian motion (Figure (\ref{fig:mBM area subsampling})), revealing the latent multiscale structure. The difference between the original and subsampled area processes for physical Brownian motion is shown in Figure (\ref{fig:area differences pBM}). As the time-lag $\tau$ increases, said difference approaches the theoretically expected area correction implied by \eqref{eq:pBM RP limit}-\eqref{eq:area correction pBM}. The fact that the difference between the original and the subsampled area process reaches a plateau at around $\tau = 700$ is illustrated in Figure (\ref{fig:L2area}), as opposed to the difference between the original and the subsampled paths, see Figure (\ref{fig:L2path}). Consequently, the choice $\tau = 700$ strikes a balance between separating the original and subsampled area processes as much as possible while maintaining similarity between the original and subsampled paths (as suggested by the discussion motivating \eqref{eq:discrete lift}).   
\begin{figure}
\captionsetup{format=plain}
\caption{\small Output (empirical mean of the $\widehat{\theta}$-components) of the RP-EnKF-scheme \eqref{eq:numerical scheme}, for data obtained from the dynamics \eqref{eq:pBM Z}  driven by either mathematical of physical Brownian motion. The blue line indicates the true value $\theta_{\mathrm{true}}=\frac{1}{2}$. To suppress sampling error, we plot averaged results computed from $5$ independent runs.}
     \centering
     \begin{subfigure}[b]{0.45\textwidth}
         \centering \includegraphics[width=\textwidth]{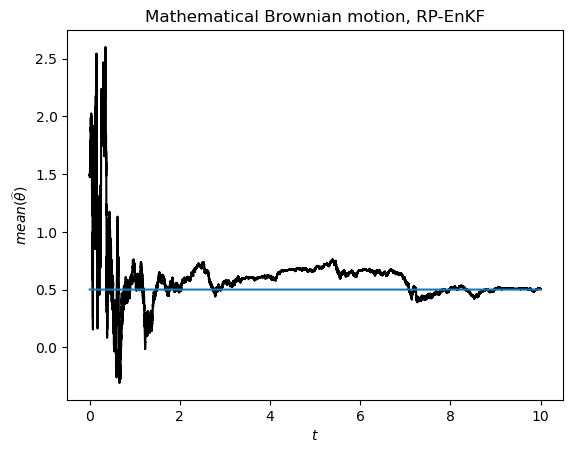}
      \caption{\footnotesize mathematical Brownian motion, $\varepsilon =0$.}
     \label{fig:RP EnKF mBM}
     \end{subfigure}
     \begin{subfigure}[b]{0.45\textwidth}
         \centering
    \includegraphics[width=\textwidth]{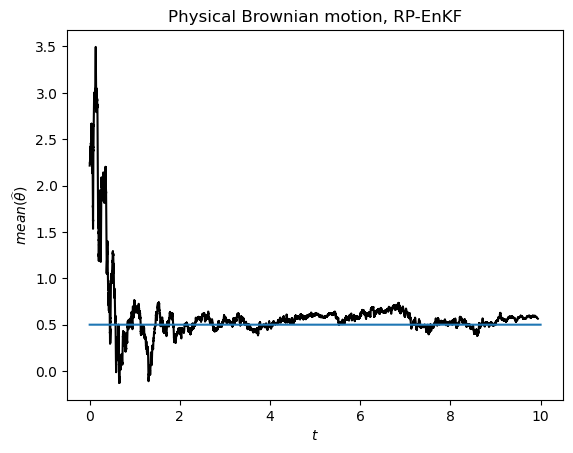} 
    \caption{\footnotesize physical Brownian motion, $\varepsilon = 10^{-2}$.}
     \label{fig:RP EnKF pBM}
     \end{subfigure}
\label{fig:RP-EnKF}
\end{figure}

\begin{figure}
\caption{\small Illustration of the subsampling procedure from Section \ref{sec: numerical examples} for constructing the discrete-time lift $\Delta \mathbb{Y}$}
\label{fig:pBM subsampling}
     \centering
     \begin{subfigure}[t]{0.3\textwidth}
         \centering \includegraphics[width=\textwidth]{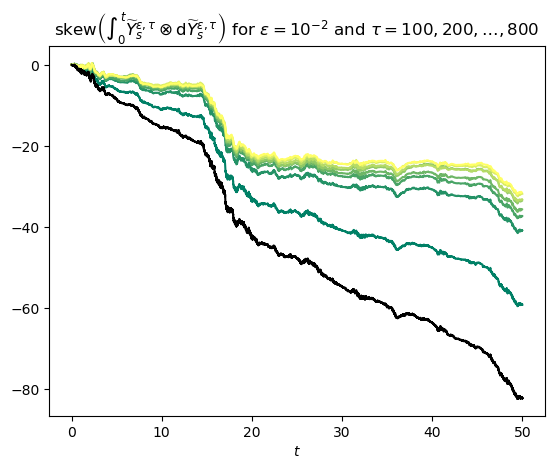}
      \caption{\footnotesize Area processes  for physical Brownian motion ($\varepsilon = 10^{-2}$).}
     \label{fig:subsampling pBM area}
     \end{subfigure}
     \begin{subfigure}[t]{0.3\textwidth}
         \centering
    \includegraphics[width=\textwidth]{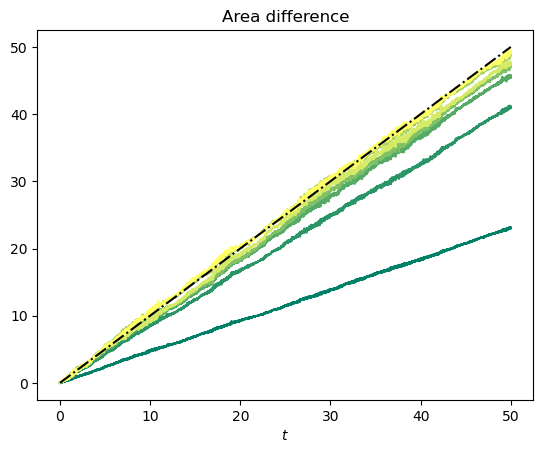} 
    \caption{\footnotesize{Differences of subsampled and original area processes ($\varepsilon=10^{-2}$).}}
     \label{fig:area differences pBM}
     \end{subfigure}
     \begin{subfigure}[t]{0.3\textwidth}
         \centering
    \includegraphics[width=\textwidth]{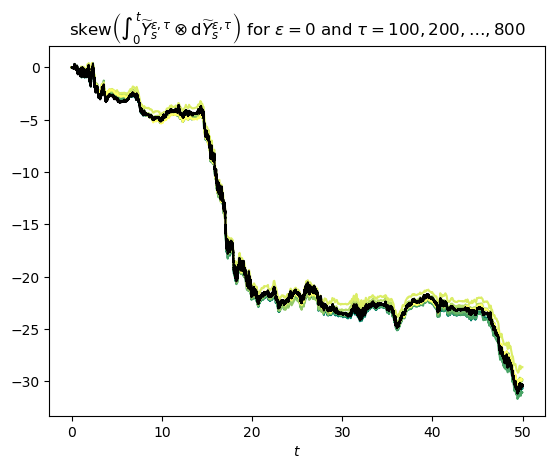} 
    \caption{\footnotesize Area processes for mathematical Brownian motion ($\varepsilon = 0)$.}
     \label{fig:mBM area subsampling}
     \end{subfigure}
     
    \begin{subfigure}[b]{0.35\textwidth}
    \centering
    \includegraphics[width=\textwidth]{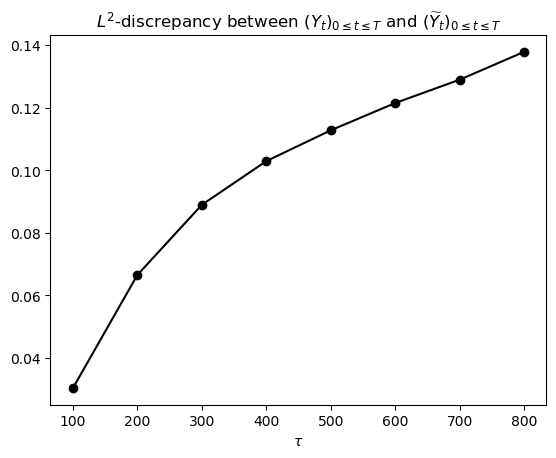} 
    \caption{\footnotesize $L^2$-discrepancy between original and subsampled path.}
    \label{fig:L2path}
    \end{subfigure}
    \begin{subfigure}[b]{0.4\textwidth}
    \centering
    \includegraphics[width=\textwidth]{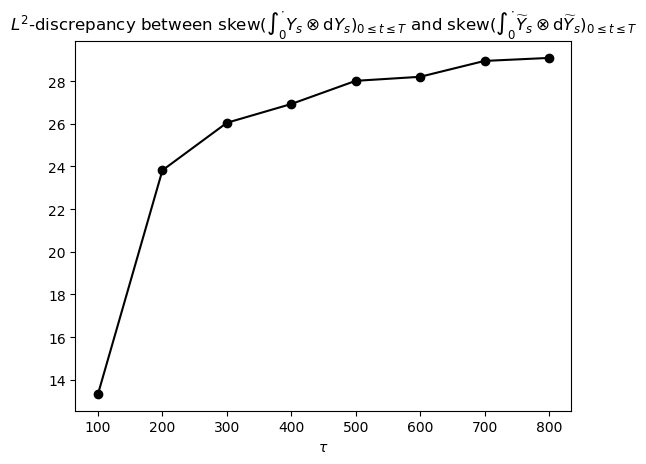} 
    \caption{\footnotesize $L^2$-discrepancy between original and subsampled area process.}
    \label{fig:L2area}
    \end{subfigure}
    \caption*{\footnotesize{\textbf{(\ref{fig:subsampling pBM area})}: Area processes $t \mapsto \skewp \left(\int_0^t \widetilde{Y}^{\varepsilon,\tau}_s \otimes \mathrm{d}\widetilde{Y}^{\varepsilon,\tau}_s\right)$ (off-diagonal component) for physical Brownian motion ($\varepsilon = 10^{-2}$) and time-lags $\tau = 100 \text{(green), \ldots, 800 \text{(yellow)}}$. The area processes associated to the path without subsampling is plotted in black.
    \\
    \textbf{(\ref{fig:area differences pBM})}:
    Differences of original and subsampled area processes in the case of physical Brownian motion ($\varepsilon = 10^{-2}$), that is $t \mapsto \int_0^t \left(\widetilde{Y}^{\varepsilon,\tau}_s \otimes \mathrm{d}Y^{\varepsilon,\tau}_s - Y^\varepsilon_s \otimes \mathrm{d}Y^\varepsilon_s\right)$, in the spirit of \eqref{eq:discrete lift}. The dashed line represents the theoretically expected area correction according to \eqref{eq:pBM RP limit}.
    \\
    \textbf{(\ref{fig:mBM area subsampling})}: Area processes associated to mathematical Brownian  motion ($\varepsilon = 0$), same colour scheme as for Figure (\ref{fig:subsampling pBM area}).}
    \\
    \textbf{(\ref{fig:L2path})}: $L^2$-discrepancy between the original path $(Y_t)_{0 \le t \le T}$ and the subsampled path $(\widetilde{Y}^\tau_t)_{0 \le t \le T}$ as a function of the time-lag $\tau$ for physical Brownian motion ($\varepsilon = 10^{-2}$).
    \\
    \textbf{(\ref{fig:L2area})}: $L^2$-discrepancy between the original area process $\skewp\left( \int_0^\cdot Y_s \otimes \mathrm{d}Y_s \right)_{0 \le t \le T}$ and the subsample area process $\skewp\left( \int_0^\cdot \widetilde{Y}^\tau_s \otimes \mathrm{d}\widetilde{Y}^\tau_s \right)_{0 \le t \le T}$ as a function of the time-lag $\tau$ for physical Brownian motion ($\varepsilon = 10^{-2}$).
    }
\end{figure}

\subsection{Fast chaotic dynamics -- Lorenz-63}
\label{sec:Lorenz}
In this example, we consider the rescaled Lorenz ordinary differential equations \cite{sparrow1982lorenz} for $L^\varepsilon_t = (L_t^{(1),\varepsilon}, L_t^{(2),\varepsilon},L_t^{(3),\varepsilon}) \in \mathbb{R}^3$,
\begin{subequations}
\label{eq:lorenz}
\begin{align}
  \dot{L}_t^{(1),\varepsilon} & = \frac{\sigma}{\varepsilon^2} \left( L^{(2),\varepsilon} - L^{(1),\varepsilon}\right), && L_0^{(1),\varepsilon} = l_0^{(1)}, 
  \\
  \dot{L}^{(2),\varepsilon}_t & = \frac{1}{\varepsilon^2}\left(\rho L^{(1),\varepsilon} - L_t^{(2),\varepsilon} - L_t^{(1),\varepsilon} L_t^{(3),\varepsilon}\right), && L_0^{(2),\varepsilon} = l_0^{(2)},
  \\
  \dot{L}^{(3),\varepsilon}_t & = \frac{1}{\varepsilon^2}\left( L_t^{(1),\varepsilon}L_t^{(2),\varepsilon} - \beta L_t^{(3),\varepsilon} \right),&& L_0^{(3),\varepsilon} = l_0^{(3)},
\end{align}
\end{subequations}
with the standard parameters
$\sigma = 10$, $\rho = 28$ and  $\beta = \frac{8}{3}$ as an example of fast chaotic dynamics approximating Brownian noise (with nontrivial area correction). We would like to stress that the phenomena described in this section can be considered generic across a wide range of fast chaotic deterministic dynamical systems, and refer the reader to \cite{chevyrev2016multiscale} for an overview.

With $\varepsilon = 1$, the system \eqref{eq:lorenz} has originally been proposed as a simplified model for atmospheric convection \cite{lorenz1963deterministic}, and can serve as a prototype for the study of chaotic ODEs. As is well known, \eqref{eq:lorenz} possesses a `strange' chaotic attractor $\Lambda$ equipped with a unique SRB  (Sinai-Ruelle-Bowen) measure\footnote{The notion of SRB measures provides a suitable generalisation of ergodic measures.} $\mu$, see \cite{young2002srb}.
For $\varepsilon \ll 1$, a random intial condition $L^\varepsilon_0 \in \mathbb{R}^3$ and after appropriate centering and rescaling, the solution $(L_t)_{t \ge 0}$ is well approximated by a Brownian motion with a nontrivial area correction in the sense of rough paths:

For H{\"o}lder-continuous observables $v:\mathbb{R}^3 \rightarrow \mathbb{R}^m$ that are $\mu$-centred (that is, $\int_{\mathbb{R}^3} v \, \mathrm{d}\mu = 0$), a functional CLT (or \emph{weak invariance principle}) holds for
\begin{equation}
W^\varepsilon_t := \frac{1}{\varepsilon} \int_0^t v(L_s^\varepsilon) \, \mathrm{d}s,
\end{equation}
assuming that $L_0^\varepsilon$ is initialised randomly according to $\mu$.
More precisely, there exists a Brownian motion $W$ (possibly on an extended probability space) with appropriate covariance $\Sigma \in \mathbb{R}^{m \times m}$ such that $W^\varepsilon \rightarrow W$ weakly in $C([0,T],\mathbb{R}^m)$ as $\varepsilon \rightarrow 0$, see \cite[Theorem 1.5]{holland2007central}\footnote{In fact, the convergence takes place $\mu$-almost surely, see \cite[Theorem 1.1]{holland2007central}. The covariance $\Sigma$ is given in terms of suitable long-time ergodic averages or (under certain conditions) Green-Kubo formulae, see \cite{chevyrev2016multiscale,pavliotis2008multiscale}. }. Moreover, it was shown in  \cite[Theorem 1.6]{balint2018statistical} that a so-called \emph{iterated weak invariance principle} holds for the iterated integrals
\begin{equation}
\mathbb{W}_t^\varepsilon = \int_0^t W_s^\varepsilon \otimes \mathrm{d}W_s^\varepsilon,
\end{equation}
that is $(W^\varepsilon,\mathbb{W}^\varepsilon) \rightarrow (W,\mathbb{W})$ weakly in $C([0,T],\mathbb{R}^m\times\mathbb{R}^{m \times m})$, where
\begin{equation}
\mathbb{W}_{t} = \int_0^t W_{s} \otimes \circ \mathrm{d}W_s + Dt,
\end{equation}
with area correction $D \in \mathbb{R}_{\mathrm{skew}}^{m \times m}$. We refer to \cite[Theorem 4.4]{chevyrev2016multiscale} for the corresponding statement in $p$-variation rough path topology. 

In what follows, we consider $(Z_t^\varepsilon)_{t \ge 0}$ to be driven by $(L^{(1:2),\varepsilon}_t)_{t \ge 0} := (L^{(1),\varepsilon}_t,L^{(2),\varepsilon}_t)_{t \ge 0}$,
\begin{equation}
\label{eq:chaotic pert}
    \mathrm{d}Z_t^\varepsilon = \theta f(Z_t^\varepsilon)\, \mathrm{d}t + \frac{\lambda}{\varepsilon}\mathrm{d}L_t^{(1:2),\varepsilon},
\end{equation}
with the parameter $\theta \in \mathbb{R}$ to be inferred from noisy observations
\begin{equation}
\label{eq:chaotic obs}
    \mathrm{d}Y^\varepsilon_t = \mathrm{d}Z_t^\varepsilon + R^{1/2} \, \mathrm{d}V_t, 
\end{equation}
and $\lambda > 0$ mediating the strength of the chaotic perturbation. Since $\int_{\mathbb{R}^d}(L^{(1)},L^{(2)})^\top \, \mathrm{d}\mu = 0$ according to \cite[Section Section 11.7.2]{pavliotis2008multiscale}, we expect \eqref{eq:chaotic pert} to be well approximated by
\begin{equation}
\label{eq:chaotic reduced}
\mathrm{d}Z_t^0 = \theta f(Z_t^0) \, \mathrm{d}t + G^{1/2}\mathrm{d}W_t,
\end{equation}
in the regime $\varepsilon \ll 1$,
with standard two-dimensional Brownian motion $(W_t)_{t \ge 0}$ and appropriate covariance $G \in \mathbb{R}^{2 \times 2}$. Replacing \eqref{eq:chaotic pert} by \eqref{eq:chaotic reduced} is often a desirable simplification both computationally and conceptually, see, for instance, 
\cite{givon2004extracting},
\cite[Section 11.7.2]{pavliotis2008multiscale}, and \cite[Section 3.1]{wouters2019stochastic}.
To apply our methodology, we need to presuppose $G$; estimates can be obtained using the approaches suggested in \cite[Example 6.2]{givon2004extracting} or \cite{krumscheid2013semiparametric}, for instance. Here we use the value $G^{1/2}_{11} = 0.13$ reported in \cite[Section 3.2.6]{krumscheid2013semiparametric} for $\lambda = \frac{2}{45}$ and note that $W^1 = W^2$ almost surely by a short calculation using \eqref{eq:lorenz}. To test the RP-EnKF, we simulate data according to \eqref{eq:lorenz}, \eqref{eq:chaotic pert} and \eqref{eq:chaotic obs} using an Euler-Maruyama discretisation with time step $\Delta t= 10^{-5}$, an observation noise level of $R = 0.01$, a  parameter value of $\theta_{\mathrm{true}} = 0.5$ to be recovered. Furthermore, we set $\varepsilon = 0.05$ and $f(z_1,z_2) = (z_1- z_2, z_1 + z_2)^\top$. The RP-EnKF is set up according to the model \eqref{eq:chaotic reduced} with observations \eqref{eq:chaotic obs}, using the scheme \eqref{eq:numerical scheme}. As an illustration for the challenges that are posed by the attempt to incorporate data from \eqref{eq:chaotic pert}-\eqref{eq:chaotic obs} into a model of the form \eqref{eq:chaotic reduced}, we plot the output of the EnKF scheme \eqref{eq:EnKF} in Figure (\ref{fig:EnKF Lorenz}), noting that it fails to recover the correct parameter value $\theta_{\mathrm{true}}$. Figures (\ref{fig:Lorenz_RP_1}) and (\ref{fig:Lorenz_RP_500}) show the output of the RP-EnKF scheme \eqref{eq:numerical scheme}, with $\tau = 1$ (implying $\Delta \mathbb{Y}_k^{\skewp} = 0$) and $\tau = 500$, respectively. We see that the area correction obtained through the subsampling procedure is necessary to obtain satisfactory numerical results. The time-lag $\tau = 500$ has been determined in the same way as in Section \ref{subsec:magnetic field}; in particular, plots for the subsampled area processes are qualitatively similar to Figures \ref{fig:pBM subsampling} and \ref{fig:area differences pBM} but are omitted here for the sake of brevity.
\begin{figure}
\captionsetup{format=plain}
\caption{\small Output (empirical mean of the $\widehat{\theta}$-components) EnKF and RP-EnKF (with and without area correction), for data obtained from the dynamics \eqref{eq:chaotic pert} perturbed by fast chaotic noise (Lorenz-63). The blue line indicates the true value of $\theta$. To suppress sampling error, average results are plotted computed from $5$ independent runs.} 
\label{fig:Lorenz}
     \centering
     \begin{subfigure}[b]{0.3\textwidth}
         \centering \includegraphics[width=\textwidth]{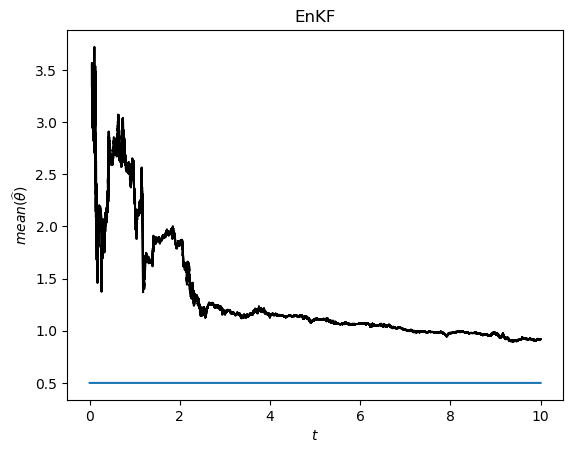}
         \caption{\footnotesize EnKF according to \eqref{eq:EnKF}.}
     \label{fig:EnKF Lorenz}
     \end{subfigure}
     \begin{subfigure}[b]{0.3\textwidth}
         \centering
    \includegraphics[width=\textwidth]{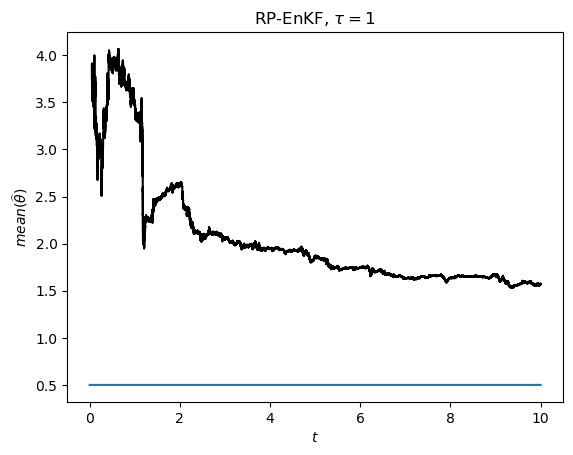} 
    \caption{\footnotesize RP-EnKF without area correction.}
    \label{fig:Lorenz_RP_1}
    \end{subfigure}
    \begin{subfigure}[b]{0.3\textwidth}
         \centering
    \includegraphics[width=\textwidth]{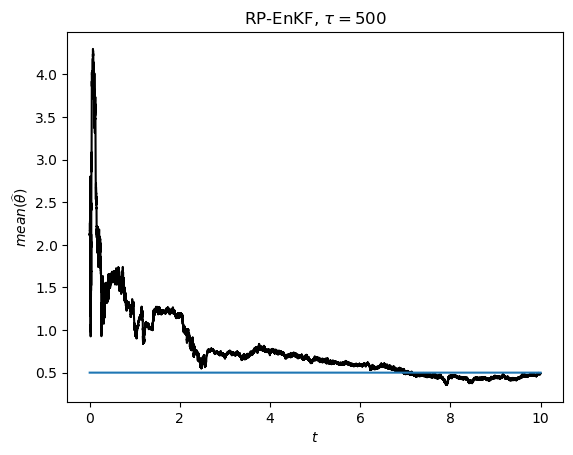} 
    \caption{RP-EnKF with area correction.}
    \label{fig:Lorenz_RP_500}
    \end{subfigure}
\end{figure}

\subsection{Homogenisation in a two-scale potential}
\label{sec:two-scale potential}

In our final example, we consider the motion of a Brownian particle in a (rugged) two-scale potential,
\begin{equation}
\label{eq:twoscale}
\mathrm{d}Z_t^\varepsilon = - \theta \nabla V(Z_t^\varepsilon) \, \mathrm{d}t - \frac{1}{\varepsilon} \nabla p \left(\frac{Z_t^\varepsilon}{\varepsilon} \right) \mathrm{d}t + \sqrt{2 \sigma} \, \mathrm{d}W_t,
\end{equation}
where \textcolor{red}{$p \in C^{\infty}(\mathbb{R}^2;\mathbb{R})$, $p(x_1,x_2) = p_1(x_1) + p_2(x_2)$}, is an $L$-periodic function in both directions, that is, $p_i(x + L) = p_i(x)$, for all $x \in \mathbb{R}$ and $i=1,2$, modelling small-scale fluctuations around the potential $V$. It is well known that for $T>0$, the law of the solution $(Z_t^\varepsilon)_{t \ge 0}$ converges weakly in $C([0,T];\mathbb{R}^d)$ to the law associated to
\begin{equation}
\label{eq:homogenised}
\mathrm{d}Z_t = - \theta \mathcal{K} \nabla V(Z_t,\theta) \, \mathrm{d}t + \sqrt{2 \sigma \mathcal{K}} \, \mathrm{d}W_t,
\end{equation}
where $\mathcal{K} = \text{diag}(L^2/(C_1 \widehat{C}_1),L^2/(C_2 \widehat{C}_2))$, with 
\begin{equation}
\label{eq:Z integrals}
C_i = \int_0^L e^{-\frac{p_i(y)}{\sigma}} \, \mathrm{d}y, \qquad \widehat{C}_i = \int_0^L e^{\frac{p_i(y)}{\sigma}}\, \mathrm{d}y,
\end{equation}
see \cite{pavliotis2007parameter} and \cite[Chapter 11]{pavliotis2008multiscale}. Similar results in a rough-path context can be found in \cite{lejay2003importance}. The homogenised dynamics \eqref{eq:homogenised} encapsulates the rugged landscape described by $p$ in the diffusion-mass matrix $\mathcal{K}$. Like in the previous experiments, we consider the task of estimating the parameter $\theta$ from noisy observations of \eqref{eq:twoscale}, using the EnKF and RP-EnKF based on \eqref{eq:homogenised}. We choose $V(z) = \frac{1}{2}|z|^2$, $p_1(x) = \cos(x)$, $p_2(x) = \frac{1}{2} \cos(x)$ and $\sigma = 1$. Data from the two-scale dynamics \eqref{eq:twoscale} is simulated for $\varepsilon = 10^{-2}$, $\theta_{\mathrm{true}} = 1$ using an Euler-Maruyama discretisation with time step $\Delta t = 10^{-4}$. With the same values for $\theta_{\mathrm{true}}$, $\Delta t$ and $\sigma$, we simulate data from the reduced model \eqref{eq:homogenised}, where $\mathcal{K} \approx \text{diag}(0.62386,0.884176)$ has been obtained by numerical integration in \eqref{eq:Z integrals}.
Both \eqref{eq:twoscale} and \eqref{eq:homogenised} are perturbed by noise according to  \eqref{eq:chaotic obs} with $R = 10^{-2}$ and, as in Sections \ref{subsec:magnetic field} and \ref{sec:Lorenz}, the resulting observation paths are used in the EnKF- and RP-EnKF schemes (see equations \eqref{eq:EnKF} and \eqref{eq:numerical scheme}) to estimate $\theta_{\mathrm{true}}$. We display the results for the means of $\widehat{\theta}$ over time using the EnKF and the RP-EnKF in Figures \ref{fig:EnKF-homo} and \ref{fig:RP-EnKF-homo}, respectively.

Clearly, the RP-EnKF deals adequately with multiscale data, while the EnKF fails to recover the true parameter $\theta_{\mathrm{true}}$ in this setting. We would like to stress that the RP-EnKF scheme has been implemented without area correction, that is imposing $\mathbb{Y}_k = \mathbb{Y}_k^{\sym}$ as defined in \eqref{eq:sym from data}. The choice $\mathbb{Y}_k^{\skewp} = 0$ is motivated by a plot analogous and qualitatively similar to Figure \ref{fig:mBM area subsampling} (omitted due to space considerations), showing that subsampling does not indicate substantial L{\'e}vy area \textcolor{red}{correction} terms (intuitively, the dynamics \eqref{eq:twoscale} does not contain significant `rotational' contributions in the regime $\varepsilon \rightarrow 0$).
\begin{figure}
\captionsetup{format=plain}
\caption{\small Output (empirical mean of the $\widehat{\theta}$-components) of the EnKF-scheme \eqref{eq:EnKF}, for data obtained from the reduced model \eqref{eq:homogenised} (left) and the multiscale model \eqref{eq:twoscale} (right). The blue line indicates the true value $\theta_{\mathrm{true}}=1$. To suppress sampling error, we plot averaged results computed from $5$ independent runs.}
\label{fig:EnKF-homo}
     \centering
     \begin{subfigure}[b]{0.45\textwidth}
         \centering \includegraphics[width=\textwidth]{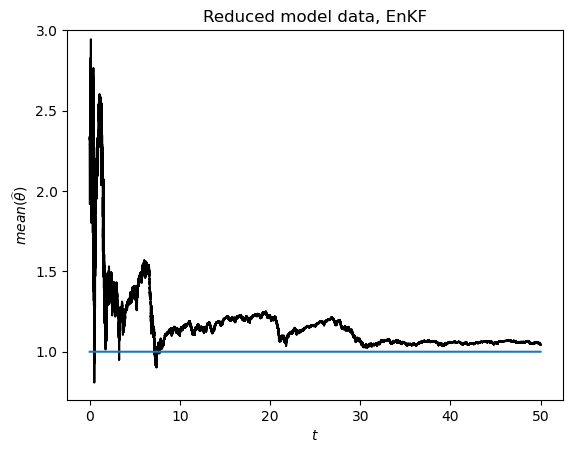}
      \caption{\footnotesize Data obtained from the reduced model \eqref{eq:homogenised}.}
     \label{fig:reducedEnKF}
     \end{subfigure}
     \begin{subfigure}[b]{0.45\textwidth}
         \centering
    \includegraphics[width=\textwidth]{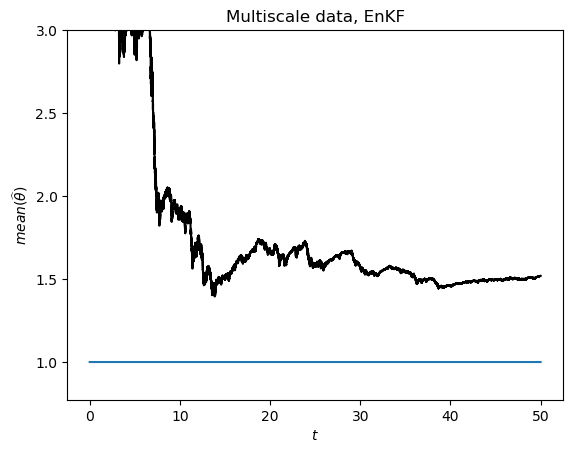} 
    \caption{\footnotesize Data obtained from the multiscale model \eqref{eq:twoscale}.}
     \label{fig:MultiEnKF}
     \end{subfigure}
\end{figure}

\begin{figure}
\captionsetup{format=plain}
\caption{\small Output (empirical mean of the $\widehat{\theta}$-components) of the RP-EnKF-scheme \eqref{eq:numerical scheme}, for data obtained from the reduced model \eqref{eq:homogenised} (left) and the multiscale model \eqref{eq:twoscale} (right). No area correction is used, that is, $\mathbb{Y}_k^{\skewp} = 0$. The blue line indicates the true value $\theta_{\mathrm{true}}=1$. To suppress sampling error, we plot averaged results computed from $5$ independent runs.}
\label{fig:RP-EnKF-homo}
     \centering
     \begin{subfigure}[b]{0.45\textwidth}
         \centering \includegraphics[width=\textwidth]{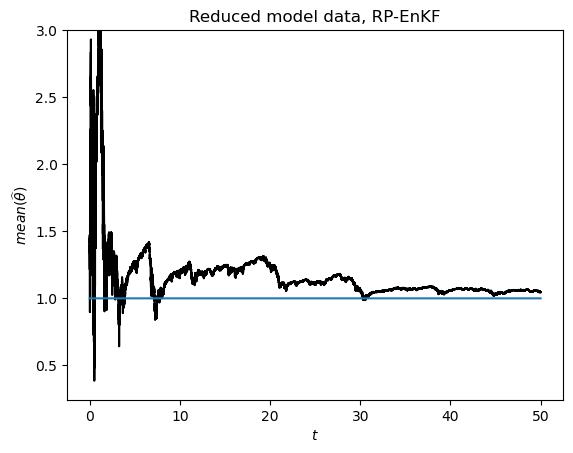}
      \caption{\footnotesize Data obtained from the reduced model \eqref{eq:homogenised}.}
     \label{fig:reducedRP}
     \end{subfigure}
     \begin{subfigure}[b]{0.45\textwidth}
         \centering
    \includegraphics[width=\textwidth]{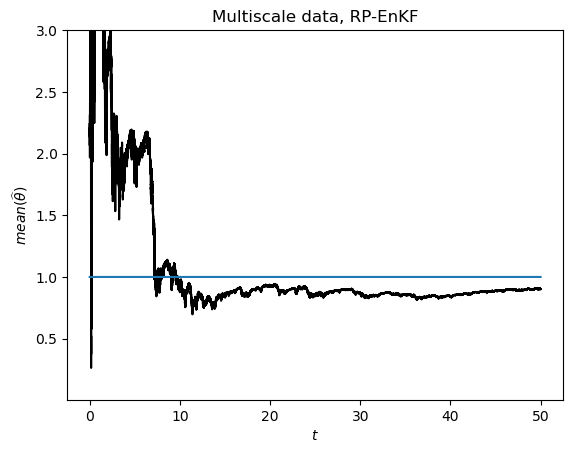} 
    \caption{\footnotesize Data obtained from the multiscale model \eqref{eq:twoscale}.}
     \label{fig:MultiRP}
     \end{subfigure}
\end{figure}

\vspace{0.5cm}

\textbf{Acknowledgements.}
This research has been partially funded by 
 Deutsche Forschungsgemeinschaft (DFG) through the grant 
 CRC 1114 \lq Scaling Cascades in Complex Systems\rq \,(project A02, project number 235221301). 
 This work was commenced while T.N. was at TU Berlin, supported by DFG Research Unit FOR2402 ``Rough paths and SPDEs'' and M.C. was at WIAS Berlin, supported by the Einstein Center Berlin, ECMath Project “Stochastic methods for the analysis of lithium-ion batteries”.
 We would like to thank R. Chew, A. Djurdjevac,  G. Hastermann and R. Klein for very stimulating discussions.  N.N. would like to thank M. Engel for organising a seminar on homogenisation at Imperial College London in 2016 leading to Section \ref{sec:Lorenz}, as well as for sharing his insight into the topic.
\appendix

\section{From the filter to the McKean-Vlasov}
\label{app:proof McKean}
In this section we prove the formal connection between the filtering problem and the McKean-Vlasov equation.
In order to prove Proposition \ref{prop:McKean}, we need a few preparations. Let us define
\begin{equation*}
M_t = \int_0^t h(X_s) \cdot C^{-1} \, \mathrm{d}Y_s,
\end{equation*}
and the likelihood
\begin{equation*}
l_t = \exp \left( M_t - \frac{1}{2} \langle M \rangle_t \right) = \exp \left( \int_0^t h(X_s) \cdot C^{-1} \, \mathrm{d}Y_s - \frac{1}{2} \int_0^t h(X_s) \cdot C^{-1} h(X_s) \, \mathrm{d}s\right).
\end{equation*}
We can now introduce the unnormalised filtering measures
\begin{equation*}
\rho_t[\phi] = \mathbb{E}\left[ \phi(X_t) l_t \vert \mathcal{Y}_t\right], \qquad t \ge 0.
\end{equation*}
The measures $\rho_t$ satisfy the Zakai equation:
\begin{proposition}[Zakai equation]
	\label{prop:Zakai}
	The evolution of $(\rho_t)_{t \ge 0}$ is given by
	\begin{equation}
	\label{eq:Zakai}
	\rho_t[\phi] = \rho_0[\phi] + \int_0^t \rho_s [\mathcal{L}\phi] \, \mathrm{d}s + \int_0^t \rho_s [\phi h ] \cdot C^{-1} \, \mathrm{d}Y_s + \int_0^t \rho_s[\nabla \phi] \cdot B \, \mathrm{d}Y_s,
	\qquad
	\mathbb{P}-a.s,
	\qquad \forall \phi\in C^2_b(\mathbb{R}^D),
	\end{equation}
	where 
	\begin{equation*}
	\mathcal{L}\varphi = f \cdot \nabla\phi + \frac{1}{2} \operatorname{Trace}(G\nabla^2\phi)
	\end{equation*}
	is the generator associated to \eqref{eq:signal}. The filtering measures $(\pi_t)_{t \ge 0}$ can be recovered from the Kalliampur-Striebel formula (or Bayes theorem)
	\begin{equation}
	\label{eq:KS}
	\pi_t [\phi] = \frac{\rho_t [\phi]}{\rho_t[\mathbf{1}]}.
	\end{equation}
\end{proposition}
\begin{proof}
	This follows from the results in \cite[Section 3.8]{bain2008fundamentals}.
\end{proof}
Next, we write the Zakai equation in its Stratonovich form:
\begin{lemma}
The Stratonovich-version of the Zakai equation is given by
    \begin{subequations}
		\label{eq:Zakai Strato}
		\begin{align*}
		\rho_t[\phi] & = \rho_0[\phi] + \int_0^t \rho_s [\mathcal{L}\phi] \, \mathrm{d}s +  \int_0^t \rho_s [\phi h ]  C^{-1} \circ \mathrm{d}Y_s + \int_0^t \rho_s[\nabla \phi] \cdot B \circ \mathrm{d}Y_s
		\\
		& - \frac{1}{2} \int_0^t  \left(\rho_s [(\phi h)^T C^{-1} h] + \rho_s[\operatorname{Trace}(\nabla(\phi h)\cdot B) + \nabla \phi \cdot B h] \right) \mathrm{d}s 
		\\
		 & - \frac{1}{2} \int_0^t   \rho_s [ \operatorname{Trace}(D^2\phi \ {\underbrace{G^{1/2} U^T C^{-1} U G^{1/2}}_{BCB^T}}) ]\mathrm{d}s.
		\end{align*}
	\end{subequations}
\end{lemma}

\begin{proof}
Using \eqref{eq:Zakai}, we see that for $i=1,\dots, d$
\begin{subequations}
	\begin{align*}
	\rho_t[\phi h_i] = \sum_{j,k=1}^d\int_0^t \rho_s [\phi h_i h_j] (C^{-1})^{jk} \, \mathrm{d}Y_s^k + \sum_{r=1}^D \sum_{k=1}^d\int_0^t \rho_s [\partial_r(\phi h_i)] B^{rk} \, \mathrm{d}Y_s^k + FV,
	\end{align*}
\end{subequations}
where $FV$ stands for a contribution of finite variation. Recalling $\langle Y^{k},Y^{\bar{k}} \rangle_t = C^{k\bar{k}} t$, for $k,\bar{k} = 1,\dots,d$, we obtain
\begin{equation*}
 \int_0^t \rho_s [\phi h ]  C^{-1}  \mathrm{d}Y_s = \int_0^t \rho_s [\phi h ]  C^{-1}  \circ \mathrm{d}Y_s - \frac{1}{2} \int_0^t 
 \sum_{i=1}^d\left( 
 \sum_{j=1}^d\rho_s[\phi h_i h_j] (C^{-1})^{ij} + \sum_{\textcolor{red}{j}=1}^D\rho_s[\partial_j(\phi h_i) B^{ji}  ]\right) \mathrm{d}s. 
\end{equation*}	
Similarly, for $r=1,\dots,D$,
\begin{equation*}
\rho_t[\partial_r \phi] = \sum_{j,k=1}^d\int_0^t \rho_s[(\partial_r \phi)h_j ] (C^{-1})^{jk} \, \mathrm{d}Y_s^k + \sum_{p=1}^D\sum_{k=1}^d\int_0^t \rho_s[\partial_r \partial_p \phi] B^{pk} \, \mathrm{d}Y_s^k + FV, 
\end{equation*}
implying
\begin{equation*}
\int_0^t \rho_s[\nabla \phi]  B  \,  \mathrm{d}Y_s = \int_0^t \rho_s[\nabla \phi]  B \circ \mathrm{d}Y_s - \frac{1}{2} \int_0^t \sum_{r=1}^D \left(\sum_{i=1}^d\rho_s[(\partial_r \phi)h_i] B^{ri}  + \sum_{p=1}^D\rho_s [\partial_{\textcolor{red}{r}} \partial_p \phi] (BCB^T)^{rp} \right) \mathrm{d}s.
\end{equation*}
\end{proof}

We now proceed to the proof of Proposition \ref{prop:McKean}:
\begin{proof}[Proof of Proposition \ref{prop:McKean}]
Let us define 
\begin{equation*}
\widehat{\rho}_t[\phi] = \mathbb{E} \left[\phi(\widehat{X}_t) l_t \vert \mathcal{Y}_t \right], \qquad t \ge 0,
\end{equation*}
which is the unnormalised filtering measure associated to \eqref{eq: true McKean Vlasov}.
Since $X_t$ and $\widehat{X}_t$ are independent given $\mathcal{Y}_t$, we have that 
\begin{equation*}
\widehat{\rho}_t[\phi] = \mathbb{E} \left[\phi(\widehat{X}_t)  \vert \mathcal{Y}_t \right] \mathbb{E} \left[l_t \vert \mathcal{Y}_t\right] = \widehat{\pi}_t[\phi] \rho_t[\mathbf{1}].
\end{equation*}
From \eqref{eq:KS} it thus follows that $\pi_t =  \widehat{\pi}_t$ is equivalent to $\rho_t = \widehat{\rho}_t$. In the following we therefore compute the evolution of $\widehat{\rho}_t$.
Notice first that 
\begin{equation*}
M_t = \int_0^t h(X_s) \cdot C^{-1} \, \mathrm{d}Y_s = \int_0^t h(X_s)\cdot C^{-1} \circ \mathrm{d}Y_s - \frac{1}{2} \langle h(X) C^{-1},Y\rangle_t, 
\end{equation*}
where, using $B= G^{\frac{1}{2}} U^T C^{-1}$,
\begin{equation*}
\langle h(X) C^{-1},Y\rangle_t = \sum_{i,j = 1}^d\langle h^i(X) (C^{-1})_{ij},Y^j\rangle_t = \sum_{r=1}^D\sum_{i=1}^d\int_0^t \partial_r h^i(X_s) B^{ri} \, \mathrm{d}s
= \int_0^t \operatorname{Trace}(\nabla h(X_s) \cdot B) \, \mathrm{d}s.
\end{equation*}
From $\mathrm{d}l_t = l_t \, \mathrm{d}M_t$ we have
\begin{subequations}
	\begin{align*}
	l_t & = 1 + \int_0^t l_s \circ \mathrm{d}M_s - \frac{1}{2} \int_0^t l_s \, \mathrm{d}\langle M \rangle_s
	\\
	& = 1+ \int_0^t l_s h(X_s) \circ C^{-1} \, \mathrm{d}Y_s - \frac{1}{2}\int_0^t l_s \operatorname{Trace}(\nabla h(X_s)\cdot B) \, \mathrm{d}s -  \frac{1}{2} \int_0^t l_s h(X_s) \cdot C^{-1} h(X_s) \, \mathrm{d}s, 
	\end{align*}
\end{subequations}
implying
\begin{subequations}
	\label{eq:phi Xhat}
	\begin{align}
	\phi(\widehat{X}_t) l_t =& \phi(\widehat{X}_0)  + \int_0^t \phi(\widehat{X}_s) l_s h(X_s)\cdot  C^{-1} \circ \mathrm{d}Y_s - \frac{1}{2} \int_0^t \phi(\widehat{X}_s) l_s h(X_s) \cdot C^{-1} h(X_s) \, \mathrm{d}s \\
	& - \frac{1}{2}\int_0^t \phi(\widehat{X}_s) l_s \operatorname{Trace}(\nabla h(X_s) \cdot B) \, \mathrm{d}s + \int_0^t l_s \nabla \phi(\widehat{X}_s) \circ \mathrm{d}\widehat{X}_s.
	\end{align}
\end{subequations}
The last term satisfies
\begin{subequations}
	\begin{align*}
	\int_0^t l_s \nabla \phi(\widehat{X}_s) \circ \mathrm{d}\widehat{X}_s & = \int_0^t l_s \nabla \phi(\widehat{X}_s) \cdot f(\widehat{X}_s) \, \mathrm{d}s + \int_0^t l_s \nabla \phi(\widehat{X}_s) \circ G^{1/2} \mathrm{d}\widehat{W}_s 
	\\
	&  + \int_0^t l_s \nabla \phi(\widehat{X}_s) K_s(\widehat{X}_s) C^{-1} \circ \mathrm{d}Y_s - \int_0^t l_s \nabla \phi(\widehat{X}_s) \cdot K_s(\widehat{X}_s) C^{-1}h(\widehat{X}_s) \, \mathrm{d}s
	\\
	& - \int_0^t l_s \nabla \phi(\widehat{X}_s) \cdot K_s(\widehat{X}_s) C^{-1} \circ \left(U \, \mathrm{d}\widehat{W}_s + R^{1/2} \, \mathrm{d}\widehat{V}_s\right) 
	 + \int_0^t l_s \nabla \phi(\widehat{X}_s) \cdot \Xi_s(\widehat{X}_s) \, \mathrm{d}s.
	\end{align*}
\end{subequations}
Next, we convert the terms involving $\widehat{V}$ and $\widehat{W}$ back to their It{\^o}-form,
\begin{subequations}
\begin{align*}
\int_0^t l_s \nabla \phi(\widehat{X}_s) \circ G^{1/2} \mathrm{d}\widehat{W}_s  & = \int_0^t l_s \nabla \phi(\widehat{X}_s) G^{1/2} \mathrm{d}\widehat{W}_s
\\
& + \frac{1}{2} \int_0^t l_s \operatorname{Trace}(D^2\phi(\widehat{X}_s) G) \, \mathrm{d}s - \frac{1}{2} \int_0^t l_s \operatorname{Trace}\left(D^2\phi(\widehat{X}_s) (K_s(\widehat{X}_s)B^T)\right) \mathrm{d}s.
\end{align*}
\end{subequations}
Here we used the fact that, for every $i,k = 1,\dots,D$ we have $\left\langle l, (G^{1/2})^{ik} \widehat{W}^k \right\rangle_t = 0$ and
\begin{equation*}
\left\langle \widehat{X}^j, (G^{1/2})^{ik} \widehat{W}^k \right\rangle_t = G^{ij} t - \int_0^t \left(K_s(\widehat{X}_s)C^{-1}UG^{1/2} \right)^{ji} \, \mathrm{d}s,
\end{equation*}
where $G^{1/2}U^TC^{-1}= B$. Similarly,
\begin{subequations}
	\begin{align*}
& \int_0^t l_s\nabla \phi(\widehat{X}_s) \cdot K_s(\widehat{X}_s) C^{-1} \circ \left(U \, \mathrm{d}\widehat{W}_s + R^{1/2} \, \mathrm{d}\widehat{V}_s\right) 
= \int_0^t l_s \nabla \phi(\widehat{X}_s) \cdot K_s(\widehat{X}_s) C^{-1} \left(U \, \mathrm{d}\widehat{W}_s + R^{1/2} \, \mathrm{d}\widehat{V}_s\right)
 \\
 & + \frac{1}{2} \int_0^t l_s \operatorname{Trace}[ \nabla (\nabla \phi \cdot K_s) (\widehat{X}_s)\cdot C^{-1}  (G^{1/2}U^T  -  K_s(\widehat{X}_s) )] \, \mathrm{d}s,  
\end{align*}
\end{subequations}
using $C = UU^T + R$ and, for $k=1,\dots,D$ and $l=1,\dots,d$,
\begin{equation*}
\langle \widehat{X}^k, (U\widehat{W} + R^{1/2} \widehat{V})^l \rangle_t = (G^{1/2} U^T)^{kl} t - \int_0^t K_s^{kl} (\widehat{X}_s) \, \mathrm{d}s.
\end{equation*}
We now take the conditional expectation in \eqref{eq:phi Xhat}. Note that the conditional expectation w.r.t. $\mathcal{Y}$ commutes with the integration in $\mathrm{d}Y$ and that the conditional expectation of the integrals in $\mathrm{d}W$ and $\mathrm{d}V$ vanishes because $(W,V)$ are independent from $\mathcal{Y}$, see \cite[Appendix B]{coghi2019stochastic}. We obtain
\begin{subequations}
\label{eq:Zakai feedback}
\begin{align}
\widehat{\rho}_t [\phi]  & = \widehat{\rho}_0[\phi] + \int_0^t \widehat{\rho}_s[\phi] \pi_s [h] \cdot  C^{-1} \circ \mathrm{d}Y_s - \frac{1}{2} \int_0^t \widehat{\rho}_s[\phi] \pi_s[h \cdot C^{-1} h] \, \mathrm{d}s
\\ 
\label{eq:short Strato correction}
& + \int_0^t \widehat{\rho}_s [\mathcal{L} \phi] \, \mathrm{d}s 
- \frac{1}{2} \int_0^t \widehat{\rho}_s[\operatorname{Trace}(D^2 \phi K_s B^T)] \mathrm{d}s 
+ \int_0^t \widehat{\rho}_s[\nabla \phi \cdot K_s] C^{-1} \circ \mathrm{d}Y_s
\\
\label{eq:long Strato correction}
& - \int_0^t \widehat{\rho}_s [\nabla \phi \cdot K_s C^{-1}h] \, \mathrm{d}s
- \frac{1}{2} \int_0^t \widehat{\rho}_s[ \operatorname{Trace}(\nabla(\nabla\phi \cdot K_s) \cdot C^{-1} (G^{1/2}U^T  -  K_s ))] \, \mathrm{d}s.
\\
& - \frac{1}{2} \int_0^t \widehat{\rho}_s[\phi] \pi_s[\operatorname{Trace}(Dh B)] \, \mathrm{d}s + \int_0^t \widehat{\rho}_s[\nabla\phi \cdot \Xi]\,\mathrm{d}s
\end{align}
\end{subequations}

Importantly, we have used the fact that $X_t$ and $\widehat{X}_t$ are independent given $\mathcal{Y}_t$, and also that $\widehat{W}_t$ and $\widehat{V}_t$ are independent from $\mathcal{Y}_t$.
The next step is to compare \eqref{eq:Zakai Strato} and \eqref{eq:Zakai feedback}. The $\mathrm{d}Y$-contributions agree if and only if
\begin{equation*}
\rho_s[\phi h] + \rho_s [\nabla \phi] \cdot BC= \widehat{\rho}_s[\phi] \pi_s[h] + \widehat{\rho}_s[\nabla \phi \cdot K_s],
\end{equation*}
which we recognise (after identifying $\widehat{\pi} = \pi$ and $\widehat{\rho} = \rho$) to be a weak version of \eqref{eq:Poisson K intro}.

To compare the $\mathrm{d}s$-contributions, let us first manipulate the second term in \eqref{eq:long Strato correction} using \eqref{eq:Poisson K intro}, 
\begin{subequations}
	\begin{align*} 
	& \widehat{\rho}_s[ \partial_m (\partial_i \phi K_s^{ij}) (C^{-1})^{jk} (G^{1/2}U^T  -  K_s )^{mk}] 
	 = - \int (\partial_i \phi K_s^{ij}) (C^{-1})^{jk} \partial_m \left( \widehat{\rho}_s (G^{1/2}U^T  -  K_s )^{mk} \right) \mathrm{d}x
	\\
	 & = - \int (\partial_i \phi K_s^{ij}) (C^{-1})^{jk} \left( h^k - \widehat{\pi}_s[h^k]  \right) \mathrm{d}\widehat{\rho}_s
	 = -\widehat{\rho}_s[ \nabla \phi \cdot K_s C^{-1} \left( h - \widehat{\pi}_s[h]  \right)].
	\end{align*}
\end{subequations}
Similarly, the middle term in \eqref{eq:short Strato correction} satisfies
\begin{subequations}
	\begin{align*}
	& \widehat{\rho}_s[(\partial_i \partial_j \phi)K_s^{il}] B^{jl} = - B^{jl} \int (\partial_j \phi) \partial_i \left( \widehat{\rho}_s K_s^{il}\right) \mathrm{d}x 
 = B^{jl} \widehat{\rho}_s \left[(\partial_j \phi) \left(h^l - \widehat{\pi}_is s[h^j]\right) \right] - B^{jl} \int (\partial_j \phi) \partial_i \left( \widehat{\rho}_s (BC)^{il}\right) \mathrm{d}x
	 \\
	 & =   B^{jl} \widehat{\rho}_s \left[(\partial_j \phi) \left(h^l - \widehat{\pi}_s[h^l]\right) \right] + \widehat{\rho}_s [\partial_i \partial_j \phi] (BCB^T)^{ij}
	 =    \widehat{\rho}_s \left[\nabla \phi \cdot B \left(h - \widehat{\pi}_s[h]\right) \right] + \widehat{\rho}_s [
	 \operatorname{Trace}(D^2\phi (BCB^T))].
	\end{align*}
\end{subequations}
We now collect terms and compare the $\mathrm{d}s$-contributions in \eqref{eq:Zakai Strato} and \eqref{eq:Zakai feedback}, arriving at
\begin{subequations}
	\label{eq:ds condition}
\begin{align}
\label{eq: ds conditions line 1}
&-\frac{1}{2}\left(\rho_s [\phi h\cdot C^{-1} h] + \rho_s[\operatorname{Trace}(D(\phi h)B) + \nabla \phi \cdot B h] \right)
\\
\label{eq: ds conditions line 2}
& = - \frac{1}{2}\widehat{\rho}_s[\phi] \pi_s[h \cdot C^{-1} h] - \frac{1}{2}  \widehat{\rho}_s \left[\nabla\phi \cdot B \left(h - \widehat{\pi}_s[h]\right) \right]
\\
\label{eq:partial phi K}
 & - \frac{1}{2} \widehat{\rho}_s \left[
\nabla \phi \cdot K_s C^{-1}\left( h + \widehat{\pi}_s[h]  \right) \right] + \widehat{\rho}_s [\nabla \phi \cdot \Xi_s] - \frac{1}{2} \widehat{\rho}_s[\phi]\pi_s[\operatorname{Trace}(DhB)].
\end{align}
\end{subequations}

Next, we work on the first term in \eqref{eq:partial phi K},
\begin{subequations}
\nonumber
	\begin{align*}
&\widehat{\rho}_s \left[
(\partial_i \phi K_s^{ij}) (C^{-1})^{jk} \left( h^k + \widehat{\pi}_s[h^k]  \right) \right] = - \int \phi \partial_i \left(\widehat{\rho}_s K_s^{ij}\right) (C^{-1})^{jk} \left( h^k + \widehat{\pi}_s[h^k]  \right) \mathrm{d}x
 -  \widehat{\rho}_s \left[
\phi K_s^{ij} (C^{-1})^{jk} \partial_i h^k    \right]
\\
& = \int \phi (h^j - \pi[h^j]) (C^{-1})^{jk} (h^k + \pi[h^k])\, \mathrm{d}\widehat{\rho}_s - \int \phi \partial_i \left(\widehat{\rho}_s (BC)^{ij}\right) (C^{-1})^{jk} \left( h^k + \widehat{\pi}_s[h^k]  \right) \mathrm{d}x  -  \widehat{\rho}_s \left[
\phi K_s^{ij} (C^{-1})^{jk} \partial_i h^k    \right]
\\
& = \widehat{\rho}_s [\phi (h^j - \pi[h^j])  (C^{-1})^{jk} ( h^k + \widehat{\pi}_s[h^k]  )]  + \widehat{\rho}_s [ \partial_i \phi (h^k + \widehat{\pi}_s[h^k]) ] B^{jk}
 + \widehat{\rho}_s[\phi B^{ik} \partial_i h^k] -  \widehat{\rho}_s \left[
\phi K_s^{ij} (C^{-1})^{jk} \partial_i h^k    \right]
\\
& = \widehat{\rho}_s [ \phi h^j h^k] (C^{-1})^{jk} - \widehat{\rho}_s [ \phi] \pi_s[ h^j]\pi_s [h^k] (C^{-1})^{jk}  + \widehat{\rho}_s [ \partial_i \phi (h^k + \widehat{\pi}_s[h^k]) ] B^{jk}
+ \widehat{\rho}_s[\phi B^{ik} \partial_i h^k] -  \widehat{\rho}_s \left[
\phi K_s^{ij} (C^{-1})^{jk} \partial_i h^k    \right]
\\
& = \widehat{\rho}_s [ \phi h C^{-1} h]
- \widehat{\rho}_s [ \phi] \pi_s[ h] C^{-1}\pi_s [h] + \widehat{\rho}_s [ \nabla \phi \cdot B (h + \widehat{\pi}_s[h]) ]
+ \widehat{\rho}_s[\phi \operatorname{Trace}(Dh B)] -  \widehat{\rho}_s \left[\phi \operatorname{Trace}(K_s C^{-1}Dh^T) \right]
\end{align*}
\end{subequations}
Plugging this into \eqref{eq:ds condition}  we see that after a great number of cancellations that \eqref{eq:ds condition} reduces to
\begin{equation*}
    \rho_s[\nabla \phi \cdot \Xi_s] = \frac{1}{2}\rho_s[\phi]\left(
    \pi[\operatorname{Trace}(D h \cdot B) - hC^{-1}h] 
    -\pi[h]C^{-1}\pi[h]
    \right)
    -\frac{1}{2} \rho_s[\phi \operatorname{Trace}(K_sC^{-1}D h].
\end{equation*}
Using \eqref{eq:Poisson K intro} in its weak formulation tested against $h$ on the first term in the right-hand side we obtain
\begin{equation}
    \label{eq: weak poisson gamma}
    \pi[\nabla \phi \cdot \Xi_s] = \frac{1}{2} \left( \pi[\phi]\pi[\operatorname{Trace}(KC^{-1}D h)] - \pi[\phi \operatorname{Trace}(KC^{-1}D h)])\right),
\end{equation}
which is a weak version of \eqref{eq:Poisson Gamma intro}.

We have thus obtained constraints on the coefficients of $K$ and $\Xi$ of the McKean-Vlasov equation \eqref{eq: true McKean Vlasov}
such that $\hat{\rho}$ solves the Zakai equation satisfied by the conditional density $\rho$ of the signal of the filtering problem. We can conclude because the solution to the Zakai equation is assumed to be unique.
\end{proof}


\section{On the connection between McKean-Vlasov filtering and maximum likelihood estimation}
\label{app:MLE}

Here we discuss the relationship between parameter estimation based on the ensemble Kalman filter and the maximum likelihood approach considered in \cite{diehl2016pathwise}. In a nutshell, the two methods essentially agree in the noiseless case $R=0$ with Gaussian initial condition, as relevant McKean-Vlasov dynamics can be solved explicitly in this case. This connection has already been pointed out in \cite{nusken2019state}.

We are given a $d$-dimensional process $Z$, which depends on an $m$-dimensional parameter $\theta$. We observe $Z$ through the observation $Y$, which is possibly noisy.
\begin{subequations}
    \label{eq: mle system}
    \begin{align}
         \mathrm{d}Z_t & =  g(Z_t)\theta \ \mathrm{d}t + \tilde{G}^{\frac{1}{2}} \mathrm{d}W_t,  \\
         \mathrm{d}\theta_t & = 0, \\
         \mathrm{d}Y_t & = \mathrm{d}Z_t + R^{\frac{1}{2}} \mathrm{d}V_t.
    \end{align}
\end{subequations}
Here $W$ and $V$ are independent $d$-dimensional Brownian Motions, $g: \mathbb{R}^d \to \mathbb{R}^{d\times m}$ is a given $C^2_b$ function and $\tilde{G}\in \R^{d\times d}$ is a given positive semidefinite deterministic matrix.

Let $D=d+m$. We set the $D$-dimensional signal process as $X := (Z,\theta)$ and replace the equation for $\mathrm{d}Z$ into the equation for the observation $\mathrm{d}Y$. System \eqref{eq: mle system} becomes equivalent to system \eqref{eq:signal obs} with the following drift coefficients
\begin{equation*}
    f(x) = f(z,\theta) = ( g(z) \theta , 0),
    \qquad
    h(x) = h(z,\theta) =   g(z) \theta.
\end{equation*}
Moreover, the coeffiecents of the noise are
\begin{equation*}
    G = \begin{pmatrix}
    \tilde{G} & 0\\ 0 & 0
    \end{pmatrix},
    \qquad
    U = \begin{pmatrix}
    \tilde{G}^{\frac{1}{2}} & 0
    \end{pmatrix},
    \qquad
    G = \tilde{G}.
\end{equation*}
Remember $\pi_t[\phi]:= \mathbb{E}[\phi(X_t) \mid \mathcal{Y}_t]$ is the conditional law of $X$ given $Y$.

First we consider noiseless observations, i.e. we set $R = 0$. In this case we have that
\begin{equation*}
    C := U U^{\top} + R = \tilde{G},
    \qquad
    B := G^{\frac{1}{2}}U^{\top}C^{-1}
    = \begin{pmatrix}
    \mathrm{Id}_{d\times d}\\ 0_{m\times m}
    \end{pmatrix}.
\end{equation*}
To convert system \eqref{eq: mle system} to the McKean-Vlasov equation we need to compute $P$ as defined in \eqref{eq: approx P}, which can be decomposed as follows
\begin{equation*}
    P(\pi) = \begin{pmatrix}
    \operatorname{Cov}_\pi(z,h)\\
    \operatorname{Cov}_\pi(\theta,h)
    \end{pmatrix} C^{-1} + B
    := \begin{pmatrix}
    \pi \left[ z (h-\pi[h])^{\top}\right]\\
    \pi \left[ \theta (h-\pi[h])^{\top}\right]
    \end{pmatrix} C^{-1} + B.
\end{equation*}
In this case with $R = 0$, we have $Y = Z$, which implies 
\begin{equation*}
    \pi_t(\mathrm{d}z,\mathrm{d}\theta) = \delta(z-Z_t)\pi^{\theta}(\mathrm{d}\theta),
    \qquad
    \pi_t[h] = \int_{\R^D} g(z) \theta \pi_t(\mathrm{d}z,\mathrm{d}\theta)
    =  g(Z_t) \int_{\R^m}\theta \pi_t^{\theta}(\mathrm{d}\theta)
    =: g(Z_t)\bar{\theta}_t
\end{equation*}
We have the following
\begin{equation*}
    \operatorname{Cov}_{\pi_t}(z,h) = \int_{\R^{d+m}} z\left( g(z) \theta  - g(z)\bar{\theta}_t \right)^{\top} \pi_t(\mathrm{d}z,\mathrm{d}\theta)
    = Z_t  \left(\int_{\R^m}(\theta -\bar{\theta})\pi^{\theta}_t(\mathrm{d}\theta)\right)^{\top} g(Z_t)^{\top} 
    = 0_{d\times d}.
\end{equation*}
\begin{equation*}
    \operatorname{Cov}_{\pi_t}(\theta,h) = \int_{\R^{m}} \theta \left(\theta - \bar{\theta}_t \right)^{\top} \pi^{\theta}_t(\mathrm{d}\theta)
    g(Z_t)^{\top} 
    = \operatorname{Var}(\theta_t) g(Z_t)^{\top}.
\end{equation*}
We have thus obtained $P$ explicitly
\begin{equation*}
    P(\pi_t) = \begin{pmatrix}
    \operatorname{Cov}_\pi(z,h)\\
    \operatorname{Cov}_\pi(\theta,h)
    \end{pmatrix} C^{-1} + B
    = \begin{pmatrix}
    0_{d\times d} \\
    \operatorname{Var}(\theta_t) g(Z_t)^{\top}
    \end{pmatrix} \tilde{G}^{-1}
    + \begin{pmatrix}
    \mathrm{Id}_{d\times d}\\ 0_{m\times m}
    \end{pmatrix}
    = \begin{pmatrix}
    \mathrm{Id}_{d\times d}\\
    \operatorname{Var}(\theta_t) g(Z_t)^{\top} \tilde{G}^{-1}
    \end{pmatrix}
\end{equation*}
We now compute $\widehat{\Gamma}$, which is defined in \eqref{eq: approx Gamma}. We start with the $z$ contribution, for $\gamma =1, \dots , d$, we have
\begin{equation*}
    \pi[z^{\gamma} (D h - \pi[D h])^{\top}]
    = Z^{\gamma}_t \pi[D h - \pi[D h]]^\top = 0_{D\times d}.
\end{equation*}
Now compute the contribution in $\theta$. For $\gamma = 1,\dots , m$ and $i=1,\dots, d$, we have
\begin{equation*}
     \pi_t\left[\bar{\theta}^{\gamma} D h^{\top}\right]_{j,i}
    =
    \sum_{k=1}^m \partial_{z^j} g^{i,k}(Z_t) \operatorname{Cov}(\theta^k_t,\theta^\gamma_t)1_{j\leq d}
    = (\partial_{z^j}g(Z_t) \operatorname{Var}(\theta_t))^{i,\gamma}1_{j\leq d}
    \qquad
    j = 1,\dots, d+m.
\end{equation*}
So that we have
\begin{equation*}
    P(\pi_t)^{\top}\pi_t\left[\theta^{\gamma} \left(D h - \pi_t[D h] \right)^{\top}\right]
    =
    \begin{pmatrix}
    \mathrm{Id}_{d\times d} &
    (\operatorname{Var}(\theta_t) g(Z_t)^{\top} \tilde{G}^{-1})^{\top}
    \end{pmatrix}
    \begin{pmatrix}
    Dg(Z_t)^{\top}\operatorname{Var}(\theta_t)^{\cdot\gamma}\\ 0_{m\times d}
    \end{pmatrix}
    = Dg(Z_t)^{\top}\operatorname{Var}(\theta_t)^{\cdot\gamma}
\end{equation*}
and 
\begin{equation*}
\Gamma^{\gamma}(\pi_t) = 
-\frac{1}{2} \operatorname{Trace}
\left(Dg(Z_t)\operatorname{Var}(\theta_t)^{\cdot\gamma}
\right).
\end{equation*}
If $m = 1$, we have
\begin{equation*}
    \Gamma(\pi_t) = \begin{pmatrix}
    0_{d}\\
    -\frac{1}{2} \operatorname{Var}(\theta_t) \operatorname{Trace}Dg(Z_t)
    \end{pmatrix}.
\end{equation*}
We can now write the McKean-Vlasov dynamics \eqref{eq: main system},
\begin{subequations}
    \label{eq: mckean in mle R = 0}
    \begin{align}
        \mathrm{d}Z_t & = \mathrm{d} \bY_t, 
        \\
        \mathrm{d}\widehat{\theta}_t & = \operatorname{Var}(\widehat{\theta}_t) g(Z_t)^{\top} \tilde{G}^{-1}\left[ \mathrm{d} \bY_t -  \tilde{G}^{\frac{1}{2}} \mathrm{d}\widehat{W}_t
        -
        g(Z_t)\widehat{\theta}_t\mathrm{d}t
        \right] - \frac{1}{2}\operatorname{Var}(\widehat{\theta}_t)\operatorname{Trace} Dg(Z_t) \mathrm{d}t.
        \label{eq: mckean in mle R = 0 second}
    \end{align}
\end{subequations}
We interpret equation \eqref{eq: mckean in mle R = 0} in the sense of Definition \ref{def: solution}.
Taking expectation in \eqref{eq: mckean in mle R = 0 second}, and writing this equation in integral form, we obtain the rough path maximum likelihood estimator from \cite{diehl2016pathwise}, up to the rescaling $\text{Var}(\widehat{\theta}_t)$, see also \cite{nusken2019state}.

We would like to obtain an equation for the variance of $\widehat{\theta}$. We start by computing the following equation for the centered random variable $\bar{\theta} := \widehat{\theta} - \pi[\theta]$,
\begin{equation*}
\mathrm{d}\bar{\theta}_t 
= -\operatorname{Var}(\widehat{\theta}_t) g(Z_t)^{\top} \tilde{G}^{-1}\left[ \tilde{G}^{\frac{1}{2}} \mathrm{d}\widehat{W}_t
        +
        g(Z_t)\bar{\theta}_t \mathrm{d}t
        \right].
\end{equation*}
By taking the square and using It\^o formula we obtain the following equation for $\operatorname{Var}(\widehat{\theta})$
\begin{equation*}
    \mathrm{d}\operatorname{Var}(\widehat{\theta}_t) = - \operatorname{Var}(\widehat{\theta}_t)^2 g(\bY_t)^{\top} \tilde{G}^{-1}g(\bY_t)\mathrm{d}t.
\end{equation*}
The solution is
\begin{equation*}
    \operatorname{Var}(\widehat{\theta}_t) = \left(\int_0^t g(\bY_s)^{\top} \tilde{G}^{-1}g(\bY_s) ds + \frac{1}{\operatorname{Var}(\widehat{\theta}_0)}\right)^{-1}.
\end{equation*}

Assume now that there exists $f: \R^d \to \R$ such that $g = \tilde{G}^{\top}\nabla f$. Since the variance is a path of bounded variation, we have that the integrand in the rough integral in equation \eqref{eq: mckean in mle R = 0 second} is controlled by $\bY$ in the following way
\begin{align*}
\operatorname{Var}(\widehat{\theta}_t) g(Y_t)^{\top} \tilde{G}^{-1} 
-
\operatorname{Var}(\widehat{\theta}_s) g(Y_s)^{\top} \tilde{G}^{-1} 
= &\operatorname{Var}(\widehat{\theta}_t) \nabla f(Y_t)^{\top}
-
\operatorname{Var}(\widehat{\theta}_s) \nabla f(Y_s)^{\top}\\
= & \operatorname{Var}(\widehat{\theta}_t) D^2f(Y_t) \delta Y_{s,t} + R_{s,t}.
\end{align*}
The Gubinelli derivative $\operatorname{Var}(\widehat{\theta}_t) D^2f(Z_t)$ is symmetric, which means that the rough integral in \eqref{eq: mckean in mle R = 0 second} does not depend on the area of the geometric rough path $\bY$. Hence, the expansion of the integral is completely determined by the path itself.

\section{Ito-Stratonovich correction for the Ensemble Kalman filter}
\label{app:Strato EnKF}

\textcolor{red}{
In Section \ref{sec:McKean formulation}, we derived the  system \eqref{eq: main system stochastic version} by approximately solving the PDEs \eqref{eq:Poisson K intro} and \eqref{eq:Poisson Gamma intro}. An alternative approach is to convert the Ensemble Kalman filter formulation 
\begin{equation}
\label{eq:Kalman Ito}
\left\{ \begin{array}{rl}
\mathrm{d}\widehat{X}_t & = f(\widehat{X}_t) \, \mathrm{d}t + G^{1/2} \, \mathrm{d}\widehat{W}_t + P(
\widehat{\pi}_t)\, \mathrm{d}I_t \\
\mathrm{d}I_t & = \mathrm{d} Y_t - \left( h(\widehat{X}_t)\, \mathrm{d}t + U \, \mathrm{d}\widehat{W}_t +  R^{1/2} \, \mathrm{d}\widehat{V}_t\right).
\end{array}
\right.
\end{equation}
from \cite{nusken2019state} to its Stratonovich form. Here we show that this leads to the same result.
\begin{lemma}
For the solution to \eqref{eq:Kalman Ito}, we have that
\begin{equation}
\int_0^t P(\widehat{\pi}_s) \,\mathrm{d}I_t = \int_0^t P(\widehat{\pi}_s) \circ \mathrm{d}I_t + \int_0^t \Gamma(\widehat{\pi}_s) \, \mathrm{d}s, 
\end{equation}
thus recovering \eqref{eq: main system stochastic version} from \eqref{eq:Kalman Ito}.
\end{lemma}
\begin{proof}
Direct calculation (using It{\^o}'s formula for the expression \eqref{eq: approx P} and applying the conditional expectation) shows that
\begin{equation}
\mathrm{d}P^{ij}(\widehat{\pi}_t) = \left( \widehat{\pi}_t \left[ x^i \nabla_l h^k P^{lm}(\widehat{\pi}_t) \right] - \widehat{\pi}_t [x^i] \widehat{\pi}_t[\nabla_l h^k P^{lm}(\widehat{\pi}_t) ]\right) (C^{-1})^{kj} \, \mathrm{d}Y^m + FV,  
\end{equation}
where $FV$ stands for a process of finite variation. The claim therefore follows from the conversion formula
\begin{equation}
    \int_0^t P(\widehat{\pi}_s) \,\mathrm{d}I_t = \int_0^t P(\widehat{\pi}_s) \circ \mathrm{d}I_t - \frac{1}{2} \left\langle P(\widehat{\pi}_{\cdot}), I_{\cdot} \right\rangle = \int_0^t P(\widehat{\pi}_s) \circ \mathrm{d}I_t - \frac{1}{2} \left\langle P(\widehat{\pi}_{\cdot}), Y_{\cdot} \right\rangle, 
\end{equation}
together with the fact that $\langle Y, Y \rangle_t = Ct$.
\end{proof}}

\bibliographystyle{abbrv}
\bibliography{refs_final}

\end{document}